\documentclass[11pt,twoside]{article}
\usepackage{mathrsfs}
\usepackage{amsthm}
\usepackage{amssymb}
\usepackage{amsmath}

\allowdisplaybreaks

\def\rr{{\mathbb R}}
\def\rn{{{\rr}^n}}
\def\rnz{{{\rr}^{n+1}_+}}

\def\zz{{\mathbb Z}}
\def\nn{{\mathbb N}}

\def\cc{{\mathbb C}}
\def\cn{{\mathbb N}}
\def\nn{{\mathcal N}}
\def\cs{{\mathcal S}}

\def\cm{{\mathcal M}}
\def\car{{\mathcal R}}

\def\ca{{\mathcal A}}

\def\fz{\infty}
\def\az{\alpha}
\def\supp{{\mathop\mathrm{\,supp\,}}}
\def\dist{{\mathop\mathrm {\,dist\,}}}
\def\loc{{\mathop\mathrm{\,loc\,}}}

\def\lz{\lambda}
\def\dz{\delta}

\def\ez{\epsilon}

\def\bz{\beta}
\def\ro{\rho}

\def\gz{{\gamma}}
\def\oz{{\omega}}

\def\vz{\varphi}

\def\sz{\sigma}
\def\pa{\partial}

\def\wz{\widetilde}
\def\nz{\nabla}

\def\hs{\hspace{0.3cm}}

\def\ls{\lesssim}

\def\bbmo{{\mathop\mathrm{BMO}}}
\def\bmo{{\mathop\mathrm{BMO}_{\ro,L}(\rn)}}
\def\bmoz{{\mathop\mathrm{BMO}_{\ro,L^\ast}(\rn)}}
\def\bmoo{{\mathop\mathrm{BMO}^M_{\ro,L}(\rn)}}
\def\bmoq{{\mathop\mathrm{BMO}^{q,M}_{\ro,L}(\rn)}}
\def\bmos{{\mathop\mathrm{BMO}^{q,M}_{\ro,L^\ast}(\rn)}}
\def\bmol{{\mathop\mathrm{BMO}^M_{\ro,L^\ast}(\rn)}}

\def\com{\complement}
\def\r{\right}
\def\lf{\left}
\def\lfr{\lfloor}
\def\rf{\rfloor}
\def\la{\langle}
\def\ra{\rangle}

\newtheorem{thm}{Theorem}[section]
\newtheorem{lem}{Lemma}[section]
\newtheorem{prop}{Proposition}[section]
\newtheorem{rem}{Remark}[section]
\newtheorem{cor}{Corollary}[section]
\newtheorem{defn}{Definition}[section]

\numberwithin{equation}{section}

\begin{document}
\arraycolsep=1pt
\title{{\vspace{-5cm}\small\hfill\bf J. Funct. Anal., to appear}\\
\vspace{4cm}\bf New Orlicz-Hardy spaces associated with divergence
form elliptic operators}
\author{Renjin Jiang\vspace{0.3cm}\footnote{Present
Adress: Department of Mathematics and Statistics,
University of Jyv\"{a}skyl\"{a}, P.O. Box 35 (MaD), 40014, Finland}
\, and Dachun Yang\footnote{Corresponding author.\endgraf
\hspace{0.25cm}{\it E-mail addresses:} rj-jiang@mail.bnu.edu.cn (R. Jiang),
dcyang@bnu.edu.cn (D. Yang)}\\
\footnotesize\it School of Mathematical Sciences, Beijing Normal University,
Laboratory of Mathematics\\
\footnotesize\it and Complex Systems, Ministry of Education,
Beijing 100875, People's Republic of China\\\\
\small Dedicated to Professor Lizhong Peng in celebration of his 66th birthday
}
\date{ }
\maketitle

\noindent$\line(1,0){425}$

\medskip

{\noindent{{\bf Abstract}

\bigskip

\small

Let $L$ be the divergence form
elliptic operator with complex
bounded measurable coefficients, $\omega$ the positive concave
function on $(0,\infty)$ of strictly critical lower type $p_\oz\in
(0, 1]$ and $\rho(t)={t^{-1}}/\omega^{-1}(t^{-1})$ for $t\in
(0,\infty).$ In this paper, the authors study the Orlicz-Hardy space
$H_{\omega,L}({\mathbb R}^n)$ and its dual space
$\mathrm{BMO}_{\rho,L^\ast}({\mathbb R}^n)$, where $L^\ast$ denotes
the adjoint operator of $L$ in $L^2({\mathbb R}^n)$. Several
characterizations of $H_{\omega,L}({\mathbb R}^n)$, including the
molecular characterization, the Lusin-area function characterization
and the maximal function characterization, are established. The
$\rho$-Carleson measure characterization and the John-Nirenberg
inequality for the space $\mathrm{BMO}_{\rho,L}({\mathbb R}^n)$ are
also given. As applications, the authors show that the Riesz
transform $\nabla L^{-1/2}$ and the Littlewood-Paley $g$-function
$g_L$ map $H_{\omega,L}({\mathbb R}^n)$ continuously into
$L(\omega)$. The authors further show that the Riesz transform
$\nabla L^{-1/2}$ maps $H_{\omega,L}({\mathbb R}^n)$ into the classical
Orlicz-Hardy space $H_{\omega}({\mathbb R}^n)$ for $p_\omega\in (\frac{n}{n+1},1]$ and
the corresponding fractional integral $L^{-\gamma}$ for certain
$\gamma>0$ maps $H_{\omega,L}({\mathbb R}^n)$ continuously into
$H_{\widetilde{\omega},L}({\mathbb R}^n)$, where $\widetilde{\omega}$
is determined by $\omega$ and $\gamma$, and satisfies the same property as $\omega$.
All these results are new even when $\omega(t)=t^p$ for all $t\in
(0,\infty)$ and $p\in (0,1)$.}

\bigskip

}

\smallskip

\noindent$\line(1,0){425}$

\vspace{0.15cm}

\section{Introduction\label{s1}}

\hskip\parindent Ever since Lebesgue's theory of
integration has taken a center stage in concrete problems
of analysis, the need for more inclusive classes of
function spaces than the $L^p(\rn)$-families naturally arose.
It is well known that the Hardy spaces $H^p(\rn)$ when $p\in (0,1]$
is a good substitute of $L^p(\rn)$ when studying the boundedness
of operators, for example, the Riesz operator is bounded on
$H^p(\rn)$, but not on $L^p(\rn)$ when $p\in (0,1]$.
The theory of Hardy spaces $H^p$ on the Euclidean space
$\rn$ was initially developed by Stein and Weiss \cite{sw}.
Later, Fefferman and Stein \cite{fs}
systematically developed a real-variable theory for the Hardy spaces
$H^p(\rn)$ with $p\in (0,1]$, which
now plays an important role in various fields of analysis and
partial differential equations; see, for example,
\cite{s93,clms,g2,m94, s94}. A key feature of the classical
Hardy spaces is their atomic decomposition characterizations,
which were obtained by Coifman \cite{co} when $n=1$
and Latter \cite{la} when $n>1$. On the other hand,
as another generalization of $L^p(\rn)$,
the Orlicz space was introduced by Birnbaum-Orlicz
in \cite{bo} and Orlicz in \cite{o32}, since then, the theory of the Orlicz
spaces themselves has been well developed and the
spaces have been widely used in probability, statistics,
potential theory, partial differential equations,
as well as harmonic analysis and some
other fields of analysis; see, for example,
\cite{rr91,rr00, byz08,mw, aikm, io}. Moreover,
the Orlicz-Hardy spaces are also good substitutes
of the Orlicz spaces in dealing with many
problems of analysis, say, the boundedness
of operators. In particular, Str\"omberg \cite{s79}
and Janson \cite{ja}
introduced generalized Hardy spaces $H_\oz(\rn)$,
via replacing the norm $\|\cdot\|_{L^p(\rn)}$ by the Orlicz-norm
$\|\cdot\|_{L(\oz)}$ in the definition
of $H^p(\rn)$, where $\oz$ is an Orlicz
function on $[0, \fz)$ satisfying some control conditions. Viviani \cite{v} further
characterized these spaces $H_\oz$ on spaces of homogeneous type via atoms.
The dual spaces of these spaces were also studied
in \cite{s79,ja,v,hsv}. All theories of these spaces are
intimately connected with properties of harmonic analysis and
of the Laplacian operator on $\rn$.

In recent years, function spaces, especially Hardy spaces and BMO spaces,
associated with different operators inspire great interests;
see, for example, \cite{adm,amr,ar,dxy,dy1,dy2,hm1,ya1,jyz,hlmmy}
and their references. In particular, Auscher, Duong and McIntosh \cite{adm}
first introduced the Hardy space $H^1_L(\rn)$
associated with an operator $L$ whose heat kernel satisfies a pointwise
Poisson type upper bound by means of a corresponding variant of the Lusin-area
function, and established its molecular characterization.
Duong and Yan \cite{dy1,dy2}
introduced its dual space $\mathrm{BMO\,}_L(\rn)$
and established the dual relation between $H_L^1(\rn)$
and $\mathrm{BMO\,}_L(\rn)$. Yan \cite{ya1} further
generalized these results to the Hardy spaces $H^p_L(\rn)$ with certain
$p\le 1$ and their dual spaces.
Also, Auscher and Russ \cite{ar} studied the Hardy space $H^1_L$ on
strongly Lipschitz domains associated with a divergence
form elliptic operator $L$ whose heat kernels have
the Gaussian upper bounds and regularity.
Very recently, Auscher, McIntosh and Russ \cite{amr}
treated the Hardy space $H^p$ with $p\in [1,\fz]$ associated to Hodge
Laplacian on a Riemannian manifold with
doubling measure, and Hofmann and Mayboroda \cite{hm1}
further studied the Hardy space $H_L^1(\rn)$ and its dual space adapted to a
second order divergence form elliptic operator $L$ on $\rn$ with bounded
complex coefficients and these operators may
not have the pointwise heat kernel bounds.

Motivated by \cite{hm1,ja,v}, in this paper, we study Orlicz-Hardy
spaces  $H_{\oz,L}(\rn)$ associated to the divergence form elliptic
operator $L$ in \cite{hm1} and their dual space
$\mathrm{BMO}_{\ro,L^\ast}(\rn)$, where $L^\ast$ denotes the adjoint
operator of $L$ in $L^2(\rn)$, the positive function
$\omega$ on $(0,\infty)$ is concave and of strictly critical lower type
$p_\oz\in (0, 1]$ and $\rho(t)={t^{-1}}/\omega^{-1}(t^{-1})$ for all $t\in (0,\infty).$
A typical example of such Orlicz functions is $\oz(t)=t^p$ for all
$t\in (0,\fz)$ and $p\in (0,1]$. As applications, we obtain the boundedness
of the Riesz transform, the Littlewood-Paley $g$-function
and the fractional integral associated with $L$ on $H_{\oz,L}(\rn)$,
which may not be bounded on the classical Orlicz-Hardy space $H_\oz(\rn)$
or the Orlicz space $L(\oz)$. Thus, it is necessary to introduce and
study the Orlicz-Hardy space $H_{\oz,L}(\rn)$.

Recall that the classical
$\mathop{\mathrm{BMO}}\,(\rn)$ was originally introduced and
studied by John and Nirenberg \cite{jn} in the context of partial
differential equations, which has been identified as the
dual space of $H^1(\rn)$ in the work by Fefferman and Stein
\cite{fs}. Also, the generalized
space $\mathop{\mathrm{BMO}}_\ro(\rn)$ was introduced and studied
in \cite{s79,ja,v,hsv} and it was proved therein
to be the dual space of $H_\oz(\rn)$.

To state the main content of this paper, we first recall
some notation and  known facts on
second order divergence form elliptic operators on $\rn$ with
bounded complex coefficients from \cite{a1, hm1}.
Let $A$ be an $n\times n$ matrix with entries
$\{a_{j,k}\}_{j,\,k=1}^n\subset L^\fz(\rn,\cc)$
satisfying the ellipticity conditions, namely,
there exist constants $0 <\lz_A\le \Lambda_A<\fz$ such
that for all $\xi,\,\zeta\in\cc^n$,
\begin{equation}\label{1.1}
\lz_A|\xi|^2\le\car e \la A\xi,\xi\ra \ \ \mbox{and} \ \ \ |\la A\xi,\zeta\ra|
\le \Lambda_A |\xi||\zeta|.
\end{equation}
Then the second order divergence form operator is given by
\begin{equation}\label{1.2}
 L f \equiv \mathop\mathrm{div}(A\nz f ),
 \end{equation}
interpreted in the weak sense via a sesquilinear form. Following \cite{hm1}, set
\begin{equation*}
  p_L\equiv \inf\lf\{p\ge 1:\,\sup_{t>0}\|e^{-tL}\|_{L^p(\rn)\to L^p(\rn)}<\fz\r\}
\end{equation*}
and
\begin{equation*}
 \wz p_L\equiv \sup\lf\{p\le \fz:\,\sup_{t>0}\|e^{-tL}\|_{L^p(\rn)\to L^p(\rn)}<\fz\r\}.
\end{equation*}
It was proved by Auscher \cite{a1} that
if $n=1,\,2$, then $p_L=1$ and $\wz p_L=\fz$,
and if $n\ge 3$, then $p_L<2n/(n+2)$ and $\wz p_L>2n/(n-2)$.
Moreover, thanks to a counterexample given by Frehse \cite{f},
this range is also sharp, which was pointed out to us by Professor
Pascal Auscher.

For all $f\in L^2(\rn)$ and $x\in\rn$, define
\begin{equation}\label{1.3}
  \cs_Lf(x)\equiv \lf(\iint_{\Gamma(x)}|t^2Le^{-t^2L}f(y)|^2
  \frac{\,dy\,dt}{t^{n+1}}\r)^{1/2}.
  \end{equation}
The space $H_{\oz,L}(\rn)$ is defined to be the completion of the set
$\{ f\in L^2(\rn):\, \cs_Lf\in L(\oz)\}$ with respect to the quasi-norm
\begin{equation*}
  \|f\|_{H_{\oz,L}(\rn)}\equiv \|\cs_Lf\|_{L(\oz)}=\inf\lf\{\lz>0:\,
  \int_{\rn}\oz\lf(\frac{\cs_L f(x)}{\lz}\r)\,dx\le 1\r\}.
\end{equation*}
If $p\le 1$ and $\oz(t)=t^p$ for all $t\in (0,\fz)$, we then denote
the Hardy space $H_{\oz,L}(\rn)$ by $H_L^p(\rn)$. The Hardy space
$H_L^1(\rn)$ was studied by Hofmann and Mayboroda in \cite{hm1}
(see also \cite{hm1c} for a corrected version).

In this paper, we first obtain the molecular decomposition
of the Orlicz-Hardy space $H_{\oz,L}(\rn)$. Using
this molecular decomposition, we then establish the
dual relation between the spaces $H_{\oz,L}(\rn)$ and
$\mathrm{BMO}_{\ro,L^\ast}(\rn)$, and the molecular characterization
of $H_{\oz,L}(\rn)$. Characterizations via the Lusin-area function
associated to the Poisson semigroup and the maximal functions are also
obtained. We also establish the $\ro$-Carleson measure characterization
and the John-Nirenberg inequality for the space $\bmo$.
As applications, we show that the Riesz transform $\nz
L^{-1/2}$ and the Littlewood-Paley $g$-function $g_L$ map
$H_{\oz,L}(\rn)$ continuously
into $L(\oz)$; in particular, $\nz L^{-1/2}$ maps $H_{\oz,L}(\rn)$ into
the classical Orlicz-Hardy space $H_\oz(\rn)$ for $p_\oz\in (\frac{n}{n+1},1]$.
Moreover, we show that the corresponding fractional integral $L^{-\gz}$ for
all $\gz\in (0, \frac n2(\frac 1{p_L}-\frac1 {\wz p_L}))$
maps $H_{\omega,L}({\mathbb R}^n)$ continuously into
$H_{\widetilde{\omega},L}({\mathbb R}^n)$, where $\widetilde{\omega}$
is determined by $\omega$ and $\gamma$, and satisfies the same property as $\omega$.
All these results are new even when $\oz(t)=t^p$
for all $t\in (0,\fz)$ and $p\in (0,1)$. When $p=1$ and $\oz(t)=t$
for all $t\in (0,\fz)$, some of results are also new.

The key step of the above approach is to establish a molecular
characterization of the Orlicz-Hardy space $H_{\oz,L}(\rn)$.
To this end, a main difficulty encountered
is the convergence problem of the summation of molecules, i.\,e.,
in what sense does the molecular characterization hold? In Theorem
\ref{t5.1} below, we prove that our molecular characterization
holds in the dual of $\bmoz$. This is quite different from the cases
for the Hardy space $H^1_L(\rn)$ in \cite{hm1} and the Hardy
space $H^1(\Lambda T^\ast M)$ in \cite{amr}, which only need
that the molecular characterizations hold pointwise; see \cite[(1.11)]{hm1}
(or its corrected version in \cite{hm1c}) and \cite[Definition 6.1]{amr}.
Recall that $M$ denotes a complete Riemannian manifold and
$$\Lambda T^\ast M\equiv\bigoplus_{0\le k\le\dim M}\Lambda ^kT^\ast M$$
the bundle over $M$ whose
fibre at each $x\in M$ is given by $\Lambda T^\ast_x M$,
the complex exterior algebra over the cotangent space
$T^\ast_x M$; see \cite[p.\,194]{amr}. In this paper,
to obtain the molecular characterization of $H_{\oz,L}(\rn)$,
we first need to show that the dual space of $H_{\oz,L}(\rn)$ is
$\bmoz$ in Theorem \ref{t4.1} below. The key ingredients used
in the proof of Theorem \ref{t4.1} is the Calder\'on
reproducing formula (Lemma \ref{l4.3} below) and the atomic
decomposition of the tent space $T_{\oz}(\rnz)$ (Theorem \ref{t3.1}
below). We point out that the dual space of $H^1_L(\rn)$
was already obtained in \cite[Theorems 8.2 and 8.6]{hm1}
by a different, but more complicated, approach, without invoking the atomic
decomposition of the tent space. Also, the dual space
of $H^1(\Lambda T^\ast M)$ was obtained in \cite{amr} as a direct
corollary of the dual theorem on the corresponding tent space;
see \cite[Theorem 5.8]{amr}.

Another key tool used in this paper to obtain the maximal
function characterizations of $H_{\oz,L}(\rn)$ and their applications
in boundedness of operators is Lemma \ref{l5.1} below,
which gives a sufficient condition
for the boundedness of linear or non-negative sublinear
operators from $H_{\oz,L}(\rn)$ to $L(\oz)$. Such
a condition for the molecular Hardy space
in $H^1_L(\rn)$ case was also given in
\cite[Lemma 3.3]{hm1}, which is a direct corollary
of the definition of the molecular Hardy space; see its corrected version in
\cite{hm1c}. To obtain Lemma \ref{l5.1},
we need the following important observation that for all $f\in H_{\oz,L}(\rn)\cap
L^2(\rn)$, since $t^2Le^{-t^2L}f\in T_2^2(\rnz)\cap T_{\oz}(\rnz)$,
by Proposition \ref{p3.1} below, the atomic decomposition of $t^2Le^{-t^2L}f$ holds in
both $T_{\oz}(\rnz)$ and $T_2^p(\rnz)$ for all $p\in [1,2]$.
Then by the fact that the operator $\pi_{L,M}$, which is introduced in \cite{dy2}
and initially defined on $F\in L^2(\rnz)$ with compact support by
\begin{equation}\label{1.4}
  \pi_{L,M}F\equiv C_M\int_0^\fz (t^2L)^{M+1}e^{-t^2L}F(\cdot,t)\frac{\,dt}{t},
\end{equation}
is bounded from $T_2^p(\rnz)$ to $L^p(\rn)$ for $p\in (p_L,\wz p_L)$
(see Proposition \ref{p4.1} below),
we further obtain the $L^p(\rn)$-convergence with $p\in (p_L,2]$
and the $H_{\oz, L}(\rn)$-convergence of the corresponding molecular
decomposition for functions in $H_{\oz,L}(\rn)\cap L^2(\rn)$
in Proposition \ref{p4.2} below. These convergences are
necessary and play a fundamental
role in the whole paper, which is totally different from
the $\dz$-representation used in \cite{hm1, hm1c, hlmmy}. Here and in what follows,
$M\in\cn$ and
$$C_M\int_0^\fz t^{2(M+2)}e^{-2t^2}\frac{\,dt}{t}=1.$$
We remark that the convergence of the atomic decomposition
of the tent spaces was also already carefully dealt with in
\cite{amr} (We thank Professor Pascal Auscher to point out
this to us). To be precise, in \cite[pp.\,209-210]{amr},
Auscher, McIntosh and Russ proved that for any functions
$F$ in the intersection of the tent spaces
$T^{1,2}(\Lambda T^\ast M)$ and $T^{2,2}(\Lambda T^\ast M)$
with the support $M\times[\ez,\fz)$ for some $\ez>0$,
$F_n\equiv F\chi_{B(x_0,n)\times(1/n,n)}$ for any $x_0\in M$
has an atomic decomposition which converges in both $T^{1,2}(\Lambda T^\ast M)$
and $T^{2,2}(\Lambda T^\ast M)$;
see \cite[(4.5)]{amr}. Observe that the compact support
of $F_n$ plays an important role in establishing
the convergence of its atomic decomposition in \cite{amr}. However,
Proposition \ref{p3.1} below are true for
all functions in $T_{\oz}(\rnz)\cap T_2^p(\rnz)$
without assuming the compact supports.
To obtain this proposition, we need to subtly use the construction
of the supports of atoms in the atomic decomposition
of tent spaces $T_{\oz}(\rnz)$ in Theorem \ref{t3.1} below
and the Lebesgue dominated convergence theorem.

This paper is organized as follows.

In Section 2, we recall some notions and known results concerning
operators associated with $L$ and describe some
basic assumptions on the Orlicz function $\oz$ considered in this paper.
We point out that throughout the whole paper, we always assume that
$\omega$ on $(0,\infty)$ is concave and of strictly critical lower type
$p_\oz\in (0, 1]$. These restrictions are necessary for the Orlicz-Hardy space
$H_{\oz, L}(\rn)$ to have the molecular characterization;
see Theorem 5.1 below. Thus, under these restrictions,
the Orlicz-Hardy space $H_{\oz, L}(\rn)$ behaves more closely like
the classical Hardy space. We leave the study on the Orlicz-Hardy
space with a Young function in a forthcoming paper, which may have some
properties similar to those of the spaces $H^p(\Lambda T^\ast M)$
with $p\in (1,\fz]$ as in \cite{amr}.

In Section 3, we introduce the tent spaces $T_\oz(\rnz)$ associated to $\oz$
and establish its atomic characterization; see Theorem \ref{t3.1} below.
By the proof of Theorem \ref{t3.1}, we observe that if a function
$F\in T_\oz(\rnz)\cap T_2^p(\rnz)$, $p\in (0,\fz)$, then there exists
an atomic decomposition of $F$ which converges in both
$T_\oz(\rnz)$ and $T_2^p(\rnz)$; see Proposition
\ref{p3.1} below. As a consequence, we prove that if $F\in T_\oz(\rnz)\cap T_2^2(\rnz)$,
then there exists an atomic decomposition of $F$ which converges in
both $T_\oz(\rnz)$ and $T_2^p(\rnz)$ for all $p\in [1,2]$;
see Corollary \ref{c3.1} below.

In Section 4, we first introduce the Orlicz-Hardy space
$H_{\oz,\,L}(\rn)$, and then prove that the operator $\pi_{L,M}$
in \eqref{1.4} maps the tent space $T_2^p(\rnz)$ continuously into $L^p(\rn)$ for
$p\in (p_L,\wz p_L)$ and $T_\oz(\rnz)$ continuously into $H_{\oz,L}(\rn)$
(see Proposition \ref{p4.1} below). Combined this with Corollary \ref{c3.1},
we obtain a molecular decomposition for elements in $H_{\oz,L}(\rn)\cap L^2(\rn)$
which converges in $L^p(\rn)$ for $p\in (p_L,2]$; see Proposition \ref{p4.2} below.
Via this molecular decomposition of $H_{\oz,\,L}(\rn)$, we further obtain
the duality between $H_{\oz,\,L}(\rn)$ and $\bbmo_{\rho,\,L^\ast}(\rn)$ (see Theorem
\ref{t4.1} below). We also remark that the proof of Theorem \ref{t4.1} is
much simpler than the proof of \cite[Theorem 8.2]{hm1}.

In Section 5, we introduce the molecular Hardy space, where the
summation of molecules converges in the space
$(\mathrm{BMO}_{\ro,L^\ast}(\rn))^\ast$, the dual space of
$\mathrm{BMO}_{\ro,L^\ast}(\rn)$. Then we show that the molecular
Hardy space is equivalent to the Orlicz-Hardy space $H_{\oz,L}(\rn)$
with equivalent norms; see Theorem \ref{t5.1} below. Furthermore, we
characterize  $H_{\oz,L}(\rn)$ via the Lusin-area function associated to
the Poisson semigroup, and the maximal functions; see Theorem \ref{t5.2}
below. We also point out that a sufficient condition
for the boundedness of linear or non-negative sublinear
operators from $H_{\oz,L}(\rn)$ to
$L(\oz)$ is also given in Lemma \ref{l5.1} below, which plays a key role in
the proof of Theorem \ref{t5.2} and is very useful in applications
(see Section \ref{s7} of this paper). This condition is also necessary
if $\oz(t)=t^p$ for all $t\in (0,\fz)$ and $p\in (0, 1]$.

Section 6 is devoted to establish the $\ro$-Carleson measure characterization
(see Theorem \ref{t6.1} below) and the John-Nirenberg inequality
(see Theorem \ref{t6.2} below) for the space $\bmo$.

In Section 7, as an application, we give some
sufficient conditions which guarantee the boundedness of linear
or non-negative sublinear operators from $H_{\oz,L}(\rn)$
to $L(\oz)$; in particular, we show that the Riesz transform $\nz
L^{-1/2}$ and the Littlewood-Paley $g$-function $g_L$
map $H_{\oz,L}(\rn)$ continuously into $L(\oz)$; see Theorem \ref{t7.1} below.
A fractional variant of Theorem \ref{t7.1} is
also given in this section; see Theorem \ref{t7.2} below.
Using Theorem \ref{t7.2}, we prove that the fractional integral $L^{-\gz}$ for
all $\gz\in (0, \frac n2(\frac 1{p_L}-\frac1 {\wz p_L}))$ maps
$H_{\omega,L}({\mathbb R}^n)$ continuously into
$H_{\widetilde{\omega},L}({\mathbb R}^n)$, where $\widetilde{\omega}$
is determined by $\omega$ and $\gamma$ and satisfies the
same property as $\omega$; see Theorem \ref{t7.3} below.
In particular, $L^{-\gz}$ maps $H_L^p(\rn)$ continuously into
$H_L^q(\rn)$ for $0<p<q\le 1$ and $n/p-n/q=2\gz$; see Remark \ref{r7.3} below.
Applying Theorems \ref{t7.1} and \ref{t7.3},
we further show that $\nz L^{-1/2}$ maps $H_{\oz,L}(\rn)$
continuously into $H_\oz(\rn)$ for $p_\oz\in (\frac{n}{n+1},1]$,
and in particular, $H_L^p(\rn)$ into the classical Hardy space
$H^p(\rn)$ for $p\in (\frac{n}{n+1},1]$; see Theorem \ref{t7.4}
below. Moreover, we show that
$H_{\oz,L}(\rn)\subset H_\oz(\rn)$ for all $p_\oz\in (\frac{n}{n+1},1]$
in Remark \ref{r7.4} below. It was also pointed out by Hofmann
and Mayboroda in \cite{hm1}
that $H_L^1(\rn)$ is a proper subspace of $H^1(\rn)$ for certain $L$
as in \eqref{1.2}. We remark that if $L=-\Delta+V$ with $V\in L^1_\loc(\rn)$
is the Schr\"odinger operator on $\rn$,
then it was proved in \cite{hlmmy} that
$\nz L^{-1/2}$ maps $H_L^1(\rn)$ into the classical Hardy space
$H^1(\rn)$.

We point out that this paper is strongly motivated by Hofmann and Mayboroda
\cite{hm1}, and we also directly use some estimates from \cite{hm1}
which simplify the proofs of some theorems of this paper.

Finally, we make some conventions. Throughout the whole paper,
$L$ always denotes the second order divergence form operator
as in \eqref{1.2}. We denote by $C$ a positive constant which
is independent of the main parameters, but it may vary from line
to line. The symbol $X \ls Y$ means that there exists a positive
constant $C$ such that $X \le CY$; the symbol
$\lfr\,\az\,\rf$ for $\az\in\rr$ denotes the maximal
integer no more than $\az$; $B(z_B,\,r_B)$ denotes an open ball with
center $z_B$ and radius $r_B$ and $CB(z_B,\,r_B)\equiv
B(z_B,\,Cr_B).$ Set $\cn\equiv\{1,2,\cdots\}$ and
$\zz_+\equiv\cn\cup\{0\}.$ For any subset $E$ of $\rn$, we denote by
$E^\com$ the set $\rn\setminus E.$

\section{Preliminaries\label{s2}}

\hskip\parindent  In this section, we recall some notions and notation on the
divergence form elliptic operator, and present some basic properties on
Orlicz functions and also describe some basic assumptions on them.

\subsection{Some notions on the divergence form elliptic operator $L$}

\hskip\parindent  In this subsection, we present some known facts
about the operator $L$ considered in this paper.

A family $\{S_t\}_{t>0}$ of operators is said to satisfy the $L^2$
off-diagonal estimates, which is also called the Gaffney estimates
(see \cite{hm1}), if there exist positive constants $c,\,C$ and $\bz$
such that for arbitrary closed sets $E,\,F\subset \rn$,
\begin{equation*}
   \|S_tf\|_{L^2(F)}\le Ce^{-(\frac{\dist(E,F)^2}{ct})^\bz}\|f\|_{L^2(E)}
\end{equation*}
for every $t>0$ and every $f\in L^2(\rn)$ supported in $E$. Here and in what
follows, for any $p\in (0,\fz]$ and $E\subset \rn$,
$\|f\|_{L^p(E)}\equiv \|f\chi_E\|_{L^p(\rn)}$; for any sets $E,\,F\subset \rn$,
$\dist(E,F)\equiv \inf\{|x-y|:\,x\in E,\,y\in F\}.$

The following results were obtained in \cite{a1,ahlmt,hm1,hm2}.

\begin{lem}[\cite{hm2}] \label{l2.1} If two families of operators, $\{S_t\}_{t>0}$
and $\{T_t\}_{t>0}$,
satisfy Gaffney estimates, then so does $\{S_tT_t\}_{t>0}$.
Moreover, there exist positive constants $c,\,C,$ and $\bz$
 such that for arbitrary closed
sets $E,\,F \subset \rn$,
\begin{equation*}
\|S_sT_tf\|_{L^2(F)}\le Ce^{-(\frac{\dist(E,F)^2}
{c\max\{s,t\}})^\bz}\|f\|_{L^2(E)}
\end{equation*}
for every $s,\,t>0$ and every $f\in L^2(\rn)$ supported in $E$.
\end{lem}

\begin{lem}[\cite{ahlmt,hm2}] \label{l2.2} The families,
\begin{equation}\label{2.1}
 \{e^{-tL}\}_{t>0}, \ \ \{tLe^{-tL}\}_{t>0}, \ \ \{t^{1/2}\nz e^{-tL}\}_{t>0},
\end{equation}
as well as
\begin{equation}\label{2.2}
 \{(I+tL)^{-1}\}_{t>0}, \ \ \{t^{1/2}\nz(I+tL)^{-1}\}_{t>0},
\end{equation}
are bounded on $L^2(\rn)$ uniformly in $t$
and satisfy the Gaffney estimates with positive constants $c,\,C$
 depending on $n,\,\lz_A,\,\Lambda_A$ as in \eqref{1.1} only.
  For the operators in \eqref{2.1}, $\bz=1$, while in \eqref{2.2}, $\bz=1/2$.
\end{lem}

\begin{lem}[\cite{a1,hm1}] \label{l2.3} There exist $p_L\in [1,\frac{2n}{n+2})$,
 $\wz p_L\in (\frac{2n}{n-2},\fz]$ and $c,C\in (0,\fz)$
 such that

 ${\rm (i)}$ for every $p$ and $q$ with $p_L < p\le q < \wz p_L$,
 the families $\{e^{-tL}\}_{t>0}$ and $\{tLe^{-tL}\}_{t>0}$ satisfy
 $L^p-L^q$ off-diagonal estimates, i.\,e.,
for arbitrary closed sets $E,\,F\subset \rn$,
\begin{equation*}
  \|e^{-tL}f\|_{L^q(F)}+\|tLe^{-tL}f\|_{L^q(F)}\le
  Ct^{\frac n2(\frac{1}{q}-\frac{1}{p})}e^{-\frac{\dist(E,F)^2}{ct}}
  \|f\|_{L^p(E)}
\end{equation*}
for every $t > 0$ and every $f\in  L^p(\rn)$ supported in E.
The operators $\{e^{-tL}\}_{t>0}$ and $\{tLe^{-tL}\}_{t>0}$ are bounded
from $L^p(\rn)$ to $L^q(\rn)$ with the norm
$Ct^{\frac n2(\frac{1}{q}-\frac{1}{p})}$;

${\rm (ii)}$ for every $p\in (p_L,\wz p_L)$,
the family $\{(I+tL)^{-1}\}_{t>0}$  satisfies
$L^p-L^p$ off-diagonal estimates, i.\,e.,
for arbitrary closed sets $E,\,F\subset \rn$,
\begin{equation*}
  \|(I+tL)^{-1}f\|_{L^q(F)}\le Ct^{\frac n2(\frac{1}{q}-\frac{1}{p})}
  e^{-\frac{\dist(E,F)}{ct^{1/2}}}
  \|f\|_{L^p(E)}
\end{equation*}
for every $t > 0$ and every $f\in  L^p(\rn)$ supported in E.
\end{lem}

\begin{lem}[\cite{hm1}] \label{l2.4} Let $k\in\cn$ and $p\in (p_L,\wz p_L)$. Then the
operator given by for any $f\in L^p(\rn)$ and $x\in\rn$,
\begin{equation*}
\cs_L^kf(x)\equiv \lf(\iint_{\Gamma(x)}|(t^2L)^ke^{-t^2L}f(y)|^2
\frac{\,dy\,dt}{t^{n+1}}\r)^{1/2},
\end{equation*}
is bounded on $L^p(\rn)$.
\end{lem}

\subsection{Orlicz functions \label{s2.2}}

\hskip\parindent Let $\omega$ be a positive function defined on
$\rr_+\equiv(0,\,\fz).$ The function $\omega$ is said to be of upper
type $p$ (resp. lower type $p$) for certain $p\in[0,\,\fz)$,  if there
exists a positive constant $C$ such that for all $t\geq 1$ (resp.
$t\in (0, 1]$) and $s\in (0,\fz)$,
\begin{equation}\label{2.3}
\omega(st)\le Ct^p \omega(s).
\end{equation}

Obviously, if $\oz$ is of lower type $p$ for certain $p>0$, then
$\lim_{t\to0+}\oz(t)=0.$ So for the sake of convenience, if it is
necessary, we may assume that $\oz(0)=0.$ If $\oz$ is of both upper
type $p_1$ and lower type $p_0$, then $\oz$ is said to be of type
$(p_0,\,p_1).$ Let
\begin{equation*}
p_\oz^+\equiv\inf\{ p>0:\ \mathrm{there\ exists} \ C>0 \ \mathrm{such \ that }
\ \eqref{2.3} \ \mathrm{holds\ for\ all}\ t\in[1,\fz),\ s\in (0,\fz)\},
\end{equation*}
and
\begin{equation*}
p_\oz^-\equiv\sup\{ p>0:\ \mathrm{there\ exists} \ C>0 \ \mathrm{such \ that }
\ \eqref{2.3} \ \mathrm{holds\ for\ all}\
t\in(0,1],\ s\in (0,\fz)\}.
\end{equation*}
The function $\oz$ is said to be of strictly lower type $p$ if for all $t\in(0,1)$
and $s\in (0,\fz)$, $\omega(st)\le t^p \omega(s),$ and define
\begin{equation*}
p_\oz\equiv\sup\{ p>0: \omega(st)\le t^p \omega(s) \ \mathrm{holds\ for\
all}\ s\in (0,\fz)\ \mathrm{and}\ t\in(0,1)\}.
\end{equation*}
It is easy to see that $p_\oz\le p_\oz^{-}\le{p_\oz^+}$ for all $\oz.$
In what follows, $p_\oz$, $p_\oz^-$ and ${p_\oz^+}$ are called to be the
strictly critical lower type index, the critical lower
type index and the critical upper type index of $\oz$, respectively.

\begin{rem}\label{r2.1}\rm
We claim that if $p_\oz$ is defined as above,
then $\oz$ is also of strictly lower type $p_\oz$.
In other words, $p_\oz$ is attainable.
In fact, if this is not the case, then there exist
certain $s\in (0,\fz)$ and $t\in (0,1)$ such that
$\oz(st)>t^{p_\oz}\oz(s)$. Hence there exists $\ez\in(0,p_\oz)$
small enough such that $\oz(st)>t^{p_\oz-\ez }\oz(s)$, which is
contrary to the definition of $p_\oz$. Thus, $\oz$ is of strictly
lower type $p_\oz$.
\end{rem}

Throughout the whole paper, we always assume that $\oz$ satisfies
the following assumption.

\begin{proof}[\bf Assumption (A)] Let $p_\oz$ be defined as above.
Suppose that $\oz$ is a positive Orlicz function
on $\rr_+$ with $p_\oz\in (0,1]$, which is continuous, strictly
increasing and concave.
\end{proof}

Notice that if $\oz$ satisfies Assumption (A), then $\oz(0)=0$
and $\oz$ is obviously of upper type 1. Since $\oz$ is concave,
it is subadditive. In fact, let $0<s<t$, then
$$\oz(s+t)\le \frac{s+t}{t}\oz(t)\le \oz(t)
+\frac{s}{t}\frac{t}{s}\oz(s)=\oz(s)+\oz(t).$$
For any concave function $\oz$ of strictly lower type $p$, if we set
$\wz\oz(t)\equiv\int_0^t\oz(s)/s\,ds$ for $t\in [0,\fz)$, then by
\cite[Proposition 3.1]{v}, $\wz\oz$ is equivalent to $\oz$, namely,
there exists a positive constant $C$ such that $C^{-1}\oz(t)\le
\wz\oz(t)\le C\oz(t)$ for all $t\in [0,\fz)$; moreover, $\wz\oz$
is strictly increasing, concave and continuous function of
strictly lower type $p.$ Since all our results are invariant on equivalent
functions, we always assume that $\oz$ satisfies Assumption (A);
otherwise, we may replace $\oz$ by $\wz\oz.$

\newtheorem{conveni}{}
\renewcommand\theconveni{}

\begin{proof}[\bf Convention (C)] From Assumption (A), it follows that
$0<p_\oz\le p_\oz^{-}\le {p_\oz^+}\le 1.$ In what follows,
if \eqref{2.3} holds for ${p_\oz^+}$ with $t\in [1,\fz)$,
then we choose ${\wz p_\oz}\equiv{p_\oz^+}$; otherwise  $p_\oz^+<1$ and
we choose ${\wz p_\oz}\in (p_\oz^+,1)$ to be close enough to
$p_\oz^+$.
\end{proof}

For example, if $\oz(t)=t^p$ with $p\in (0, 1]$, then $p_\oz=p_\oz^+=\wz p_\oz=p$;
if $\oz(t)=t^{1/2}\ln(e^4+t)$, then $p_\oz=p_\oz^+=1/2$, but $1/2<\wz p_\oz<1$.

Let $\oz$ satisfy Assumption (A). A measurable function $f$ on
$\rn$ is said to be in the Lebesgue type space $L(\oz)$ if
$$\int_{\rn}\oz(|f(x)|)\,dx< \fz.$$
Moreover, for any $f\in L(\oz)$, define
$$\|f\|_{L(\oz)}\equiv\inf\lf\{\lz>0:\ \int_{\rn}\oz\lf(\frac{|f(x)|}
{\lz}\r)\,dx\le 1\r\}.$$

Since $\oz$ is strictly increasing, we define the function $\ro(t)$
on $\rr_+$ by setting, for all $t\in (0,\fz)$,
\begin{equation}\label{2.4}
\ro(t)\equiv\frac{t^{-1}}{\oz^{-1}(t^{-1})},
\end{equation}
where and in what follows, $\oz^{-1}$ denotes the inverse function of
$\oz.$ Then the types of $\oz$ and $\rho$ have the following
relation; see \cite{v} for its proof.

\begin{prop}\label{p2.1}
Let $0<p_0\le  p_1\le1$ and $\oz$ be an increasing function. Then $\oz$
is of type $(p_0,\, p_1)$ if and only if $\ro$ is of type
$(p_1^{-1}-1,\,p_0^{-1}-1).$
\end{prop}

\section{Tent spaces associated to Orlicz functions\label{s3}}
 \hskip\parindent In this section, we study the tent spaces associated
to Orlicz functions. We first recall some notions.

For any $\nu>0$ and  $x\in\rn$, let $\rr^{n+1}_+\equiv \rn\times (0,\fz)$ and
$$\Gamma_\nu(x)\equiv\{(y,t)\in\rr^{n+1}_+:\,|x-y|<\nu t\}$$ denoting the  cone of
aperture $\nu$ with vertex $x\in\rn$. For any closed set $F$ of $\rn$,
denote by $\car_\nu{F}$ the union of all cones with vertices in $F$,
i.\,e., $\car_\nu{F}\equiv\cup_{x\in F}\Gamma_\nu(x)$; and for any open set
$O$ in $\rn$, denote the tent over $O$ by $T_\nu(O)$, which is
defined as $T_\nu(O)\equiv[\car_\nu(O^\com)]^\com.$ Notice that
$$T_\nu(O)=\{(x,t)\in\rn\times(0,\fz):\,\dist(x,O^\com)\ge \nu t\}.$$
In what follows, we denote $\Gamma_1(x)$, $\car_1(F)$ and $T_1(O)$
simply by $\Gamma(x)$, $\car(F)$ and $\widehat O$, respectively.

Let $F$ be a closed subset of $\rn$ and $O\equiv F^\com$. Assume that
$|O|<\fz$. For any fixed $\gz\in(0,1)$, we say that $x\in\rn$ has
the global $\gz$-density with respect to $F$ if
\begin{equation*}
\frac{|B(x,r)\cap F|}{|B(x,r)|}\ge \gz
\end{equation*}
for all $r>0$. Denote by $F^\ast$ the set of all such $x$. Obviously, $F^\ast$
is a closed subset of $F$. Let $O^\ast\equiv (F^\ast)^\com$. Then it is easy to
see that $O\subset O^\ast$. In fact, we have
$$O^\ast = \{x\in\rn:\, \cm(\chi_O)(x) > 1-\gz\} ,$$
where $\cm$ denotes the Hardy-Littlewood maximal function on $\rn$.
As a consequence, by the weak type $(1,1)$ of $\cm$, we have
$|O^\ast| \le  C(\gz)|O|,$ where and in what follows, $C(\gz)$ denotes a positive
constant depending on $\gz$.

 The proof of the following lemma is
similar to that of \cite[Lemma 2]{cms}; we omit the details.

\begin{lem}\label{l3.1}
Let $\nu,\ \eta\in(0,\fz)$. Then there exist positive constants
 $\gz\in(0,1)$ and $C(\gz,\nu,\eta)$
 such that for any closed subset $F$ of $\rn$
whose complement has finite measure and any non-negative measurable
function $H$ on $\rr^{n+1}_+$,
$$\iint_{\car_{\nu}(F^\ast)} H(y,t)t^n\,dy\,dt\le
C(\gz,\nu, \eta)\int_F\lf\{\iint_{\Gamma_{\eta}
(x)}H(y,t)\,dy\,dt\r\}\,dx,$$
where $F^\ast$ denotes the set of
points in $\rn$ with global $\gz$-density with respect to $F$.

\end{lem}

Let $\nu\in (0,\fz)$. For all measurable functions $g$ on
${\rr}^{n+1}_+$ and all $x\in\rn$, let
\begin{equation*}
\ca_\nu(g)(x)\equiv
\lf(\iint_{\Gamma_\nu(x)}|g(y,t)|^2\frac{\,dy\,dt}{t^{n+1}}\r)^{1/2},
\end{equation*}
and denote $\ca_1(g)$ simply by $\ca(g)$.

Coifman, Meyer and Stein \cite{cms} introduced the tent space
$T_2^p(\rnz)$ for $p\in(0,\fz)$, which is defined as the space of all
measurable functions $g$ such that
$\|g\|_{T_2^p(\rnz)}\equiv\|\ca(g)\|_{L^p(\rn)}<\fz.$

On the other hand, let $\oz$ satisfy Assumption (A).
Harboure, Salinas and Viviani \cite{hsv} defined the
tent space $T_\oz(\rnz)$ associated to the function $\oz$
as the space of measurable functions $g$ on $\rr^{n+1}_+$ such that
$\ca(g)\in L(\oz)$ with the norm defined by
$$\|g\|_{T_\oz(\rnz)}\equiv \|\ca(g)\|_{L(\oz)}=\inf\lf\{\lz>0:\ \int_{\rn} \oz
\lf(\frac{\ca(g)(x)}{\lz}\r) \,dx\le 1\r\}.$$

\begin{lem}\label{l3.2}
Let $\eta,\,\nu\in(0,\fz)$. Then there exists a positive
constant $C$, depending on $\eta$ and $\nu$, such that
for all measurable functions $H$ on $\rr^{n+1}_+$,
\begin{equation}\label{3.1}
C^{-1}\int_{\rn} \oz(\ca_\eta (H)(x))\,dx\le
\int_{\rn} \oz(\ca_\nu (H)(x))\,dx\le
C\int_{\rn} \oz(\ca_\eta (H)(x))\,dx.
\end{equation}
\end{lem}
\begin{proof}[\bf Proof.]
By the symmetry, we only need to establish the first inequality
in \eqref{3.1}. To this end, let
$\lz\in(0,\fz)$ and $O_\lz\equiv \{x\in\rn:\, \ca_\nu(H)(x)>\lz\}.$
If $|O_\lz|=\fz$, then $\int_{\rn} \oz(\ca_\nu (H)(x))\,dx=\fz$ and
the inequality automatically holds. Now, assume that
$|O_\lz|<\fz$. Applying Lemma \ref{l3.1} with
$F_\lz\equiv (O_\lz)^\com$, we have
\begin{eqnarray*}
  \iint_{\car_\eta(F_\lz^\ast)}|H(y,t)|^2\frac{\,dy\,dt}{t}&&
  \ls \int_{F_\lz}
  \iint_{\Gamma_\nu(x)}|H(y,t)|^2\frac{\,dy\,dt}{t^{n+1}}\,dx
  \ls \int_{F_\lz}[\ca_\nu(H)(x)]^2\,dx.
\end{eqnarray*}
Here and in what follows, we denote $(F_\lz)^\ast$ and
$(O_\lz)^\ast=((F_\lz)^\ast)^\com$ simply by $F_\lz^\ast$
and $O_\lz^\ast$, respectively. Observe that
$$\int_{F_\lz^\ast}\iint_{\Gamma_\eta(x)}
|H(y,t)|^2\frac{\,dy\,dt}{t^{n+1}}\,dx\ls
\iint_{\car_\eta(F_\lz^\ast)}|H(y,t)|^2\frac{\,dy\,dt}{t},$$
which implies that
$$\int_{F_\lz^\ast}[\ca_\eta(H)(x)]^2\,dx\ls
\int_{F_\lz}[\ca_\nu(H)(x)]^2\,dx.$$
Here and in what follows, for a measurable function $g$ on $\rn$ and $\lz>0$,
let $\sz_g(\lz)$ denote the distribution of $g$, namely,
$\sz_g(\lz)=|\{x\in\rn:\, |g(x)|>\lz\}|$. Hence, we have
\begin{eqnarray*}
  \sz_{ \ca_\eta(H)}(\lz)&&\le |O_{\lz}^\ast|+\frac{1}{\lz^2}
  \int_{(O_{\lz})^\com}[\ca_\nu(H)(x)]^2\,dx
  \ls |O_{\lz}|+\frac{1}{\lz^2}\int_0^\lz t\sz_{\ca_\nu(H)}(t)\,dt.
\end{eqnarray*}
Since $\oz$ is of upper type 1 and lower type $p_\oz\in (0,1]$, we have
\begin{equation}\label{3.2}
\oz(t)\sim \int_0^t\frac{\oz(u)}{u}\,du
\end{equation}
for each $t\in (0,\fz)$, which further implies that
\begin{eqnarray*}
\int_{\rn} \oz(\ca_\eta (H)(x))\,dx&&\sim \int_{\rn}
\int_0^{\ca_\eta (H)(x)}\frac{\oz(t)}{t}\,dt\,dx
\sim \int_0^\fz \sz_{\ca_\eta (H)}(t)\frac{\oz(t)}{t}\,dt\\
&&\ls \int_0^\fz \sz_{\ca_\nu (H)}(t)\frac{\oz(t)}{t}\,dt+
\int_0^\fz \frac{\oz(t)}{t^3}\int_0^t s\sz_{\ca_\nu(H)}(s)\,ds\,dt\\
&&\ls\int_0^\fz \sz_{\ca_\nu (H)}(t)\frac{\oz(t)}{t}\,dt+
\int_0^\fz s\sz_{\ca_\nu(H)}(s)\int_s^\fz\frac{\oz(t)}{t^3}\,dt \,ds\\
&&\ls\int_{\rn} \oz(\ca_\nu (H)(x))\,dx.
\end{eqnarray*}
This proves the first inequality in \eqref{3.1}, and hence,
 finishes the proof of Lemma \ref{l3.2}.
\end{proof}

We next give the atomic characterization of the tent space $T_\oz(\rnz)$.
Let $p\in (1,\fz)$. A function $a$ on $\rr^{n+1}_+$ is called an $(\oz,p)$-atom if

 (i) there exists a ball $B\subset
\rn$ such that $\supp a\subset \widehat{B};$

 (ii) $\|a\|_{T_2^p(\rnz)}\le |B|^{1/p-1}[\ro(|B|)]^{-1}.$

Since $\oz$ is concave, by the Jensen inequality, it is
easy to see that for all $(\oz,p)$-atoms $a$, we have
$\|a\|_{T_{\oz}(\rnz)}\le 1.$

Furthermore, if $a$ is an $(\oz, p)$-atom for all $p\in (1,\fz)$, we then
call $a$ an $(\oz,\fz)$-atom.

\begin{thm}\label{t3.1}
Let $\oz$ satisfy Assumption (A). Then for any $f\in
T_\oz(\rnz)$, there exist $(\oz,\fz)$-atoms $\{a_j\}_{j=1}^\fz$
and numbers $\{\lz_j\}_{j=1}^\fz\subset \cc$ such that
for almost every $(x,t)\in\rnz$,
\begin{equation}\label{3.3}
  f(x,t)=\sum_{j=1}^\fz\lz_ja_j(x,t).
\end{equation}
Moreover, there exists a positive constant $C$
such that for all $f\in T_\oz(\rnz)$,
\begin{equation}\label{3.4}
\Lambda(\{\lz_ja_j\}_j)\equiv\inf\lf\{\lz>0:\,\sum_{j=1}^\fz|B_j|\oz\lf(\frac{|\lz_j|}
{\lz |B_j|\ro(|B_j|)}\r)\le1\r\}\le
C\|f\|_{T_\oz(\rnz)},
\end{equation}
where $\widehat{B_j}$ appears as the support of $ a_j$.
\end{thm}

\begin{proof}[\bf Proof.] We prove this theorem by borrowing some
ideas from the proof of Theorem 1 in Coifman, Meyer and Stein \cite{cms}.
Let $f\in T_\oz(\rnz)$. For any $k\in\zz$, let
$O_k\equiv \{x\in\rn:\,\ca(f)(x)>2^k\}$
and $F_k\equiv (O_k)^\com$. Since $f\in T_\oz(\rnz)$, for each $k$, $O_k$
is an open set and $|O_k|<\fz$.

Since $\oz$ is of upper type $1$, by Lemma \ref{l3.1},
for $k\in\zz$ and $k\le 0$, we have
\begin{eqnarray*}
  \iint_{\car(F_k^\ast)}|f(y,t)|^2\frac{\,dy\,dt}{t}&&\ls
  \int_{F_k}\iint_{\Gamma (x)}|f(y,t)|^2\frac{\,dy\,dt}{t^{n+1}}\,dx\\
  &&\ls \int_{F_k}\ca(f)(x)^2\,dx
  \ls\int_{F_k}\oz(\ca(f)(x))dx\to 0,
\end{eqnarray*}
as $k\to -\fz$, which implies that $f=0$ almost everywhere in
$\cap_{k\in\zz}\car(F_k^\ast)$, and hence,
$\supp f\subset \{\cup_{k\in\zz} \widehat{O_k^\ast}\cup E\},$
where $E\subset \rnz$ and $\iint_{E}\,\frac {dx\,dt}t=0$.

Thus, for each $k$, by applying the Whitney decomposition to the set $O_k^\ast$,
we obtain a set $I_k$ of indices and a family $\{Q_{k,j}\}_{j\in I_k}$  of disjoint cubes such that

(i) $\cup_{j\in I_k}Q_{k,j}=O_k^\ast$, and if $i\neq j$, then $Q_{k,j}\cap Q_{k,i}=\emptyset$,

(ii) $\sqrt n\ell(Q_{k,j})\le \dist(Q_{k,j}, (O_k^\ast)^\com)\le 4\sqrt n\ell(Q_{k,j})$,
where $\ell(Q_{k,j})$ denotes the side-length of $Q_{k,j}$.

Next, for each $j\in I_k$, we choose a ball $B_{k,j}$ with the same center as
$Q_{k,j}$ and with radius $\frac {11}{2}\sqrt n$-times $\ell(Q_{k,j})$. Let
$A_{k,j}\equiv \widehat{B_{k,j}}\cap(Q_{k,j}\times (0,\fz))\cap (\widehat{O_k^\ast}
\setminus \widehat{O_{k+1}^\ast}),$
$$a_{k,j}\equiv  2^{-k}|B_{k,j}|^{-1}[\ro(|B_{k,j}|)]^{-1}f\chi_{A_{k,j}}$$
 and $\lz_{k,j}\equiv 2^k|B_{k,j}|\ro(|B_{k,j}|).$
Notice that $\{(Q_{k,j}\times (0,\fz))\cap (\widehat{O_k^\ast}
\setminus \widehat{O_{k+1}^\ast})\}\subset \widehat{B_{k,j}}$. From this,
we conclude that $f=\sum_{k\in \zz}\sum_{j\in I_k}\lz_{k,j}a_{k,j}$ almost everywhere.

Let us show that for each $k\in\zz$ and $j\in I_k$,
$a_{k,j}$ is an $(\oz,\fz)$-atom supported
in $\widehat B_{k,j}$. Let $p\in (1,\fz)$, $q\equiv p'$ be the conjugate index
of $p$, i.\,e., $1/q+1/p=1$, and $h\in T_2^q(\rnz)$
 with $\|h\|_{T_2^q(\rnz)}\le 1$. Since
$A_{k,j}\subset (\widehat{O_{k+1}^\ast})^\com=\car(F_{k+1}^\ast)$, by Lemma
\ref{l3.1} and the H\"older inequality, we have
\begin{eqnarray*}
  |\la a_{k,j}, h\ra |&&\le \iint_{\rnz} |(a_{k,j}\chi_{A_{k,j}})
  (y,t)h(y,t)|\frac{\,dy\,dt}{t}\\
  &&\ls \int_{F_{k+1}} \iint_{\Gamma(x)}|a_{k,j}(y,t)h(y,t)|
  \frac{\,dy\,dt}{t^{n+1}}\,dx
 \ls \int_{(O_{k+1})^\com} \ca(a_{k,j})(x)\ca(h)(x)\,dx\\
  &&\ls 2^{-k}|B_{k,j}|^{-1}[\ro(|B_{k,j}|)]^{-1}
  \lf(\int_{B_{k,j}\cap O_{k+1}^\com} [\ca(f)(x)]^p\,dx
  \r)^{1/p}\|h\|_{T_2^q(\rnz)}\\
  &&\ls |B_{k,j}|^{1/p-1}[\ro(|B_{k,j}|)]^{-1},
  \end{eqnarray*}
which implies that $a_{k,j}$ is an $(\oz,p)$-atom supported in $\widehat B_{k,j}$
for all $p\in (1,\fz)$, hence, an $(\oz,\fz)$-atom.

By \eqref{3.2}, for any $\lz>0$, we further obtain
\begin{eqnarray}\label{3.5}
&&\sum_{k\in\zz}\sum_{j\in I_k}|B_{k,j}|\oz\lf(\frac{|\lz_{k,j}|}
{ \lz|B_{k,j}|\ro(|B_{k,j}|)}\r)\nonumber\\
&&\hs\ls \sum_{k\in\zz}\sum_{j\in I_k} |Q_{k,j}|\oz\lf(\frac {2^k}{\lz}\r)
\ls \sum_{k\in\zz}|O_{k}^\ast|\oz\lf(\frac {2^k}{\lz}\r)\ls \sum_{k\in\zz}|O_{k}|
\oz\lf(\frac {2^k}{\lz}\r)\nonumber\\
&&\hs\ls \sum_{k\in\zz}\int_{O_k}\oz\lf(\frac {2^k}{\lz}\r)\,dx\ls \int_{\rn}
\sum_{k<\log_2[\ca(f)(x)]}\oz\lf(\frac {2^k}{\lz}\r)\,dx\nonumber\\
&&\hs\ls \int_{\rn} \sum_{k<\log_2[\ca(f)(x)]}\int_{2^{k}}^{2^{k+1}}
\oz\lf(\frac {t}{\lz}\r)\frac{\,dt}{t}\,dx\nonumber\\
&&\hs\ls\int_{\rn} \int_{0}^{\frac{2\ca(f)(x)}{\lz}}\oz(t)\frac{\,dt}{t}\,dx
\ls\int_{\rn}\oz\lf(\frac{\ca(f)(x)}{\lz}\r)\,dx,
\end{eqnarray}
which implies that \eqref{3.4} holds, and hence, completes the
proof of Theorem \ref{t3.1}.
\end{proof}

\begin{rem}\label{r3.1}\rm
(i) Notice that the definition $\Lambda(\{\lz_ja_j\}_j)$ in \eqref{3.4}
is different from \cite{hsv,v}. In fact,
if $p\in (0,1]$ and $\oz(t)=t^p$ for all $t\in (0,\fz)$, then
$\Lambda(\{\lz_ja_j\}_j)$ here coincides with $(\sum_j |\lz_j|^p)^{1/p}$, which
seems to be natural.

(ii) Let $\{\lz_j^i\}_{i,j}\subset \cc$ and $\{a_j^i\}_{i,j}$ be $(\oz,p)$-atoms
for certain $p\in (1,\fz)$, where $i=1,2$.
If $\sum_j\lz_j^1a_j^1,\,\sum_j\lz_j^2a_j^2\in T_\oz(\rnz)$, then
by the fact that $\oz$ is subadditive and of strictly lower type $p_\oz$,
we have
$$[\Lambda(\{\lz_j^ia_j^i\}_{i,j})]^{p_\oz}\le \sum_{i=1}^2
[\Lambda(\{\lz_j^ia_j^i\}_j)]^{p_\oz}.$$

(iii) Since $\oz$ is concave, it is of upper type 1. Then, with
the same notation as in Theorem \ref{t3.1}, we have
$\sum_{j=1}^\fz|\lz_j|\le C\Lambda(\{\lz_ja_j\}_j) \le C\|f\|_{T_\oz(\rnz)}.$
\end{rem}

Let $p\in (0,1]$ and $q\in (p,\fz)\cap [1,\fz)$. Recall that
a function $a$ on $\rr^{n+1}_+$ is called a $(p,q)$-atom if

 (i) there exists a ball $B\subset
\rn$ such that $\supp a\subset \widehat{B};$

 (ii) $\|a\|_{T_2^q(\rnz)}\le |B|^{1/q-1/p}.$

We have the following convergence result.
\begin{prop}\label{p3.1}
  Let $\oz$ satisfy Assumption (A) and $p\in (0,\fz)$.
  If $f\in (T_\oz(\rnz)\cap T_{2}^p(\rnz))$, then the decomposition
  \eqref{3.3} holds in both $T_\oz(\rnz)$ and $T_{2}^p(\rnz)$.
\end{prop}

\begin{proof}[\bf Proof.] We use the same notation as in the proof of Theorem \ref{t3.1}.
We first show that \eqref{3.3} holds in $T_\oz(\rnz)$. In fact, since $\oz$
is concave and $\oz^{-1}$ is convex, by the Jensen inequality and the H\"older
inequality, for each $k\in\zz$ and $j\in I_k$, we have
\begin{eqnarray*}
  \oz^{-1}\lf(\frac{1}{|B_{k,j}|}\int_\rn \oz(\ca(\lz_{k,j}a_{k,j})(x))\,dx\r)
&&\le \frac{|\lz_{k,j}|}{|B_{k,j}|}\int_\rn \ca(\lz_{k,j}a_{k,j})(x)\,dx\\
&&\le \frac{|\lz_{k,j}|}{|B_{k,j}|^{1/2}}\|a_{k,j}\|_{T_2^2(\rnz)}\le
\frac{|\lz_{k,j}|}{|B_{k,j}|\ro(|B_{k,j}|)}.
\end{eqnarray*}
From this and the continuity of $\oz$ together with the subadditive property
of $\oz$ and $\ca$, it follows that
\begin{eqnarray}\label{3.6}
  &&\int_\rn \oz \lf(\ca\lf(f-\sum_{|k|+|j|\le N}\lz_{k,j}a_{k,j}\r)(x)\r)\,dx\nonumber\\
&&\hs\le \sum_{|k|+|j|> N}\int_\rn \oz \lf(\ca(\lz_{k,j}a_{k,j})(x)\r)\,dx
\ls \sum_{|k|+|j|> N}|B_{k,j}|
\oz\lf(\frac{|\lz_{k,j}|}{|B_{k,j}|\ro(|B_{k,j}|)}\r)\to 0,\quad\quad
\end{eqnarray}
as $N\to \fz$, by \eqref{3.5}. Now for any $\ez>0$, by the fact that
$\oz$ is of upper type $1$ and \eqref{3.6},
there exists $N_0\in\cn$ such that when $N>N_0$,
$$\int_\rn \oz \lf(\frac 1\ez
\ca\lf[f-\sum_{|k|+|j|\le N}\lz_{k,j}a_{k,j}\r](x)\r)\,dx\le 1,$$
which implies that when $N>N_0$,
$\|f-\sum_{|k|+|j|\le N}\lz_{k,j}a_{k,j}\|_{T_\oz(\rnz)}\le \ez$. Thus, \eqref{3.3}
holds in $T_\oz(\rnz)$.

We now prove that \eqref{3.3} holds in $T_2^p(\rnz)$.
For the case $p\in (0,1]$, notice that $\{A_{k,j}\}_{k\in\zz,j\in I_k}$
are independent of $\oz$. In this case, letting $\wz a_{k,j}\equiv  2^{-k}
|B_{k,j}|^{-1/p}f\chi_{A_{k,j}}$
and $\wz\lz_{k,j}\equiv 2^k|B_{k,j}|^{1/p},$ we then have
that $\{a_{k,j}\}_{k\in\zz,j\in I_k}$ are $(p,q)$-atoms, where $q\in (p,\fz)\cap [1,\fz)$, and
$\sum_{k\in\zz}\sum_{j\in I_k}|\wz \lz_{k,j}|^p\ls \|f\|_{T_2^p(\rnz)}^p$,
which combined with the fact
that $\lz_{k,j}a_{k,j}=\wz\lz_{k,j}\wz a_{k,j}$ implies that \eqref{3.3} holds
in $T_2^p(\rnz)$ in this case.

Let us now consider the case $p\in (1,\fz)$. To prove that \eqref{3.2} holds in
$T_2^p(\rnz)$, it suffices to show that for any $\bz>0$,
there exists $N_0\in \cn$ such that if $N>N_0$, then
\begin{equation}\label{3.7}
 \lf\|\sum_{|k|+|j|>N}\lz_{k,j}a_{k,j}\r\|_{T_2^p(\rnz)}=
 \lf\|\sum_{|k|+|j|>N}f\chi_{A_{k,j}}\r\|_{T_2^p(\rnz)}<\bz.
 \end{equation}

To see this, noticing that $\{A_{k,j}\}_{k\in\zz,j\in I_k}$ are disjoint, hence, we have
\begin{equation}\label{3.8}
\sum_{k\in\zz}\sum_{j\in I_k}|f\chi_{A_{k,j}}|=|f|.
\end{equation}
Write $ \mathrm{H}_{N,1}\equiv\sum_{k<-N,\,j\in I_k}f\chi_{A_{k,j}}$ and
$\mathrm{H}_{N,2}\equiv\sum_{k>N,\,j\in I_k}f\chi_{A_{k,j}}$.
To estimate the term  $\mathrm{H}_{N,1}$, let $q$ be the conjugate index of $p$
and $h\in T_2^q(\rnz)$ with $\|h\|_{T_2^q(\rnz)}\le 1$. Notice that for each $k<-N$,
$A_{k,j}\subset (\widehat{O_{-N}^\ast})^\com$, and hence
$\supp \mathrm{H}_{N,1}\subset (\widehat{O_{-N}^\ast})^\com= \car(F_{-N}^\ast)$.
From this, \eqref{3.8}, Lemma \ref{l3.1} and the H\"older inequality, we deduce that
\begin{eqnarray*} |\la \mathrm{H}_{N,1}, h\ra |&&\le
\iint_{\car(F_{-N}^\ast)} \lf|\sum_{k<-N,\,j\in I_k}(f\chi_{A_{k,j}})(y,t) h(y,t)
\r|\frac{\,dy\,dt}{t}\\
&&\ls \int_{F_{-N}}\iint_{\Gamma(x)}\lf|\sum_{k<-N,\,j\in I_k}(f\chi_{A_{k,j}})(y,t) h(y,t)
\r|\frac{\,dy\,dt}{t^{n+1}}\,dx\\
&&\ls \int_{F_{-N}}\ca(f)(x)\ca(h)(x)\,dx\ls
\lf(\int_{F_{-N}}[\ca(f)(x)]^p\,dx\r)^{1/p},
\end{eqnarray*}
which implies that
$$\|\mathrm{H}_{N,1}\|_{T_2^p(\rnz)}\ls \lf(\int_{F_{-N}}[\ca(f)(x)]^p
\,dx\r)^{1/p}.$$
 Then by the Lebesgue dominated convergence theorem, we have
$$\lim_{N\to \fz}\|\mathrm{H}_{N,1}\|_{T_2^p(\rnz)}=0,$$
which implies that there exists $N_1\in \cn$ such that if $N\ge N_1$,
 then $\|\mathrm{H}_{N,1}\|_{T_2^p(\rnz)}<\bz/3.$

For the term  $\mathrm{H}_{N,2}$, notice that for each $k>N$,
$A_{k,j}\subset \widehat{O_{N}^\ast}$ and hence,
$\supp \mathrm{H}_{N,2}\subset \widehat{O_{N}^\ast}$,
which together with \eqref{3.8} implies that
\begin{eqnarray*} \|\mathrm{H}_{N,2}\|_{T_2^p(\rnz)}^p&&=
\int_{\rn}\lf[\ca\lf(\sum_{k>N,\,j\in I_k}f\chi_{A_{k,j}}\r)(x)\r]^p\,dx
\le\int_{O_{N}^\ast}[\ca(f)(x)]^p\,dx.
\end{eqnarray*}
Since $|O_{N}^\ast|\ls |O_N|\to 0$ as $N\to\fz$,
by the continuity of Lebesgue integrals (or the Lebesgue dominated
convergence theorem in measures), we have
$$\lim_{N\to \fz}\|\mathrm{H}_{N,2}\|_{T_2^p(\rnz)}=0,$$
which implies that there exists $N_2\in \cn$ such that if $N\ge N_2$,
then $\|\mathrm{H}_{N,2}\|_{T_2^2(\rnz)}<\bz/3.$

Now let $ \mathrm{H}_{N,3}\equiv\sum_{-N_1\le k\le N_2,\, |k|+|j|>N}f\chi_{A_{k,j}}.$
Since $A_{k,j}\subset \widehat{B_{k,j}}$,  by \eqref{3.8}, we obtain
\begin{eqnarray*} \|\mathrm{H}_{N,3}\|_{T_2^p(\rnz)}^p&&=
\int_{\rn}\lf[\ca\lf(\sum_{-N_1\le k\le N_2,\, |k|+|j|>N}f
\chi_{A_{k,j}}\r)(x)\r]^p\,dx\\
&&\le\int_{\bigcup_{-N_1\le k\le N_2,\, |k|+|j|>N}B_{k,j}}[\ca(f)(x)]^p\,dx.
\end{eqnarray*}
From the Whitney decomposition, it follows that for each fixed $k$,
$$\sum_{j\in I_k} |B_{k,j}|\ls\sum_{j\in I_k} |Q_{k,j}|\ls |O_k^\ast|\ls |O_k| <\fz,$$
and hence, $\lim_{N\to\fz}\sum_{\{j\in I_k:\,|j|>N\}}|B_{k,j}|=0,$
which implies that
$$\lim_{N\to \fz}\lf|\bigcup_{-N_1\le k\le N_2,\,|k|+|j|>N}B_{k,j}\r|
\le \lim_{N\to \fz}\sum_{-N_1\le k\le N_2}\sum_{|j|+|k|>N}|B_{k,j}|=0.$$
Applying the continuity of Lebesgue integrals (or the Lebesgue dominated
convergence theorem in measures) again, we obtain
$$\lim_{N\to \fz}\|\mathrm{H}_{N,3}\|_{T_2^p(\rnz)}=0,$$
which implies that there exists $N_3\in \cn$ such that if $N\ge N_3$,
then $\|\mathrm{H}_{N,3}\|_{T_2^p(\rnz)}<\bz/3.$

Letting $N_0\equiv \max\{N_1,N_2,N_3\}$ and noticing that when $N>N_0$,
$$\lf\|\sum_{|k|+|j|>N}f\chi_{A_{k,j}}\r\|_{T_2^p(\rnz)}=
\lf\|\sum_{|k|+|j|>N}|f\chi_{A_{k,j}}|\r\|_{T_2^p(\rnz)}\le
\sum_{i=1}^3\|H_{N_i,i}\|_{T_2^p(\rnz)}<\bz,$$
we then obtain \eqref{3.7},
which completes the proof of Proposition \ref{p3.1}.
\end{proof}

As a consequence of Proposition \ref{p3.1}, we have the following
corollary which plays an important role in this paper.

\begin{cor}\label{c3.1}
  Let $\oz$ satisfy Assumption (A). If $f\in T_{\oz}(\rnz)\cap
  T_2^2(\rnz)$, then $f\in T_{2}^p(\rnz)$ for all $p\in [1,2]$, and hence,
  the decomposition \eqref{3.3}  holds in $T_{2}^p(\rnz)$.
\end{cor}
\begin{proof}[\bf Proof.]
Observing that $\oz$ is of upper type $1$, we have
\begin{eqnarray*}
\int_{\rn}[\ca(f)(x)]^p\,dx&&\le\int_{\{x\in\rn:\,\ca(f)(x)< 1 \}}\ca(f)(x)\,dx
  +\int_{\{x\in\rn:\,\ca(f)(x)\ge 1 \}}[\ca(f)(x)]^2\,dx\\
  &&\ls \int_{\{x\in\rn:\,\ca(f)(x)< 1 \}}\oz(\ca(f)(x))\,dx
  +\|f\|_{T_2^2(\rnz)}^2<\fz,
\end{eqnarray*}
which implies that $f\in T_2^p(\rnz).$ Then by  Proposition \ref{p3.1},
we have that the decomposition \eqref{3.3} holds in $T_{2}^p(\rnz)$,
which completes the proof of Corollary \ref{c3.1}.
\end{proof}

In what follows, let $T_\oz^c(\rnz)$ and $T^{p,c}_2(\rnz)$ denote the set
of all functions in $T_\oz(\rnz)$ and $T^{p}_2(\rnz)$ with compact supports,
respectively, where $p\in (0,\fz)$.

\begin{lem}\label{l3.3}
(i) For all $p\in (0,\,\fz)$,
$T^{p,c}_2(\rnz)\subset T_2^{2,c}(\rnz).$
In particular, if $p\in (0,2]$, then
$T^{p,c}_2(\rnz)$ coincides with $T_2^{2,c}(\rnz)$.

(ii) Let $\oz$ satisfy Assumption (A). Then
$T^c_\oz(\rnz)$ coincides with $T_2^{2,c}(\rnz).$
\end{lem}
\begin{proof}[\bf Proof.]  By (1.3) in \cite[p.\,306]{cms}, we have
$T^{p,c}_2(\rnz)\subset T_2^{2,c}(\rnz)$
for all $p\in (0,\,\fz)$. If $p\in (0,2]$, then
from the H\"older inequality, it is easy to follow that
$ T_2^{2,c}(\rnz)\subset T^{p,c}_2(\rnz)$. Thus, (i) holds.

Let us prove (ii).
To prove $T^c_\oz(\rnz)\subset T_2^{2,c}(\rnz)$, by
(i), it suffices to show that  $T^c_\oz(\rnz) \subset
T^{p,c}_2(\rnz)$ for certain $p\in(0,\fz).$ Suppose that
$f\in T^c_\oz(\rnz)$ and $\supp f\subset K$, where $K$ is a
compact set in ${\rr}^{n+1}_+.$ Let $B$ be a ball in $\rn$
such that $K\subset \widehat B$. Then $\supp\ca(f)\subset B$.
This, together with the lower type property of $\oz$, yields that
\begin{eqnarray*}
\int_{\rn}
[\ca(f)(x)]^{p_\oz}\,dx&&=\int_{\{x\in{\rn}:\,\ca(f)(x)< 1\}}
[\ca(f)(x)]^{p_\oz}\,dx+\int_{\{x\in\rn:\,\ca(f)(x)\ge 1\}}
\cdots\\
&&\ls |B|+\int_{\rn}\oz(\ca(f)(x))\,dx<\fz.
\end{eqnarray*}
That is, $f\in T_2^{p_\oz,c}(\rnz)\subset T_2^{2,c}(\rnz)$.

Conversely, let $f\in T_2^{1,c}(\rnz)$ supporting in a compact set
$K$ in ${\rr}^{n+1}_+.$ Then there exists a ball $B$ such that $K\subset \widehat B$
and $\supp \ca(f)\subset B$. This, together with the upper type property of
 $\oz$, yields that
\begin{eqnarray*}
\int_{\rn}
\oz(\ca(f)(x))\,dx&&\ls \int_{\{x\in{\rn}:\,\ca(f)(x)< 1\}}
\oz(1)\,dx+\int_{\{x\in\rn:\,\ca(f)(x)\ge 1\}}
\ca(f)(x)\,dx\\
&&\ls  |B|+\|f\|_{T_2^1(\rnz)}<\fz,
\end{eqnarray*}
which implies that $f\in  T^c_\oz(\rnz)$, and hence,
 completes the proof of Lemma \ref{l3.3}.
\end{proof}

\section{Orlicz-Hardy spaces and their dual spaces\label{s4}}

\hskip\parindent In this section, we always assume that the Orlicz
function $\oz$ satisfies Assumption (A). We introduce the
Orlicz-Hardy space associated to $L$ via the Lusin-area function and
establish its duality. Let us begin with some notions and notation.

Let $\cs_L$ be the same as in \eqref{1.3}.
It follows from Lemma \ref{l2.4} that the operator $\cs_L$ is bounded
on $L^p(\rn)$ for $p\in (p_L,\wz p_L)$.
Hofmann and Mayboroda \cite{hm1} introduced the Hardy space
$H_L^1(\rn)$ associated to $L$ as
the completion of $\{f\in L^2(\rn):\,\cs_L f\in L^1(\rn)\}$ with
respect to the norm
$\|f\|_{H_L^1(\rn)}\equiv\|\cs_L f\|_{L^1(\rn)}.$

Using some ideas from \cite{dy2,hm1}, we now introduce the
Orlicz-Hardy space $H_{\oz,L}(\rn)$ associated to $L$ and $\oz$ as
follows.
\begin{defn}\label{d4.1}
Let $\oz$ satisfy Assumption (A). A function $f\in L^2(\rn)$ is said to be in $\wz
H_{\oz,L}(\rn)$ if $\cs_L f\in L(\oz)$; moreover, define
$$\|f\|_{H_{\oz,L}(\rn)}\equiv \|\cs_L f\|_{L(\oz)}=\inf\lf\{\lz>0:\int_{\rn}\oz\lf
(\frac{\cs_L f(x)}{\lz}\r)\,dx\le 1\r\}.$$ The Orlicz-Hardy space
$H_{\oz,L}(\rn)$ is defined to be the  completion of $\wz
H_{\oz,L}(\rn)$  in the norm $\|\cdot\|_{H_{\oz,L}(\rn)}.$
\end{defn}

In what follows, for a ball $B\equiv B(x_B,r_B)$, we let $U_0(B)\equiv B$, and
for $j\in \cn$, $U_j(B)\equiv B(x_B,2^{j}r_B)\setminus B(x_B,2^{j-1}r_B)$.

\begin{defn}\label{d4.2}
 Let $q\in (p_L, \wz p_L)$, $M\in\cn$ and $\ez\in(0,\fz)$.
 A function $\az\in L^q(\rn)$ is called an $(\oz,q,M,\ez)$-molecule
adapted to $B$ if there exists a ball $B$ such that

{\rm (i)} $\|\az\|_{L^q(U_j(B))}\le 2^{-j\ez}|2^jB|^{1/q-1}\ro(|2^jB|)^{-1}$,
$j\in {\zz}_+$;

{\rm (ii)} for every $k=1,\cdots,M$ and $j\in {\zz}_+$, there holds
$$\|(r_B^{-2}L^{-1})^{k}\az\|_{L^q(U_j(B))}\le
2^{-j\ez}|2^jB|^{1/q-1}[\ro(|2^jB|)]^{-1}.$$
Finally, if $\az$ is an $(\oz,q,M,\ez)$-molecule for all
$q\in (p_L, \wz p_L)$, then $\az$ is called an $(\oz,\fz,M,\ez)$-molecule.
\end{defn}

\begin{rem}\label{r4.1}\rm
(i) Since $\oz$ is of strictly lower type $p_\oz$, we have that
for all $f_1,\,f_2\in H_{\oz,L}(\rn)$,
$$\|f_1+f_2\|_{H_{\oz,L}(\rn)}^{p_\oz} \le \|f_1\|_{H_{\oz,L}(\rn)}^{p_\oz}
 +\|f_2\|_{H_{\oz,L}(\rn)}^{p_\oz}.$$
In fact, if letting $\lz_1\equiv \|f_1\|_{H_{\oz,L}(\rn)}^{p_\oz}$ and
$\lz_2\equiv \|f_2\|_{H_{\oz,L}(\rn)}^{p_\oz}$, by the subadditivity,
the continuity and the lower type $p_\oz$ of $\oz$, we have
\begin{eqnarray*}
  \int_{\rn}\oz\lf(\frac{\cs_L(f_1+f_2)(x)}{(\lz_1+\lz_2)^{1/p_\oz}}\r)\,dx
&&\le \sum_{i=1}^2\int_{\rn}\oz\lf(\frac{\cs_L(f_i)(x)}{(\lz_1+\lz_2)^{1/p_\oz}}\r)\,dx\\
&&\le \sum_{i=1}^2\frac{\lz_i}{\lz_1+\lz_2}
\int_{\rn}\oz\lf(\frac{\cs_L(f_i)(x)}{\lz_1^{1/p_\oz}}\r)\,dx\le 1,
\end{eqnarray*}
which implies $\|f_1+f_2\|_{H_{\oz,L}(\rn)}\le (\|f_1\|_{H_{\oz,L}(\rn)}^{p_\oz}
+\|f_2\|_{H_{\oz,L}(\rn)}^{p_\oz})^{1/p_\oz}$, and hence, the desired conclusion.

(ii) From the theorem of completion of Yosida \cite[p.\,56]{yo}, it
follows that
  $\wz H_{\oz,\,L}(\rn)$  is dense in $H_{\oz,\,L}(\rn)$,
  namely, for any $f\in H_{\oz,\,L}(\rn)$, there exists a Cauchy sequence
   $\{f_k\}^{\fz}_{k=1}\subset \wz H_{\oz,\,L}(\rn)$
 such that $\lim_{k\to\fz}\|f_k-f\|_{H_{\oz,\,L}(\rn)}=0.$
 Moreover, if $\{f_k\}^{\fz}_{k=1}$ is a Cauchy sequence in
$\wz H_{\oz,\,L}(\rn)$, then there uniquely exists $f\in
H_{\oz,\,L}(\rn)$ such that
$\lim_{k\to\fz}\|f_k-f\|_{H_{\oz,\,L}(\rn)}=0.$

(iii) If $\oz(t)= t$, then the space $H_{\oz,\,L}(\rn)$ is just the
space $H^1_L(\rn)$ introduced by  Hofmann and Mayboroda \cite{hm1}.
Furthermore, when $\oz(t)\equiv t^p$ for all $t\in(0,\fz)$
with $p\in (0,1]$, we then denote the space $H_{\oz,L}(\rn)$ simply
by $H^p_L(\rn)$.
\end{rem}

\subsection{Molecular decompositions of $H_{\oz,L}(\rn)$\label{s4.1}}

\hskip\parindent In what follows, let $L^2_c(\rnz)$ denote the set of all functions
in $L^2(\rnz)$ with compact supports.
Recall that $\oz$ is a concave function of strictly lower type $p_\oz$, where
$p_\oz\in (0, 1]$.

\begin{prop}\label{p4.1} Let $\oz$ satisfy Assumption (A),
$M\in\cn$ and $M>\frac n2 (\frac1{p_\oz}-\frac 12)$, and $\pi_{L,M}$ be
as in \eqref{1.4}.

{\rm (i)} The operator $\pi_{L,M}$, initially defined on $T_2^{p,c}(\rnz)$, extends to
a bounded linear operator from $T_2^{p}(\rnz)$ to $L^p(\rn),$ where $p\in (p_L,\wz p_L)$.

{\rm (ii)} The operator $\pi_{L,M}$,  initially defined on $T_\oz^{c}(\rnz)$,
extends to a bounded linear operator from $T_\oz(\rnz)$ to $H_{\oz,L}(\rn).$
\end{prop}
\begin{proof}[\bf Proof.] Let $k\in\cn$. By Lemma \ref{l2.4} and a duality argument, we know that
the operator $\cs_{L^\ast}^k$ is bounded on $L^p(\rn)$ for
$p\in (p_{L^\ast},\wz p_{L^\ast})$,
 where  $\frac 1 {p_{L^\ast}}+\frac 1 {\wz p_{L}}=1
=\frac 1 {p_{L}}+\frac 1 {\wz p_{L^\ast}}$.

Let $f\in T_2^{p,c}(\rnz)$, where $p\in (p_L,\wz p_L)$. For any $g\in L^q(\rn)
\cap L^2(\rn)$, where
$\frac 1 {p}+\frac 1 {q}=1$, by the H\"older inequality, we have
\begin{eqnarray*}
\lf|\int_\rn \pi_{L,M}(f)(x)g(x)\,dx\r|&&\ls \lf|\iint_{\rr^{n+1}_+}f(y,t)
(t^2L^\ast)^{M+1}e^{-t^2L^\ast}g(y)\frac{\,dy\,dt}{t}\r|\\
&&\ls \int_\rn \ca(f)(x)\cs_{L^\ast}^{M+1}g(x)\,dx
\ls \|\ca(f)\|_{L^p(\rn)}\|\cs_{L^\ast}^{M+1}g\|_{L^q(\rn)}\\
&&\ls\|f\|_{T^p_2(\rnz)}\|g\|_{L^q(\rn)},
\end{eqnarray*}
which implies that $\pi_{L,M}$ maps $T_2^{p,c}(\rnz)$ continuously into
$L^p(\rn).$ Then by a density argument, we obtain that $\pi_{L,M}$ is
bounded from $T_2^{p}(\rnz)$ to $L^p(\rn)$. This proves (i).

Let us prove (ii). Assume that
$f\in T_\oz^c(\rnz)$. By Theorem \ref{t3.1}, we have
$f=\sum_{j=1}^\fz \lz_ja_j$ pointwise, where $\{\lz_j\}_{j=1}^\fz$
and $\{a_j\}_{j=1}^\fz$ are
as in Theorem \ref{t3.1} and $\Lambda(\{\lz_ja_j\}_j)\ls \|f\|_{T_\oz(\rnz)}$.
From Lemma \ref{l3.3} (ii), it follows that $f\in T_2^{2,c}(\rnz)$,
which together with (i) and Corollary \ref{c3.1} further implies that
\begin{equation*}
\pi_{L,M}(f)=\sum_{j=1}^\fz\lz_j\pi_{L,M}(a_j)\equiv\sum_{j=1}^\fz\lz_j\az_j
\end{equation*}
in $L^p(\rn)$ for $p\in (p_L,2]$.

On the other hand, notice that the operator $\cs_L$ is bounded on
 $L^p(\rn)$, which together with the subadditivity and the continuity of $\oz$
 yields that
\begin{equation}\label{4.1}
\int_\rn \oz(\cs_L(\pi_{L,M}(f))(x))\,dx\le
\sum_{j=1}^\fz\int_\rn \oz(|\lz_j|\cs_L(\az_j)(x))\,dx.
\end{equation}

We claim that for any fixed $\ez\in(0,\fz)$, $\az_j=\pi_{L,M}(a_j)$
is a multiple of an $(\oz,\fz,M,\ez)$-molecule adapted to $B_j$ for each $j$.

In fact, assume that $a$ is an $(\oz,\fz)$-atom supported
in the ball $B\equiv B(x_B,r_B)$
and $q\in (p_L,\wz p_L)$. Since for $q\in (p_L,2)$,
each $(\oz,2,M,\ez)$-molecule is also an
$(\oz,q,M,\ez)$-molecule, to prove the above claim, it suffices to show that
$\az\equiv \pi_{L,M}(a)$ is a multiple of an $(\oz,q,M,\ez)$-molecule adapted to
$B$ with $q\in [2,\wz p_L)$.

 By (i), for $i=0,\,1,\,2$, we have
$$\|\az\|_{L^q(U_i(B))}=\|\pi_{L,M}(a)\|_{L^q(U_i(B))}\ls \|a\|_{T^q_2(\rnz)}\ls
|B|^{1/q-1}[\ro(|B|)]^{-1}.$$
For $i\ge 3$, let $q'\in (1,2]$ being the conjugate number of $q$ and
$h\in L^{q'}(\rn) $ satisfying $\|h\|_{L^{q'}(\rn)}\le 1$ and $\supp h\subset U_i(B)$.
By the H\"older inequality and Lemmas \ref{l2.1} and \ref{l2.3},
we have
\begin{eqnarray}\label{4.2}
  &&|\la\pi_{L,M}(a),h\ra|\nonumber\\
  &&\hs\ls \int_0^{r_B}\int_B|a(x,t)(t^2L^\ast)^{M+1}
  e^{-t^2L^\ast}(h)(x)|\frac{\,dx\,dt}{t}\nonumber\\
&&\hs\ls \|\ca(a)\|_{L^q(\rn)}\|\ca(\chi_{\widehat B}(t^2L^\ast)^{M+1}
  e^{-t^2L^\ast}(h))\|_{L^{q'}(\rn)}\nonumber\\
  &&\hs\ls \|a\|_{T_2^q(\rnz)}|B|^{1/q'-1/2}\lf(\int_{\widehat B}|(t^2L^\ast)^{M+1}
  e^{-t^2L^\ast}(h)(x,t)|^2\frac{\,dx\,dt}{t}\r)^{1/2}\nonumber\\
  &&\hs\ls \|a \|_{T_2^q(\rnz)}|B|^{1/q'-1/2}\lf(\int_{0}^{r_B}\lf[t^{n(1/2-1/q')}
  \exp\lf\{-\frac{\dist(B,U_i(B))^2}{ct^2}\r\}\r]^2\frac{\,dt}{t}\r)^{1/2}\nonumber\\
  &&\hs\ls |B|^{-1/2}\ro(|B|)^{-1}\lf(\int_{0}^{r_B}t^{n(1-2/q')}\lf[\frac{t}{2^ir_B}
  \r]^{2(\ez+n/p_\oz-n/q)}  \frac{\,dt}{t}\r)^{1/2}\nonumber\\
  &&\hs\ls 2^{-i\ez}|2^iB|^{1/q-1}[\ro(|2^iB|)]^{-1},
\end{eqnarray}
which implies that $\az$ satisfies Definition \ref{d4.2} (i).

We now show that $\az$ also satisfies Definition \ref{d4.2} (ii).
Let $k\in \{1,\,\cdots,\,M\}$. If $i=0,\,1,\,2,$ let $h$ be the same as in the proof
of \eqref{4.2}; similarly to the proof of \eqref{4.1}, we have
\begin{eqnarray*}
  \lf|\la (r_{B}^{-2}L^{-1})^k\pi_{L,M}(a), h\ra\r|
&&\ls \int_0^{r_B}\int_B\lf(\frac{t}{r_B}\r)^{2k}|a(x,t)(t^2L^\ast)^{M+1-k}
  e^{-t^2L^\ast}(h)(x)|\frac{\,dx\,dt}{t}\nonumber\\
&&\ls \|\ca(a)\|_{L^q(\rn)}\|\cs_{L^\ast}^{M+1-k}(h)\|_{L^{q'}(\rn)}\\
&&\ls \|a\|_{T_2^q(\rnz)}\ls|B|^{1/q-1}[\ro(|B|)]^{-1},
\end{eqnarray*}
which is the desired estimate, where we used the  H\"older inequality
and Lemma \ref{l2.4} by noticing that $q'\in (p_{L^\ast},2]$.
If $i\ge 3$, an argument similar to that used in the estimate of \eqref{4.2}
 also yields the desired estimate.
Thus, $\az=\pi_{L,M}(a)$ is a multiples of an $(\oz,q,M,\ez)$-molecule
adapted to $B$ with $q\in [2,\wz p_L)$, and the claim is proved.

Let $q\in (p_L,\wz p_L)$ and $\ez>n(\frac{1}{p_\oz}-\frac{1}{\wz p_\oz})$,
where $\wz p_\oz$ is as in Convention (C).
We now claim that for all $(\oz,q,M,\ez)$-molecules
$\az$ adapted to the ball $B\equiv B(x_B,r_B)$ and
 $\lz\in\cc$,
 \begin{equation}\label{4.3}
\int_\rn \oz(|\lz|\cs_L(\az)(x))\,dx\ls |B|\oz\lf(\frac{|\lz|}{|B|\ro(|B|)}\r).
 \end{equation}

Once this is proved, then we have $\|\az\|_{H_{\oz,L}(\rn)}\ls 1$, which together
with \eqref{4.1} further implies that for all $f\in T_\oz^c(\rnz)$,
$$\int_\rn \oz(\cs_L(\pi_{L,M}(f))(x))\,dx\ls
\sum_{j=1}^\fz|B_j|\oz\lf(\frac{|\lz_j|}{|B_j|\ro(|B_j|)}\r).$$
Thus, for all $f\in T_\oz^c(\rnz)$, we have
$$\|\pi_{L,M}(f)\|_{H_{\oz,L}(\rn)}\ls \Lambda(\{\lz_ja_j\}_j)\ls
\|f\|_{T_{\oz}(\rnz)},$$
which combined with a density argument implies (ii).

Now, let us prove the claim \eqref{4.3}. Observe that if $q>2$,
then an $(\oz,q,M,\ez)-$molecule is also an
$(\oz,2,M,\ez)-$molecule. Thus, to prove the claim \eqref{4.3}, it suffices to show
\eqref{4.3} for $q\in(p_L,2]$. To this end, write
\begin{eqnarray*}&&\int_{\rn}\oz(|\lz|\cs_L(\az)(x))\,dx\\
&&\hs\le \int_{\rn}\oz(|\lz|\cs_L([I-e^{-r^2_BL}]^M\az)(x))\,dx
+\int_{\rn}\oz(|\lz|\cs_L((I-[I-e^{-r^2_BL}]^M)\az)(x))\,dx\\
&&\hs\ls\sum_{j=0}^\fz \int_{\rn}\oz(|\lz|\cs_L([I-e^{-r^2_BL}]^M
(\az\chi_{U_j(B)}))(x))\,dx\\
&&\hs\hs+\sum_{j=0}^\fz \sup_{1\le k\le M}
\int_{\rn}\oz\lf(|\lz|\cs_L\lf\{\lf[\frac{k}{M}r_B^{2}Le^{-\frac{k}{M}r^2_BL}\r]^M
(\chi_{U_j(B)}(r_B^{-2}L^{-1})^{M}\az)\r\}(x)\r)\,dx\\
&&\hs\equiv\sum_{j=0}^\fz \mathrm{H}_j+\sum_{j=0}^\fz\mathrm{I}_j.
\end{eqnarray*}

For each $j\ge 0$, let $B_j\equiv 2^jB$.
Since $\oz$ is concave, by the Jensen inequality and the H\"older inequality, we obtain
\begin{eqnarray*}
\mathrm{H}_j&&\ls \sum_{k=0}^\fz \int_{U_k(B_j)}\oz(|\lz|
\cs_L([I-e^{-r^2_BL}]^M(\az\chi_{U_j(B)}))(x))\,dx\\
&&\sim \sum_{k=0}^\fz \int_{2^kB_j}\oz(|\lz|\chi_{U_k(B_j)}(x)
\cs_L([I-e^{-r^2_BL}]^M(\az\chi_{U_j(B)}))(x))\,dx\\
&&\ls \sum_{k=0}^\fz |2^kB_j|\oz\lf(\frac{|\lz|}{|2^kB_j|}\int_{U_k(B_j)}
\cs_L([I-e^{-r^2_BL}]^M(\az\chi_{U_j(B)}))(x)\,dx\r)\\
&&\ls \sum_{k=0}^\fz |2^kB_j|\oz\lf(\frac{|\lz|}{|2^kB_j|^{1/q}}
\|\cs_L([I-e^{-r^2_BL}]^M(\az\chi_{U_j(B)}))\|_{L^q(U_k(B_j))}\r).
\end{eqnarray*}
By the proof of \cite[Lemma 4.2]{hm1} (see \cite[(4.22) and (4.27)]{hm1}),
we have that
for $k=0,\,1,\,2$,
$$\|\cs_L([I-e^{-r^2_BL}]^M(\az\chi_{U_j(B)}))\|_{L^q(U_k(B_j))}
\ls \|\az\|_{L^q(U_j(B))},$$
and for $k\ge 3$,
$$\|\cs_L([I-e^{-r^2_BL}]^M(\az\chi_{U_j(B)}))\|^2_{L^q(U_k(B_j))}\ls
k\lf(\frac{1}{2^{k+j}}\r)^{4M+2(n/2-n/q)}\|\az\|_{L^q(U_j(B))}^2,$$
which, together with Definition \ref{d4.2}, $2Mp_\oz>n(1-p_\oz/2)$ and Assumption (A),
implies that
\begin{eqnarray*}
\mathrm{H}_j&&\ls |B_j|\oz\lf(\frac{|\lz|2^{-j\ez}}{|B_j|\ro(|B_j|) }\r)+
\sum_{k=3}^\fz |2^kB_j|\oz\lf(\frac{|\lz|\sqrt k 2^{-{(2M+n/2-n/q)(j+k)}-j\ez}}
{|2^kB_j|^{1/q} |B_j|^{1-1/q}\ro(|B_j|) }\r)\\
&&\ls 2^{-jp_\oz\ez}\lf\{1+\sum_{k=3}^\fz \sqrt k2^{kn(1-p_\oz/q)}
2^{-{p_\oz(2M+n/2-n/q)(j+k)}}\r\}
|B_j|\oz\lf(\frac{|\lz|}{|B_j|\ro(|B_j|) }\r)\\
&&\ls 2^{-jp_\oz\ez}
|B_j|\oz\lf(\frac{|\lz|}{|B_j|\ro(|B_j|) }\r).
\end{eqnarray*}
Since $\ro$ is of lower type $1/\wz p_\oz-1$ and $\ez>n(1/p_\oz-1/\wz p_\oz)$, we further have
\begin{eqnarray}\label{4.4}
\sum_{j=0}^\fz \mathrm{H}_j&&\ls\sum_{j=0}^\fz 2^{-jp_\oz\ez}
|B_j|\lf\{\frac{|B|\ro(|B|)}{|B_j|\ro(|B_j|)}\r\}^{p_\oz}
\oz\lf(\frac{|\lz|}{|B|\ro(|B|) }\r)\nonumber\\
&& \ls\sum_{j=0}^\fz 2^{-jp_\oz\ez}|B_j|\lf\{\frac{|B|}{|B_j|}\r\}^{p_\oz/\wz p_\oz}
\oz\lf(\frac{|\lz|}{|B|\ro(|B|) }\r)\nonumber\\
&&\ls\sum_{j=0}^\fz 2^{-jp_\oz\ez} 2^{jn(1-p_\oz/\wz p_\oz)}|B|
\oz\lf(\frac{|\lz|}{|B|\ro(|B|) }\r)
\ls |B|\oz\lf(\frac{|\lz|}{|B|\ro(|B|) }\r).
\end{eqnarray}
Similarly, we have
$$\sum_{j=0}^\fz \mathrm{I}_j\ls |B|\oz\lf(\frac{|\lz|}{|B|\ro(|B|)) }\r),$$
which completes the proof of \eqref{4.3}, and hence,
the proof of Proposition \ref{p4.1}.
\end{proof}

\begin{prop}\label{p4.2}
Let $\oz$ satisfy Assumption (A), $\ez>n(1/p_\oz-1/p_\oz^+)$
and $M> \frac n2 (\frac1{p_\oz}-\frac 12)$. If
$f\in H_{\oz,L}(\rn)\cap L^2(\rn)$, then $f\in L^p(\rn)$ for all $p\in (p_L,2]$ and
there exist $(\oz,\fz,M,\ez)$-molecules $\{\az_j\}_{j=1}^\fz$  and
numbers $\{\lz_j\}_{j=1}^\fz\subset \cc$ such that
\begin{equation}\label{4.5}
  f=\sum_{j=1}^\fz\lz_j\az_j
\end{equation}
in both $H_{\oz,L}(\rn)$ and $L^p(\rn)$ for all $p\in (p_L,2]$.
Moreover, there exists a positive constant $C$ independent of $f$ such that
for all $f\in H_{\oz,L}(\rn)\cap L^2(\rn)$,
\begin{equation}\label{4.6}
\Lambda(\{\lz_j\az_j\}_j)\equiv\inf\lf\{\lz>0:\,\sum_{j=1}^\fz
|B_j|\oz\lf(\frac{|\lz_j|}{\lz
|B_j|\ro(|B_j|)}\r)\le1\r\}\le C\|f\|_{H_{\oz,L}(\rn)},
\end{equation}
where for each $j$, $\az_j$ is adapted to the ball $B_j$.
\end{prop}

\begin{proof}[\bf Proof.]
Let $f\in H_{\oz,L}(\rn)\cap L^2(\rn)$. For each $N\in\cn$, define
$O_N\equiv\{(x,t)\in\rnz:\, |x|< N, \,1/N<t<N\}$. Then by the
$L^2(\rn)$-functional calculi for $L$, we have
\begin{equation*}
f=C_{M}\int^\fz_0 (t^2L)^{M+2}e^{-2t^2L}f\frac{\,dt}{t}=
\lim_{N\to\fz}\pi_{L,M}(\chi_{O_N}(t^2Le^{-t^2L}f))
\end{equation*}
 in $L^2(\rn),$ where $M\in\cn$, $\pi_{L,M}$ and $C_M$ are as in \eqref{1.4}.

 On the other hand, by Definition \ref{d4.1} and Lemma \ref{2.4}, we have
 $t^2Le^{-t^2L}f\in T_2^2(\rnz)\cap T_\oz(\rnz)$. An application of
 Corollary \ref{c3.1} shows that $t^2Le^{-t^2L}f\in T_2^p(\rnz)$,
 which together with Proposition \ref{p4.1} (i) implies that
$\{\pi_{L,M}(\chi_{O_N}(t^2Le^{-t^2L}f))\}_N$ is a Cauchy sequence
in $L^p(\rn)$. Then via taking subsequence, we have
\begin{equation*}
f=\lim_{N\to\fz}\pi_{L,M}(\chi_{O_N}(t^2Le^{-t^2L}f))
\end{equation*}
in $L^p(\rn)$.

Now applying Theorem \ref{t3.1} and  Proposition \ref{p3.1} to $t^2Le^{-t^2L}f$,
we obtain $(\oz,\fz)$-atoms $\{a_j\}_{j=1}^\fz$ and numbers
$\{\lz_j\}_{j=1}^\fz\subset \cc$ such that $t^2Le^{-t^2L}f=\sum_{j=1}^\fz\lz_ja_j$
in $T_2^p(\rnz)$ and
$\Lambda(\{\lz_ja_j\}_j)\ls \|t^2Le^{-t^2L}f\|_{T_{\oz}(\rnz)}$,
which combined with Proposition \ref{p4.1} (i) further yields that
\begin{equation}\label{4.7}
f=\pi_{L,M}(t^2Le^{-t^2L}f)=\sum_{j=1}^\fz\lz_j\pi_{L,M}(a_j) \equiv
\sum_{j=1}^\fz\lz_j\az_j
\end{equation} in  $L^p(\rn)$ for $p\in (p_L,2]$.
By the proof of Proposition \ref{p4.1}, we know that $\az_j$
is a multiple of an $(\oz,\fz,M,\ez)$-molecule for any $\ez>0$,
and $M\in\cn$ and $M>\frac n2 (\frac1{p_\oz}-\frac 12)$.
Notice that $\Lambda(\{\lz_j\az_j\}_j)=\Lambda(\{\lz_ja_j\}_j)$. We therefore
obtain \eqref{4.6}.

To finish the proof of Proposition \ref{p4.2}, it remains to show that
\eqref{4.5} holds in $H_{\oz,L}(\rn)$. In fact, by Lemma \ref{l2.4},
\eqref{3.5}, \eqref{4.3} and \eqref{4.7} together with the continuity
and the subadditivity of $\oz$, we have
\begin{eqnarray*}
  \int_\rn \oz\lf(\cs_L\lf(f-\sum_{j=1}^N\lz_j\az_j\r)(x)\r)\,dx&&
  \le \sum_{j=N+1}^\fz \int_\rn \oz\lf(\cs_L(\lz_j\az_j)(x)\r)\,dx\\
  &&\ls \sum_{j=N+1}^\fz |B_j|\oz\lf(\frac{|\lz_j|}{|B_j|\ro(|B_j|)}\r)\to 0,
\end{eqnarray*}
as $N\to\fz$. We point out that here, in the last inequality, to use
\eqref{4.3}, we need to choose $\wz p_\oz$ as in Convention (C)
such that $\ez>n(1/p_\oz-1/\wz p_\oz),$ which is guaranteed by the assumption
$\ez>n(1/p_\oz-1/p_\oz^+)$. This combined with an argument similar to the
proof of Proposition \ref{p3.1} yields that $f=\sum_{j=1}^\fz\lz_j\az_j$
in $H_{\oz,L}(\rn)$, which completes the proof of Proposition \ref{p4.2}.
\end{proof}

\begin{cor}\label{c4.1}
Let $\oz$ satisfy Assumption (A), $\ez>n(1/p_\oz-1/p_\oz^+)$, $q\in (p_L,\wz p_L)$
and $M> \frac n2 (\frac1{p_\oz}-\frac 12)$. Then for every
$f\in H_{\oz,L}(\rn)$, there exist $(\oz,q,M,\ez)$-molecules $\{\az_j\}_{j=1}^\fz$
and numbers $\{\lz_j\}_{j=1}^\fz\subset \cc$ such that
$  f=\sum_{j=1}^\fz\lz_j\az_j$ in $H_{\oz,L}(\rn)$. Furthermore, if letting
$\Lambda(\{\lz_j\az_j\}_j)$ be as in \eqref{4.6}, then there exists a
positive constant $C$ independent of $f$
 such that $\Lambda(\{\lz_j\az_j\}_j)\le C\|f\|_{H_{\oz,L}(\rn)}.$
\end{cor}

\begin{proof}[\bf Proof.]
If $f\in H_{\oz,L}(\rn)\cap L^2(\rn)$, then it immediately follows from
Proposition \ref{p4.2} that all results hold.

Otherwise, there exist $\{f_k\}_{k=1}^\fz\subset (H_{\oz,L}(\rn)\cap L^2(\rn))$
such that for all $k\in\cn$,
$$\|f-f_k\|_{H_{\oz,L}(\rn)}\le 2^{-k}\|f\|_{H_{\oz,L}(\rn)}.$$
 Set $f_0\equiv0.$ Then
$f=\sum_{k=1}^\fz (f_k-f_{k-1})$ in $H_{\oz,L}(\rn).$ By Proposition
\ref{p4.2}, we have that for all $k\in\cn$, $f_k-f_{k-1}=\sum_{j=1}^\fz
\lz_j^k\az_j^k$ in $H_{\oz,L}(\rn)$ and $\Lambda(\{\lz_j^k
a_j^k\}_j)\ls \|f_k-f_{k-1}\|_{H_{\oz,L}(\rn)}$, where for all $j$ and $k$,
$\az_j^k$ is an $(\oz,q,M,\ez)$-molecule.
Thus, $f=\sum_{k,j=1}^\fz\lz_j^k\az_j^k$ in $H_{\oz,L}(\rn)$,
and it further follows from Remark \ref{r3.1} (ii) that
$$[\Lambda(\{\lz_j^k \az_j^k\}_{k,j})]^{p_\oz}\le\sum_{k=1}^\fz
[\Lambda(\{\lz_j^k a_j^k\}_j)]^{p_\oz}\ls \sum_{k=1}^\fz\|f_k-
f_{k-1}\|_{H_{\oz,L}(\rn)}^{p_\oz}\ls
\|f\|_{H_{\oz,L}(\rn)}^{p_\oz},$$
which completes the proof of Corollary \ref{c4.1}.
\end{proof}

Let $H_{\oz,\rm fin}^{q,M,\ez}(\rn)$ denote the set of all finite
combinations of $(\oz,q,M,\ez)$-molecules. From Corollary
\ref{c4.1}, we immediately deduce the following density result.
\begin{cor}\label{c4.2}
Let $\oz$ satisfy Assumption (A), $\ez>n(1/p_\oz-1/p_\oz^+)$
and $M> \frac n2 (\frac1{p_\oz}-\frac 12)$. Then the space $H_{\oz,\rm
fin}^{q,M,\ez}(\rn)$ is dense in the space $H_{\oz,L}(\rn)$.
\end{cor}

\subsection{Dual spaces of $H_{\oz,L}(\rn)$\label{s4.2}}

\hskip\parindent In this subsection, we study the dual space of the
Orlicz-Hardy
 space $H_{\oz,L}(\rn)$. We begin with some notions.

Following \cite{hm1}, for $\ez>0$ and $M\in\cn$, we introduce
the space
$$\cm_\oz^{M,\ez}(L)\equiv \{\mu\in L^2(\rn):\,
\|\mu\|_{\cm_\oz^{M,\ez}(L)}<\fz\},$$
where
$$\|\mu\|_{\cm_\oz^{M,\ez}(L)}\equiv \sup_{j\ge 0}
\lf\{2^{j\ez}|B(0,2^j)|^{1/2}
\ro(|B(0,2^j)|)\sum_{k=0}^M\|L^{-k}\mu\|_{L^2(U_j(B(0,1)))}\r\}.$$

Notice that if $\phi\in \cm_\oz^{M,\ez}(L)$ with norm 1, then $\phi$ is an
$(\oz,2,M,\ez)$-molecule adapted to $B(0,1)$. Conversely, if $\az$
is an $(\oz,2,M,\ez)$-molecule adapted to certain ball, then $\az\in
\cm_\oz^{M,\ez}(L)$.

Let $A_t$ denote either $(I+t^2L)^{-1}$ or $e^{-t^2L}$ and $f\in
(\cm_\oz^{M,\ez}(L))^\ast$, the dual of $\cm_\oz^{M,\ez}(L)$. We
claim that $(I-A_t^\ast)^Mf\in L^2_{\loc}(\rn)$ in the sense of distributions.
 In fact, for any ball $B$, if $\psi\in L^2(B)$, then it follows from the Gaffney
estimates via Lemmas \ref{l2.1} and \ref{l2.2} that $(I-A_t)^M\psi\in
\cm_\oz^{M,\ez}(L)$ for all $\ez>0$ and any fixed $t\in(0,\fz)$. Thus,
\begin{eqnarray*}
  |\la (I-A_t^\ast)^Mf,\psi\ra|\equiv|\la f,(I-A_t)^M\psi\ra|\le
  C(t,r_B,\dist(B,0))\|f\|_{(\cm_\oz^{M,\ez}(L))^\ast}\|\psi\|_{L^2(B)},
\end{eqnarray*}
which implies that $(I-A_t^\ast)^Mf\in L^2_{\loc}(\rn)$  in the sense of distributions.

Finally, for any $M\in \cn$, define
$$\cm_{\oz,L^\ast}^M(\rn)\equiv \bigcap_{\ez>n(1/p_\oz-1/p_\oz^+)}
(\cm_\oz^{M,\ez}(L))^\ast.$$

\begin{defn}\label{d4.4}
Let $q\in(p_L,\wz p_L)$, $\oz$ satisfy Assumption (A),
 $\ro$ be as in \eqref{2.4} and  $M> \frac n2 (\frac1{p_\oz}-\frac 12)$.
 A functional  $f\in\cm_{\oz,L}^M(\rn)$ is said to be in $\bmoq$ if
\begin{equation*}
\|f\|_{\bmoq}\equiv\sup_{B\subset\rn}\frac{1}{\ro(|B|)}\lf[\frac{1}{|B|}\int_B
|(I-e^{-r_B^2L})^Mf(x)|^q \,dx\r]^{1/q}< \fz,
\end{equation*}
where the supremum is taken over all balls $B$ of $\rn.$
\end{defn}

In what follows, when $q=2$, we denote $\bmoq$ simply by $\bmoo$.
The proofs of following Lemmas \ref{l4.1} and \ref{l4.2} are similar
to those of Lemmas 8.1 and 8.3 of \cite{hm1}, respectively; we omit the details.

\begin{lem}\label{l4.1}
Let $\oz$, $\ro$, $q$ and $M$ be as in Definition \ref{d4.4}.
A functional $f\in\cm_{\oz,L}^M(\rn)\subset\bmoq$ if and only if
\begin{equation*}
\sup_{B\subset\rn}\frac{1}{\ro(|B|)}\lf[\frac{1}{|B|}\int_B
|(I-(I+r_B^2L)^{-1})^Mf(x)|^q \,dx\r]^{1/q}<\fz.
\end{equation*}
Moreover, the quantity appeared in the left-hand side of the above formula
is equivalent to $\|f\|_{\bmoq}$.
\end{lem}

\begin{lem}\label{l4.2}
Let $\oz$, $\ro$ and $M$ be as in Definition \ref{d4.4}. Then
there exists a positive constant $C$ such that for all $f\in \bmoo$,
$$\sup_{B\subset \rn}\frac{1}{\ro(|B|)}\lf[\frac{1}{|B|}\iint_{\widehat B}
|(t^2L)^Me^{-t^2L}f(x)|^2 \frac{\,dx\,dt}{t}\r]^{1/2}\le
C\|f\|_{\bmoo}.$$
\end{lem}

The following lemma is a slight variant of Lemma 8.4 and Remark of Section 9 in
\cite{hm1}.
\begin{lem}\label{l4.3}
Let $\oz$, $\ro$ and $M$ be as in Definition \ref{d4.4}, $q\in (p_{L^\ast},2]$,
$\ez,\,\ez_1>0$ and $\wz M>M+\ez_1+\frac n4$. Suppose
that $f\in \cm_{\oz,L^\ast}^M(\rn)$ satisfies
\begin{equation}\label{4.8}
\int_{\rn}\frac{|(I-(I+L^\ast)^{-1})^Mf(x)|^q}{1+|x|^{n+\ez_1}}\,dx<\fz.
\end{equation}
Then for every $(\oz,q',\wz M,\ez)$-molecule $\az$,
\begin{equation*}
\la f, \az\ra =\wz C_M\iint_{\rnz} (t^2L^\ast)^Me^{-t^2L^\ast}f(x)
\overline{t^2Le^{-t^2L}\az(x)}\frac{\,dx\,dt}{t},
\end{equation*}
where $q'\in [2,\fz)$ satisfying
$1/q+1/q'=1$ and $\wz C_M$ is a positive constant satisfying $\wz C_M\int_0^\fz
t^{2(M+1)}e^{-2t^2}\frac{\,dt}{t}=1.$
\end{lem}
\begin{proof}[\bf Proof.]
If we let $\oz(t)\equiv t$ for all $t\in (0,\fz)$, then this lemma are just
Lemma 8.4 and Remark of Section 9 in \cite{hm1}.

Otherwise, let $\az$ be an $(\oz,q',\wz M,\ez)$-molecule adapted to a ball $B$.
Then from Definition \ref{d4.2}, it is easy to see that
$\ro(|B|)\az$ is an $(\wz \oz,q',\wz M,\ez)$-molecule, where
$\wz \oz(t)\equiv t$ for all $t\in (0,\fz)$, and
hence Lemma \ref{l4.3} holds for $\ro(|B|)\az$, which implies the desired
conclusion and hence, completes the proof of Lemma \ref{l4.3}.
\end{proof}

From Lemma \ref{l4.1}, it is easy to follow that all $f\in \bmoq$
satisfy \eqref{4.8} for all $\ez_1\in (0,\fz)$, and hence, Lemma \ref{l4.3}
holds for all $f\in \bmoq$.

Now, let us give the main result of this section.
\begin{thm}\label{t4.1}
Let $\oz$ satisfy Assumption (A), $\ro$ be as in \eqref{2.4},
$\ez>n(1/p_\oz-1/p_\oz^+)$,
$M> \frac n2 (\frac1{p_\oz}-\frac 12)$ and $\wz M>M+\frac n4$. Then
$(H_{\oz,L}(\rn))^\ast$, the dual space of $H_{\oz,L}(\rn)$,
coincides with $\mathrm{BMO}^{M}_{\ro,L^\ast}(\rn)$
 in the following sense:

(i) Let  $g\in \mathrm{BMO}^{M}_{\ro,L^\ast}(\rn)$. Then the linear
functional $\ell$, which is initially defined on $H^{2,\wz M,\ez}
_{\oz,\rm fin}(\rn)$ by
\begin{equation}\label{4.9}
\ell(f)\equiv \la g, f \ra,
\end{equation}
has a unique extension to $H_{\oz,L}(\rn)$ with
$\|\ell\|_{(H_{\oz,L}(\rn))^\ast}\le
C\|g\|_{\mathrm{BMO}^{s}_{\ro,L^\ast}(\rn)},$ where $C$ is a positive
constant independent of $g.$

(ii) Conversely,  for any
$\ell\in (H_{\oz,L}(\rn))^\ast$, then $\ell \in
\mathrm{BMO}^{M}_{\ro,L^\ast}(\rn)$, \eqref{4.9} holds for all
$f\in H^{2,M,\ez}_{\oz,\rm fin}(\rn)$ and
$\|\ell\|_{\mathrm{BMO}^{M}_{\ro,\,L^\ast}(\rn)}\le
C\|\ell\|_{(H_{\oz,L}(\rn))^\ast}$, where $C$ is a positive
constant independent of $\ell.$
\end{thm}

\begin{proof}[\bf Proof.]
Let $g\in \mathrm{BMO}^{M}_{\ro,L^\ast}(\rn)$.
For any $f\in H^{2,\wz M,\ez}_{\oz,\rm fin}(\rn)\subset H_{\oz,L}(\rn)$,
we have that $f\in L^2(\rn)$ and hence, $t^2Le^{-t^2L}f\in
(T_{\oz}(\rnz)\cap T_2^2(\rnz))$
by Lemma \ref{l2.4}. By Theorem \ref{t3.1}, there exist
$\{\lz_j\}_{j=1}^\fz\subset \cc$ and $(\oz,\fz)$-atoms $\{a_j\}_{j=1}^\fz$
 supported in $\{\widehat B_j\}_{j=1}^\fz$ such that
\eqref{3.4} holds. Notice that $g$ satisfies \eqref{4.8} with $q=2$ (by Lemma \ref{l4.1}),
which, together with Lemmas \ref{l4.2} and \ref{l4.3}, the H\"older inequality
and Remark \ref{r3.1} (iii), yields that
\begin{eqnarray}\label{4.10}
|\la g,f\ra|&&=\lf|C_{\wz M}\iint_\rnz(t^2L^\ast)^Me^{-t^2L^\ast}g(x)
\overline{t^2Le^{-t^2L}f(x)}\frac{\,dx\,dt}{t}\r|\nonumber\\
&&\ls\sum_{j=1}^\fz|\lz_j|\iint_\rnz|(t^2L^\ast)^Me^{-t^2L^\ast}g(x)
\overline{a_j(x,t)}|\frac{\,dx\,dt}{t}\nonumber\\
&&\ls\sum_{j=1}^\fz|\lz_j|\|a_j\|_{T^2_2(\rnz)}
\lf(\iint_{\widehat B_j}|(t^2L^\ast)^Me^{-t^2L^\ast}g(x)|^2
\frac{\,dx\,dt}{t}\r)^{1/2}\nonumber\\
&&\ls\sum_{j=1}^\fz|\lz_j|\|g\|_{\bmol}
\ls \|t^2Le^{-t^2L}f\|_{T_\oz(\rnz)}\|g\|_{\bmol}\nonumber\\
&&\sim \|f\|_{H_{\oz,L}(\rn)}\|g\|_{\bmol}.
\end{eqnarray}
Then by a density argument via Corollary \ref{c4.2}, we obtain (i).

Conversely, let $\ell\in (H_{\oz,L}(\rn))^\ast$.
For any $(\oz,2,M,\ez)$-molecule $\az$,
it follows from \ref{4.3} that $\|\az\|_{H_{\oz,L}(\rn)}\ls 1$.
Thus $|\ell(\az)|\ls \|\ell\|_{(H_{\oz,L}(\rn))^\ast}$,
which implies that $\ell\in\cm_{\oz,L^\ast}^M(\rn)$.

To finish the proof of (ii), we still need to show that
$\ell\in \bmol$. To this end, for any ball $B$, let $\phi\in L^2(B)$ with
$\|\phi\|_{L^2(B)}\le \frac{1}{|B|^{1/2}\ro(|B|)}$ and
$\wz \az\equiv (I-[I+r_B^2L]^{-1})^M\phi$.
Then from Lemma \ref{l2.3}, we deduce that
for each $j\in \zz_+$ and $k =0,1,\,\cdots,M$,
\begin{eqnarray*}\|(r_B^2L)^{-k} \wz \az\|_{L^2(U_j(B))}&&=
\|(I-[I+r_B^2L]^{-1})^{M-k}(I+r_B^2L)^{-k}\phi\|_{L^2(U_j(B))}\\
&&\ls \exp\lf\{-\frac{\dist(B,U_j(B))}{cr_B}\r\}\|\phi\|_{L^2(B)}\nonumber\\
&&\ls 2^{-2j(M+\ez)}2^{jn(1/p_\oz-1/2)}\frac{1}{|2^jB|^{1/2}\ro(|2^jB|)}
\ls 2^{-2j\ez}\frac{1}{|2^jB|^{1/2}\ro(|2^jB|)},\nonumber
\end{eqnarray*}
where $c$ is as in Lemma \ref{2.3} and $2M>n(1/p_\oz-1/2)$.
Thus $\wz \az$ is a multiple of an
$(\oz,2,M,\ez)$-molecule. Since $(I-([I+t^2L]^{-1})^\ast)^M\ell$ is
well defined and belongs to $L^2_{\loc}(\rn)$ for any fixed
$t>0$. Thus, we have
\begin{eqnarray*}
|\la (I-[(I+r_B^2L)^{-1}]^\ast)^M\ell,\phi\ra|
=|\la \ell,(I-[I+r_B^2L]^{-1})^M\phi\ra|
=|\la \ell,\wz \az\ra|\ls \|\ell\|_{(H_{\oz,L}(\rn))^\ast},
\end{eqnarray*}
which further implies that
\begin{eqnarray*}
&&\frac{1}{\ro(|B|)}\lf(\frac{1}{|B|}\int_{B}|(I-[(I+r_B^2L)^{-1}]^\ast)^M\ell(x)|^2
\,dx\r)^{1/2}\\
&&\hs=\sup_{\|\phi\|_{L^2(B)}\le 1}\lf|\lf\la \ell,(I-[I+r_B^2L]^{-1})^M
\frac{\phi}{|B|^{1/2}\ro(|B|)}\r\ra\r|\ls\|\ell\|_{(H_{\oz,L}(\rn))^\ast}.
\end{eqnarray*}
Thus, $\ell\in \bmol$ and $\|\ell\|_{\bmol}\ls \|\ell\|_{(H_{\oz,L}(\rn))^\ast}$,
which completes the proof of Theorem \ref{t4.1}.
\end{proof}

\begin{rem}\label{r4.2}\rm
It follows from Theorem \ref{t4.1} that the spaces $\bmoo$
for all $M> \frac n2 (\frac1{p_\oz}-\frac 12)$ coincide with equivalent norms.
Thus, in what follows, we denote $\bmoo$ simply by $\bmo$.
\end{rem}

\section{Several equivalent characterizations of $H_{\oz,L}(\rn)$\label{s5}}

\hskip\parindent In this section, we establish several equivalent characterizations of
the Orlicz-Hardy spaces. Let us begin with some notions.

\begin{defn}\label{d5.1} Let $q\in(p_L,\wz p_L)$, $\oz$ satisfy Assumption (A),
 $M>\frac n2 (\frac1{p_\oz}-\frac 12)$ and $\ez>n(1/p_\oz-1/p_\oz^+)$. A distribution $f\in (\bmoz)^\ast$ is said to be in
the space $H_{\oz,L}^{q,M,\ez}(\rn)$
 if there exist $\{\lz_j\}_{j=1}^\fz\subset \cc$ and $(\oz,q,M,\ez)-$molecules
  $\{\az_j\}_{j=1}^\fz$ such that
$f=\sum_{j=1}^\fz\lz_j\az_j$ in $(\bmoz)^\ast$ and
$$\Lambda(\{\lz_j\az_j\}_j)=\inf\lf\{\lz>0:\, \sum_{j=1}^\fz|B_j|\oz
\lf(\frac{|\lz_j|}{\lz|B_j|\ro(|B_j|)}\r)\le1\r\}<\fz,$$
where for each $j$, $\az_j$ is adapted to the ball $B_j$.

If $f\in H_{\oz,L}^{q,M,\ez}(\rn)$, then define its norm by
$$\|f\|_{H_{\oz,L}^{q,M,\ez}(\rn)}\equiv \inf \Lambda(\{\lz_j\az_j\}_j),$$
where the infimum is taken over all possible decompositions of $f$ as above.
\end{defn}

For any $f\in L^2(\rn)$ and $x\in\rn$, define the Lusin-area
function associated to the Poisson semigroup as follows,
\begin{equation}\label{5.1}
  \cs_Pf(x)\equiv \lf(\iint_{\Gamma(x)}|t\nabla e^{-t\sqrt L}f(y)|^2
 \frac{\,dy\,dt}{t}\r)^{1/2}.
\end{equation}

Let $\bz\in(0,\fz)$. Following \cite{hm1}, we define
nontangential the maximal operators by
setting, for all $f\in L^2(\rn)$ and $x\in\rn$,
\begin{equation}\label{5.2}
\nn_h^\bz g(x)\equiv \sup_{(y,t)\in \Gamma _\bz(x)}\lf(\frac{1}{(\bz t)^n}
\int_{B(y,\bz t)}|e^{-t^2L}g(z)|^2\,dz\r)^{1/2}
\end{equation}
and
\begin{equation}\label{5.3}
\nn_P^\bz g(x)\equiv \sup_{(y,t)\in \Gamma _\bz(x)}\lf(\frac{1}{(\bz t)^n}
\int_{B(y,\bz t)}|e^{-t\sqrt L}g(z)|^2\,dz\r)^{1/2}.
\end{equation}
In what follows, we denote  $\nn_h^1$ and $\nn_P^1$ simply by
$\nn_h$ and $\nn_P$.

We also define the radial maximal functions by
setting, for all $f\in L^2(\rn)$ and $x\in\rn$,
\begin{equation}\label{5.4}
  \car_h f(x)\equiv \sup_{t>0}\lf(\frac{1}{t^n}\int_{B(x,t)}
  |e^{-t^2L}f(y)|^2\,dy\r)^{1/2}
\end{equation}
and
\begin{equation}\label{5.5}
  \car_P f(x)\equiv \sup_{t>0}\lf(\frac{1}{t^n}\int_{B(x,t)}
  |e^{-t \sqrt L}f(y)|^2\,dy\r)^{1/2}.
\end{equation}

Similarly to Definition \ref{d4.1}, we define the space
$H_{\oz,\cs_P}(\rn)$ as follows.

\begin{defn}\label{d5.2}
Let $\oz$ satisfy Assumption (A). A function $f\in L^2(\rn)$ is
said to be in $\wz H_{\oz,\,\cs_P}(\rn)$ if $\cs_P(f)\in L(\oz)$;
moreover, define
$$\|f\|_{H_{\oz,\cs_P}(\rn)}\equiv \|\cs_P(f)\|_{L(\oz)}=
\inf\lf\{\lz>0:\int_{\rn}\oz\lf
(\frac{\cs_P(f)(x)}{\lz}\r)\,dx\le 1\r\}.$$ The Orlicz-Hardy
space $H_{\oz,\cs_P}(\rn)$ is defined to be the completion of $\wz
H_{\oz,\cs_P}(\rn)$  in the norm $\|\cdot\|_{H_{\oz,\cs_P}(\rn)}.$
\end{defn}

The spaces $H_{\oz,\nn_h}(\rn)$, $H_{\oz,\nn_P}(\rn)$,
$H_{\oz,\car_h}(\rn)$ and $H_{\oz,\car_P}(\rn)$ are defined
in a similar way. We now show that all the spaces $H_{\oz,L}(\rn)$,
$H_{\oz,L}^{q,M,\ez}(\rn)$,
$H_{\oz,\cs_P}(\rn)$, $H_{\oz,\nn_h}(\rn)$, $H_{\oz,\nn_P}(\rn)$,
$H_{\oz,\car_h}(\rn)$ and $H_{\oz,\car_P}(\rn)$ coincide with equivalent norms.

\subsection{The molecular characterization\label{s5.1}}

\hskip\parindent In this subsection, we establish the molecular characterization
of the Orlicz-Hardy spaces, which gives some understanding
of the ``distributions" in $H_{\oz, L}(\rn)$
as elements of the dual of $\bmoz$. We start with the following auxiliary result.

\begin{prop}\label{p5.1} Let $\oz$ satisfy Assumption (A).
Fix $t\in (0,\fz)$ and $\wz B\equiv B(x_0,R)$ for some $x_0\in\rn$
and $R>0$. Then there exists a positive
constant $C(t,R)$ such that for all $\phi\in L^2(\rn)$ supported in $\wz B$,
$t^2L e^{-t^2L}\phi\in\bmo$ and
$$\|t^2L e^{-t^2L}\phi\|_{\bmo}\le C(t,R)\|\phi\|_{L^2(\wz B)}.$$
\end{prop}
\begin{proof}[\bf Proof.] Let $M> \frac n2 (\frac1{p_\oz}-\frac 12)$.
For any ball $B\equiv B(x_B,r_B)$, let
$$\mathrm{H}\equiv \frac{1}{\ro(|B|)}
\lf(\frac{1}{|B|}\int_{B}|(I-e^{-r_B^2L})^Mt^2L e^{-t^2L}
\phi(x)|^2\,dx\r)^{1/2}.$$

For the case when $r_B\ge R$, from the $L^2(\rn)$-boundedness of the
operators $e^{-r_B^2L}$ and $t^2L e^{-t^2L}$ (Lemma \ref{l2.2}), it follows that
$$\mathrm{H}\ls\frac{1}{|\wz B|^{1/2}\ro(|\wz B|)}\|\phi\|_{L^2(\wz B)}.$$

Let us consider the case when $r_B<R$. It follows from the upper type property
that
\begin{equation}\label{5.6}
|\wz B|^{1/2}\ro(|\wz B|)\ls \lf(\frac{R}{r_B}\r)^{n(1/p_\oz-1/2)}|B|^{1/2}\ro(|B|).
\end{equation}
On the other hand, noticing that $I-e^{-r^2_BL}=\int_0^{r_B^2}Le^{-rL}\,dr$, thus by
the Minkowski inequality and the $L^2(\rn)$-boundedness of $t^2Le^{-t^2L}$
(Lemma \ref{l2.2}), we have
\begin{eqnarray}\label{5.7}
&& \lf(\int_{B}|(I-e^{-r_B^2L})^Mt^2L e^{-t^2L}\phi(x)|^2\,dx\r)^{1/2}\nonumber\\
&&\hs=\lf(\int_{B}\lf|\int_0^{r_B^2}\cdots\int_0^{r_B^2} t^2L^{M+1}
e^{-(r_1+\cdots+r_M+t^2)L}\phi(x)\,dr_1\cdots\,dr_M\r|^2\,dx\r)^{1/2}\nonumber\\
&&\hs\ls\int_0^{r_B^2}\cdots\int_0^{r_B^2}\frac{t^2}{(r_1+\cdots+r_M+t^2)^{M+1}}
\|\phi\|_{L^2(\wz B)}\,dr_1\cdots\,dr_s
\ls \lf(\frac{r_B}{t}\r)^{2M}\|\phi\|_{L^2(\wz B)}.\quad\quad
\end{eqnarray}
By the fact that $M>\frac n2(\frac{1}{p_\oz}-\frac 12)$ and the estimates
\eqref{5.6} and \eqref{5.7},
 we obtain
$$\mathrm{H}\ls \lf(\frac{R}{t}\r)^{2M}\frac{\|\phi\|_{L^2(\wz B)}}
{|\wz B|^{1/2}\ro(|\wz B|)}.$$

Thus $\|t^2L e^{-t^2L}\phi\|_{\bmo}\le C(t,R)\|\phi\|_{L^2(\wz B)}$, which
completes the proof of Proposition \ref{p5.1}.
\end{proof}

\begin{thm}\label{t5.1}
Let $q\in(p_L,\wz p_L)$, $\oz$ satisfy Assumption (A),
$M> \frac n2 (\frac1{p_\oz}-\frac 12)$  and $\ez>n(1/p_\oz-1/p_\oz^+)$.
Then the spaces $H_{\oz,L}(\rn)$ and $H^{q,M,\ez}_{\oz,L}(\rn)$
coincide with equivalent norms.
\end{thm}

\begin{proof}[\bf Proof.]
By Corollary \ref{c4.1}, for all
 $f\in H_{\oz,L}(\rn)$, there exist $(\oz,q,M,\ez)$-molecules
 $\{\az_j\}_{j=1}^\fz$ adapted to balls $\{B_j\}_{j=1}^\fz$ and
numbers $\{\lz_j\}_{j=1}^\fz\subset \cc$ such that
$f=\sum_{j=1}^\fz\lz_j\az_j$ in $H_{\oz,L}(\rn)$ and
$\Lambda(\{\lz_j\az_j\}_j)\ls\|f\|_{H_{\oz,L}(\rn)}.$
Then Theorem \ref{t4.1} implies that
the decomposition also holds in $(\bmoz)^\ast$, and hence,
$H_{\oz,L}(\rn)\subset H^{q,M,\ez}_{\oz,L}(\rn)$.

Conversely, let $f\in H^{q,M,\ez}_{\oz,L}(\rn)$.
Then there exist $\{\lz_j\}_{j=1}^\fz\subset\cc$ and
$(\oz,q,M,\ez)$-molecules $\{\az_j\}_{j=1}^\fz$ such that
$f=\sum_{j=1}^\fz\lz_j\az_j$ in $(\bmoz)^\ast$ and
$$\Lambda(\{\lz_j\az_j\}_j)=\inf\lf\{\lz>0:\,\sum_{j=1}^\fz |B_j|\oz\lf(\frac{|\lz_j|}
{\lz |B_j|\ro(|B_j|)}\r)\le1\r\}<\fz,$$
where for each $j$, $\az_j$ is adapted to the ball $B_j$.

 For all $x\in \rn$, by Proposition \ref{p5.1}, we have
 \begin{eqnarray*}
   \cs_L f(x)&&
   =\lf\{\int_0^\fz \lf\|t^2Le^{-t^2L}(f)\r\|_{L^2(B(x,t))}^2
   \frac{\,dt}{t^{n+1}}\r\}^{1/2}\\
   &&=\lf\{\int_0^\fz \lf(\sup_{\|\phi\|_{L^2(B(x,t))}\le 1}
   \lf|\lf\la \sum_{j=1}^\fz\lz_j\az_j, t^2L^\ast e^{-t^2L^\ast}
   \phi\r\ra\r|\r)^2
   \frac{\,dt}{t^{n+1}}\r\}^{1/2}\\
   &&\le \sum_{j=1}^\fz\lf\{\int_0^\fz \lf(\sup_{\|\phi\|_{L^2(B(x,t))}\le 1}
   \lf|\lf\la t^2Le^{-t^2L}\lz_j\az_j, \phi\r\ra\r|\r)^2
   \frac{\,dt}{t^{n+1}}\r\}^{1/2}
   \le \sum_{j=1}^\fz\cs_L(\lz_j\az_j)(x).
 \end{eqnarray*}
 Then from \eqref{4.3} together with the continuity and the subadditivity of $\oz$,
  it follows that
\begin{eqnarray*}
  \int_{\rn}\oz\lf(\cs_L f(x)\r)\,dx&&\le
  \sum_{j=1}^\fz\int_{\rn}\oz\lf(\cs_L(\lz_j\az_j)(x)\r)\,dx
  \ls \sum_{j=1}^\fz |B_j|\oz\lf(\frac{|\lz_j|}{|B_j|\ro(|B_j|)}\r),
\end{eqnarray*}
which implies that  $\|f\|_{H_{\oz,L}(\rn)}\ls \Lambda(\{\lz_j\az_j\}_j)$.
By taking the infimum over all decompositions of $f$ as above, we obtain that
$\|f\|_{H_{\oz,L}(\rn)}\ls \|f\|_{H^{q,M,\ez}_{\oz,L}(\rn)}$,
which completes the proof of Theorem \ref{t5.1}.
\end{proof}

\subsection{Characterizations by the maximal functions \label{s5.2}}

\hskip\parindent In this subsection, we characterize the Orlicz-Hardy space
 via the Lusin-area function $\cs_P$ and the maximal functions $\nn_h$, $\nn_P$,
$\car_h$ and $\car_P$. Let us begin with the following very useful
auxiliary result on the boundedness of
 linear or non-negative sublinear operators from $H_{\oz,L}(\rn)$ to $L(\oz)$.

\begin{lem}\label{l5.1}
Let $q\in(p_L,2]$, $\oz$ satisfy Assumption (A),
$M> \frac n2 (\frac1{p_\oz}-\frac 12)$  and $\ez>n(1/p_\oz-1/p_\oz^+)$.
Suppose that $T$ is a non-negative sublinear (resp. linear) operator
which maps $L^q(\rn)$ continuously into weak-$L^q(\rn)$.  If there exists a
positive constant $C$ such that for all $(\oz,\fz,M,\ez)$-molecules
$\az$ adapted to balls $B$ and $\lz\in \cc$,
\begin{equation}\label{5.8}
\int_{\rn}\oz\lf(T(\lz\az)(x)\r)\,dx\le C|B|\oz\lf(\frac{|\lz|}{|B|\ro(|B|)}\r),
\end{equation}
then $T$ extends to a bounded sublinear (resp. linear) operator
from $H_{\oz,L}(\rn)$ to $L(\oz)$; moreover, there exists a positive constant $\wz C$
such that for all $f\in H_{\oz,L}(\rn)$, $\|Tf\|_{L(\oz)}
\le \wz C \|f\|_{H_{\oz,L}(\rn)}$.
\end{lem}

\begin{proof}[\bf Proof.]
It follows from Proposition \ref{p4.2} that for every
$f\in H_{\oz,L}(\rn)\cap L^2(\rn)$, $f\in L^q(\rn)$ with $q\in (p_L,2]$
and there exists $\{\lz_j\}_{j=1}^\fz\subset\cc$ and
$(\oz,\fz,M,\ez)$-molecules $\{\az_j\}_{j=1}^\fz$ such that
$f=\sum_{j=1}^\fz\lz_j\az_j$ in both $H_{\oz,L}(\rn)$ and $L^q(\rn)$;
moreover, $\Lambda(\{\lz_j\az_j\}_j)\ls \|f\|_{H_{\oz,L}(\rn)}$. Thus
if $T$ is linear, then it follows from the fact that  $T$ is of
weak type $(q,q)$ that $T(f)=\sum_{j=1}^\fz T(\lz_j\az_j)$
almost everywhere.

If $T$ is a non-negative sublinear operator, then
$$\sup_{t>0} t^{1/q}\lf|\lf\{x\in\rn:\, \lf|T(f)(x)-
T\lf(\sum_{j=1}^N\lz_j\az_j\r)(x)\r|>t\r\}\r|\ls
\lf\|f-\sum_{j=1}^N\lz_j\az_j\r\|_{L^q(\rn)}
\to 0,$$
as $N\to \fz$. Thus there exists a subsequence  $\{N_k\}_{k}\subset\cn$
such that
$$T\lf(\sum_{j=1}^{N_k}\lz_{j}\az_{j}\r)\to T(f)$$
almost everywhere, as $k\to\fz$, which together with the non-negativity
and the sublinearity of $T$ further implies that
\begin{eqnarray*}
T(f)-\sum_{j=1}^\fz T(\lz_j\az_j)&&=T(f)-T\lf(\sum_{j=1}^{N_k}\lz_{j}\az_{j}\r)
+T\lf(\sum_{j=1}^{N_k}\lz_{j}\az_{j}\r)-\sum_{j=1}^\fz T(\lz_j\az_j)\\
&& \le T(f)-T\lf(\sum_{j=1}^{N_k}\lz_{j}\az_{j}\r).
\end{eqnarray*}
By letting $k\to \fz$, we see that
$T(f)\le \sum_{j=1}^\fz T(\lz_j\az_j)$ almost everywhere.
Thus by the subadditivity and the continuity of $\oz$ and \eqref{5.8},
we finally obtain
\begin{eqnarray*}
\int_{\rn}\oz\lf(T(f)(x)\r)\,dx\ls \sum_{j=1}^\fz\int_{\rn}
\oz\lf(T(\lz_j\az_j)(x)\r)\,dx
\ls \sum_{j=1}^\fz|B_j|\oz\lf(\frac{|\lz_j|}{|B_j|\ro(|B_j|)}\r),
\end{eqnarray*}
which implies that $\|T(f)\|_{L(\oz)}\ls \Lambda(\{\lz_j\az_j\}_j)
\ls \|f\|_{H_{\oz,L}(\rn)}$.
 This, combined with the density of $H_{\oz,L}(\rn)\cap L^2(\rn)$ in
 $H_{\oz,L}(\rn)$, then finishes the proof of Lemma \ref{l5.1}.
\end{proof}

\begin{rem}\rm\label{r5.1} Let $p\in (0,1]$. We point out that the condition
\eqref{5.8} is also necessary, if $\oz(t)\equiv t^p$ for all $t\in(0,\fz)$.
However, for a general $\oz$ as in Lemma \ref{l5.1}, it is still unclear
whether \eqref{5.8} is necessary or not.
\end{rem}

\begin{thm}\label{t5.2}
 Let $\oz$ satisfy Assumption (A). Then the spaces $H_{\oz,L}(\rn)$,
 $H_{\oz,\cs_P}(\rn)$,
 $H_{\oz,\nn_h}(\rn)$ and $H_{\oz,\nn_P}(\rn)$ coincide with equivalent norms.
\end{thm}

Before we prove Theorem \ref{t5.2}, we recall some auxiliary operators
introduced in \cite{hm1}. Let $\bz\in(0,\fz)$.
For any $g\in L^2(\rn)$ and $x\in\rn$, let
\begin{equation*}
  \wz \cs_P^\bz g(x)\equiv \lf(\iint_{\Gamma_\bz (x)}|t^2L e^{-t\sqrt L}g(y)|^2
 \frac{\,dy\,dt}{t^{n+1}}\r)^{1/2}
  \end{equation*}
and
\begin{equation*}
\wz \cs_h^\bz g(x)\equiv \lf(\iint_{\Gamma_\bz(x)}|t\nabla e^{-t^2L}g(y)|^2
\frac{\,dy\,dt}{t^{n+1}}\r)^{1/2}.
\end{equation*}
We denote $\wz \cs_P^1 g$ and $\wz \cs_h^1 g$ simply by
$\wz \cs_P g$ and $\wz \cs_h g$, respectively.

The proof of the following lemma is
similar to that of \cite[Lemma 5.4]{hm1}. We omit the details.

\begin{lem}\label{l5.2}
There exists a positive constant $C$ such that for all $g\in L^2(\rn)$ and $x\in\rn$,
\begin{equation}\label{5.9}
\wz \cs_Pg(x)\le C \cs_Pg(x)
\end{equation}
and $\cs_Lg(x)\le C \wz \cs_h g(x).$
\end{lem}

\begin{proof}[\bf Equivalence of $H_{\oz,L}(\rn)$ and $H_{\oz,\cs_P}(\rn)$.]
Let $\ez>n(\frac{1}{p_\oz}-\frac{1}{p_\oz^+})$ and
$M> \frac n2 (\frac1{p_\oz}-\frac 12)$.
Suppose that $f\in H_{\oz,\cs_P}(\rn)\cap L^2(\rn)$. It follows from
\eqref{5.9} that
$\|\wz \cs_Pf\|_{L(\oz)}\ls \|f\|_{H_{\oz,\cs_P}(\rn)}.$
Moreover, since $\cs_P$ is bounded on $L^2(\rn)$ (see (5.15) in \cite{hm1}),
by \eqref{5.9}, we have
$$\|\wz \cs_Pf\|_{L^2(\rn)}\ls \|\cs_Pf\|_{L^2(\rn)}\ls \|f\|_{L^2(\rn)}.$$
Thus we obtain $t^2L e^{-t\sqrt L}f\in (T_\oz(\rnz)\cap T^2_2(\rnz))$.
Let $\wz C$ be a positive constant such that
$\wz C\int_0^\fz t^{2(s+1)}e^{-t^2}t^2e^{-t}\frac{\,dt}{t}=1.$
Then by the $L^2(\rn)$-functional calculi, we have
$$f=\frac{\wz C}{C_M}\pi_{L,M}(t^2L e^{-t\sqrt L}f)$$
in $L^2(\rn)$, where $C_M$ is the same as in \eqref{1.4}.

Since $t^2L e^{-t\sqrt L}f\in T_\oz(\rnz)$, by Proposition \ref{p4.1},
we obtain that $f\in H_{\oz,L}(\rn)$ and
$$\|f\|_{H_{\oz,L}(\rn)}\ls \|t^2L e^{-t\sqrt L}f\|_{T_\oz(\rnz)}\sim
\|\wz \cs_Pf\|_{L(\oz)}\ls \|f\|_{H_{\oz,\cs_P}(\rn)}.$$
Then a density argument yields that $H_{\oz,\cs_P}(\rn)\subset H_{\oz,L}(\rn)$.

Conversely, similarly to the proof of \eqref{4.3},
by using the estimates in the proof of \cite[Theorem 5.3]{hm1}, we have
\begin{eqnarray*}
  \int_{\rn}\oz\lf(|\lz| \cs_P(\az)(x)\r)\,dx&&\ls
  |B|\oz\lf(\frac{|\lz|}{|B|\ro(|B|)}\r),
\end{eqnarray*}
where $\az$ is an $(\oz,2,M,\ez)$-molecule adapted to the ball $B$
and $\lz\in\cc$. By the $L^2(\rn)$-boundedness of $\cs_P$
and  Lemma \ref{l5.1}, we have
$\|f\|_{H_{\oz,\cs_P}(\rn)}=\|\cs_P f\|_{L(\oz)}\ls \|f\|_{H_{\oz,L}(\rn)}$,
which implies that $H_{\oz,L}(\rn)\subset H_{\oz,\cs_P}(\rn)$. Thus, $H_{\oz,L}(\rn)$
and $H_{\oz,\cs_P}(\rn)$ coincide with equivalent norms.
 \end{proof}

In what follows, the operators $\nn_h^\bz$ and $\nn_P^\bz$ are as in \eqref{5.2}
and \eqref{5.3}, respectively.
\begin{lem}\label{l5.3}
  Let $0<\bz<\gz<\fz$. Then there exists a positive constant $C$,
depending on $\bz$ and $\gz$,  such that
  for all $g\in L^2(\rn)$,
\begin{equation}\label{5.10}
     C^{-1}\|\nn_h^\bz g\|_{L(\oz)}\le \|\nn_h^\gz g\|_{L(\oz)}
     \le C\|\nn_h^\bz g\|_{L(\oz)}
\end{equation}
  and
\begin{equation}\label{5.11}
   C^{-1}\|\nn_P^\bz g\|_{L(\oz)}\le \|\nn_P^\gz g\|_{L(\oz)}
   \le C\|\nn_P^\bz g\|_{L(\oz)}.
\end{equation}
\end{lem}

\begin{proof}[\bf Proof.] We only prove \eqref{5.10}; the proof of \eqref{5.11}
is similar.

Since $\bz < \gz$, for any $x\in\rn$, it is easy to see
that $\nn_h^\bz g(x)\le (\frac{\gz}{\bz})^n \nn_h^\gz g(x)$, which implies the first
inequality.

To show the second inequality in \eqref{5.10}, without loss of generality,
we may assume that $\|\nn_h^\bz g\|_{L(\oz)}<\fz$.
Let $\sz\in(0,\fz)$,
\begin{equation}\label{5.12}
E_\sz\equiv \{x\in\rn:\,\nn_h^\bz g(x)>\sz\} \ \mbox{and} \
E_\sz^\ast\equiv \lf\{x\in\rn:\,\cm(\chi_{E_\sz})(x)>\lf(\frac{\bz}{3\gz}\r)^n\r\}.
\end{equation}
Suppose that $x\notin E_\sz^\ast$. Thus for any $(y,t)\in \Gamma_{2\gz}(x)$,
we have $B(y,\bz t)\nsubseteq E_\sz$; otherwise,
$$\cm(\chi_{E_\sz})(x)>\frac{|B(y,\bz t)|}{|B(x,3\gz t)|}=\lf(\frac{\bz}{3\gz}\r)^n,$$
which contradicts with $x\notin E_\sz^\ast$.
Thus there exists $z\in (B(y,\bz t)\cap (E_\sz)^\com)$, which further implies that
\begin{equation}\label{5.13}
\lf(\frac{1}{(\bz t)^n}
\int_{B(y,\bz t)}|e^{-t^2L}g(u)|^2\,du\r)^{1/2}\le \nn_h^\bz g(z)\le \sz.
\end{equation}
For every $(w,t)\in \Gamma_{\gz}(x)$, we cover the ball $B(w,\gz t)$ by no
more that $N(n,\bz,\gz)$ balls $\{B(y_i,\bz t)\}_{i=1}^{N(n,\bz,\gz)}$,
where $(y_i,t)\in \Gamma _{2\gz}(x)$ and
$N(n,\bz,\gz)$ depends only on $n,\,\bz,\,\gz$. Thus, by \eqref{5.13}, we obtain
\begin{eqnarray*}
&&\lf(\frac{1}{(\gz t)^n}\int_{B(w,\gz t)}|e^{-t^2L}g(z)|^2\,dz\r)^{1/2}\\
&&\hs\hs \le \lf(\frac{\bz}{\gz}\r)^{n/2}\sum_{i=1}^{N(n,\bz,\gz)}
\lf(\frac{1}{(\bz t)^n}
\int_{B(y_i,\bz t)}|e^{-t^2L}g(z)|^2\,dz\r)^{1/2}\le C(n,\bz,\gz) \sz,
\end{eqnarray*}
where $C(n,\bz,\gz)$ is a positive constant depending on $n,\,\bz,\,\gz$.
From this, it follows that for all $\sz>0$,
$\{x\in\rn:\,\nn_h^\gz g(x)>C(n,\bz,\gz)\sz\}\subset E_\sz^\ast$. This combined \eqref{3.2}
yields that
\begin{eqnarray*}
  \int_{\rn}\oz(\nn_h^\gz g(x))\,dx&&\sim\int_{\rn}\int_0^{\nn_h^\gz g(x)}
  \frac{\oz(t)}{t}\,dt\,dx
\sim \int_0^\fz \frac{\oz(t)}{t}|\{x\in\rn:\,\nn_h^\gz g(x)>t\}|\,dt\\
  &&\sim \int_0^\fz \frac{\oz(t)}{t}|\{x\in\rn:\,\nn_h^\gz g(x)>C(n,\bz,\gz)t\}|\,dt\\
  &&\ls\int_0^\fz \frac{\oz(t)}{t}|E_t^\ast|\,dt\ls
   \int_0^\fz \frac{\oz(t)}{t}|E_t|\,dt\sim \int_{\rn}\oz(\nn_h^\bz g(x))\,dx,
\end{eqnarray*}
which further implies that $\|\nn_h^\gz g\|_{L(\oz)}\ls \|\nn_h^\bz g\|_{L(\oz)}$,
and hence, completes the proof of Lemma \ref{l5.3}.
\end{proof}

\begin{proof}[\bf Equivalence of $H_{\oz,L}(\rn)$ and $H_{\oz,\nn_h}(\rn)$.]
By \eqref{3.2} and Lemmas \ref{l5.2} and \ref{l3.2}, we have
\begin{equation}\label{5.14}
\|\cs_L f\|_{L(\oz)}\ls \|\wz \cs_h f\|_{L(\oz)}\ls \|\wz \cs_h^{1/2} f\|_{L(\oz)}.
\end{equation}
Recall that  $\sz_{g}$ denote the distribution function of a function $g$.
The estimate (6.36) of \cite{hm1} says that for any $\lz\in(0,\fz)$,
\begin{equation}\label{5.15}
  \sz_{\wz \cs_h^{1/2} f}(\lz)\le \frac{1}{\lz^2}\int_0^\lz t
  \sz_{\nn_h^\bz f}(t)\,dt+\sz_{\nn_h^\bz f}(\lz),
\end{equation}
where $\bz\in(0,\fz)$ is large enough.

Since $\oz$ is of upper type 1, by \eqref{5.14}, \eqref{3.2}, \eqref{5.15}
and Lemma \ref{l5.3}, we obtain that
\begin{eqnarray*}
  \int_\rn \oz(\cs_L f(x))\,dx&&\ls \int_\rn \oz(\wz \cs_h^{1/2} f(x))\,dx
  \sim \int_\rn \int_0^{\wz \cs_h^{1/2} f(x)}\frac{\oz(u)}{u}\,du\,dx\\
  &&\sim \int_0^\fz \frac{\oz(u)}{u}\sz_{\wz \cs_h^{1/2} f}(u)\,du\\
  &&\ls \int_0^\fz \frac{\oz(u)}{u}\lf[\frac{1}{u^2}\int_0^u t
  \sz_{\nn_h^\bz f}(t)\,dt+\sz_{\nn_h^\bz f}(u)\r]\,du\\
  &&\ls \int_0^\fz t \sz_{\nn_h^\bz f}(t)\int_t^\fz \frac{\oz(t)}{u^2t}\,du\,dt
  +\int_{\rn}\oz(\nn_h^\bz f(x))\,dx\\
  &&\ls\int_{\rn}\oz(\nn_h^\bz f(x))\,dx \ls\int_{\rn}\oz(\nn_h f(x))\,dx,
\end{eqnarray*}
which implies that $\|f\|_{H_{\oz,L}(\rn)}\ls \|f\|_{H_{\oz,\nn_h}(\rn)}$, and
hence, $H_{\oz,\nn_h}(\rn)\subset H_{\oz,L}(\rn)$.

Conversely,  let $\car_h$ be as in \eqref{5.4}. For all $g\in L^2(\rn)$
and $x\in \rn$, we also define
\begin{equation*}
  \car_h^{M} g(x)\equiv \sup_{t>0}\lf(\frac{1}{t^n}\int_{B(x,t)}
  |(t^2L)^Me^{-t^2L}g(y)|^2\,dy\r)^{1/2}.
\end{equation*}
By the proof of \cite[Theorem 6.3]{hm1}, we know that
the operators $\car_h$ and $\car_h^{M}$ are bounded on $L^2(\rn)$.

Since Lemma \ref{l5.3} implies that for all $f\in L^2(\rn)\cap H_{\oz,L}(\rn)$,
$$\|\nn_{h}f\|_{L(\oz)}\ls \|\nn_{h}^{1/2}f\|_{L(\oz)}\ls \|\car_h f\|_{L(\oz)},$$
by Lemma \ref{l5.1} and a density argument, to show
$H_{\oz,L}(\rn)\subset H_{\oz,\nn_h}(\rn)$, it suffices to prove
that for all $(\oz,2,M,\ez)$-molecules $\az$ adapted to balls $B$ and $\lz\in \cc$,
 \begin{equation}\label{5.16}
\int_{\rn}\oz\lf(\car_h(\lz\az)(x)\r)\,dx\ls|B|\oz\lf(\frac{|\lz|}{|B|\ro(|B|)}\r).
 \end{equation}

Since $\oz$ is concave, by the Jensen inequality and the H\"older inequality, we obtain
\begin{eqnarray*}
\int_{\rn}\oz\lf(\car_h(\lz\az)(x)\r)\,dx&&\le
\sum_{j=0}^\fz \int_{U_j(B)}\oz\lf(\car_h(\lz\az)(x)\r)\,dx\\
&&\ls \sum_{j=0}^\fz |2^jB|\oz\lf(\frac{\|\car_h(\lz\az)
\|_{L^2(U_j(B))}}{|2^jB|^{1/2}}\r).
\end{eqnarray*}
For $j\in\zz_+$ and $j\le 10$, by the $L^2(\rn)$-boundedness of
the operator $\car_h$ and the definition of the molecule, we have
\begin{equation*}\sum_{j=0}^{10} |2^jB|\oz\lf(\frac{\|\car_h
(\lz\az)\|_{L^2(U_j(B))}}
{|2^jB|^{1/2}}\r)\le |B|\oz\lf(\frac{|\lz|}{|B|\ro(|B|)}\r).
\end{equation*}

Since $M>\frac n2(\frac 1{p_\oz}-\frac 12)$, we let $a\in (0,1)$
such that $a(2M+n/2)>n/p_\oz$. For $j\in\cn$ and $j>10$, write
\begin{eqnarray*}
\car_h(\lz\az)(x)&& \le \sup_{t\le 2^{aj-2}r_B}\lf(\frac{1}{t^n}\int_{B(x,t)}
  |e^{-t^2L}(\lz\az)(y)|^2\,dy\r)^{1/2}\\
  &&\hs +\sup_{t> 2^{aj-2}r_B}\lf(\frac{1}{t^n}\int_{B(x,t)}
  |e^{-t^2L}(\lz\az)(y)|^2\,dy\r)^{1/2}\equiv \mathrm{H}_j+\mathrm{I}_j.
\end{eqnarray*}

 For the case  $t\le 2^{aj-2}r_B$, let
\begin{equation}\label{5.17}
   V_j(B)\equiv 2^{j+3}B\setminus 2^{j-3}B, \  R_j(B)
   \equiv 2^{j+5}B\setminus 2^{j-5}B \
   \mbox{and} \ E_j(B)\equiv (R_j(B))^\com.
\end{equation}
If $x\in U_j(B)$ and $|x-y|<t$, then we have $y\in V_j(B)$ and $\dist (V_j(B),
E_j(B))\sim 2^jr_B$,
which together with Lemma \ref{l2.3} yields that
\begin{eqnarray*}
\|\mathrm{H}_j\|_{L^2(U_j(B))}&&\le \lf\|\sup_{t\le 2^{aj-2}r_B}
\lf(\frac{1}{t^n}\int_{B(\cdot,t)}
  |e^{-t^2L}(\lz\az\chi_{R_j(B)})(y)|^2\,dy\r)^{1/2}\r\|_{L^2(U_j(B)}\\
  &&\hs\hs+\lf\|\sup_{t\le 2^{aj-2}r_B}\lf(\frac{1}{t^n}\int_{B(\cdot,t)}
  |e^{-t^2L}(\lz\az\chi_{E_j(B)})(y)|^2\,dy\r)^{1/2}\r\|_{L^2(U_j(B)}\nonumber\\
  &&\ls \|\car_h(\lz\az\chi_{R_j(B)})\|_{L^2(\rn)}+
  |U_j(B)|^{1/2}\sup_{t\le 2^{aj-2}r_B}t^{-n/2}e^{-\frac{(2^jr_B)^2}
{ct^2}}\|\lz\az\|_{L^2(E_j(B))}  \nonumber\\
  &&\ls\|\lz\az\|_{L^2(R_j(B))}+|U_j(B)|^{1/2}\sup_{t\le 2^{aj-2}r_B}
  t^{-n/2}\lf(\frac{t}{2^jr_B}\r)^{N}\|\lz\az\|_{L^2(\rn)}\nonumber\\
  &&\ls |\lz|2^{-j\ez}[\ro(|2^jB|)]^{-1}|2^jB|^{-1/2}+|\lz|2^{j(1-a)(n/2-N)}
  [\ro(|B|)]^{-1}|B|^{-1/2},\nonumber
\end{eqnarray*}
where $c$ is a positive constant as in Lemma \ref{l2.3} and
$N\in\cn$ is large enough such that $(1-a)(N-n/2)p_\oz>n(1-p_\oz/2)$.
Then by an argument similar to the proof of \eqref{4.4} and
the fact that $\oz$ is of lower type $p_\oz$, we have
\begin{eqnarray}\label{5.18}
&& \sum_{j=11}^\fz|2^jB|\oz\lf(\frac{\|\mathrm{H}_j\|_{L^2(U_j(B))}}{|2^jB|^{1/2}}\r)
\nonumber\\
&&\hs\ls \sum_{j=11}^\fz |2^jB|\oz\lf(\frac{|\lz|2^{-j\ez}}{|2^jB|\ro(|2^jB|)}\r)+
\sum_{j=11}^\fz |2^jB|\oz\lf(\frac{|\lz|2^{j(1-a)(n/2-N)}
  }{|2^jB|^{1/2}\ro(|B|)|B|^{1/2}}\r)\nonumber\\
  &&\hs\ls |B|\oz\lf(\frac{|\lz|}{|B|\ro(|B|)}\r)+
\sum_{j=11}^\fz |B|2^{jn(1-p_\oz/2)}2^{j(1-a)(n/2-N)p_\oz}
  \oz\lf(\frac{|\lz| }{\ro(|B|)|B|}\r)\nonumber\\
&&\hs\ls |B|\oz\lf(\frac{|\lz|}{|B|\ro(|B|)}\r),
\end{eqnarray}
which is a desired estimate.

For the term $\mathrm{I}_j$, by the $L^2(\rn)$-boundedness of the
operator $\car_h^M$, we have
\begin{eqnarray*}
&&\|\mathrm{I}_j\|_{L^2(U_j(B))}\\
&&\hs\ls \lf\|\sup_{t> 2^{aj-2}r_B}\lf(\frac{1}{t^n}\int_{B(x,t)}
  \lf(\frac{r_B}{t}\r)^{4M}|(t^2L)^Me^{-t^2L}(\lz(r_B^2L)^{-M}\az)(y)
  |^2\,dy\r)^{1/2}\r\|_{L^2(U_j(B))}\nonumber\\
  &&\hs\ls 2^{-2aMj}\|\car_h^M(\lz(r_B^2L)^{-M}\az)\|_{L^2(U_j(B))}
  \ls |\lz|2^{-2aMj}[\ro(|B|)]^{-1}|B|^{-1/2},\nonumber
  \end{eqnarray*}
  which together with the fact that $ap_\oz(2M+n/2)>n$ implies that
\begin{eqnarray}\label{5.19}
 \sum_{j=11}^\fz|2^jB|\oz\lf(\frac{\|\mathrm{I}_j\|_{L^2(U_j(B))}}{|2^jB|^{1/2}}\r)
&&\ls \sum_{j=11}^\fz |2^jB|\oz\lf(\frac{|\lz|2^{-2aMj}}
{|2^jB|^{1/2}\ro(|B|)|B|^{1/2}}\r)\nonumber\\
&&\ls \sum_{j=11}^\fz 2^{jn(1-p_\oz/2)}2^{-2aMjp_\oz}|B|\oz
\lf(\frac{|\lz|}{|B|\ro(|B|)}\r)\nonumber\\
&&\ls|B|\oz\lf(\frac{|\lz|}{|B|\ro(|B|)}\r),
\end{eqnarray}
which is a desired estimate.

Combining the estimates \eqref{5.18} and \eqref{5.19} yields \eqref{5.16},
and hence, completes the proof of that $H_{\oz,L}(\rn)\subset H_{\oz,\nn_h}(\rn)$.
Therefore, $H_{\oz,L}(\rn)$ and $H_{\oz,\nn_h}(\rn)$ coincide with equivalent norms.
\end{proof}

\begin{proof}[\bf Equivalence of $H_{\oz,L}(\rn)$ and $H_{\oz,\nn_P}(\rn)$.]
The proof of the equivalence of $H_{\oz,L}(\rn)$ and $H_{\oz,\nn_P}(\rn)$
is similar to that of the equivalence of $H_{\oz,L}(\rn)$ and $H_{\oz,\nn_h}(\rn)$;
we omit the details.

This finishes the proof of Theorem \ref{t5.2}.
\end{proof}

From Theorem \ref{t5.2}, it is easy to deduce the following
radial maximal function characterizations of $H_{\oz, L}(\rn)$.
Recall that $\car_h$ and $\car_P$ are defined in \eqref{5.4}
and \eqref{5.5}, respectively.

\begin{cor}\label{c5.1}
Let $\oz$ satisfy Assumption (A). Then the spaces $H_{\oz,L}(\rn)$,
  $H_{\oz,\car_h}(\rn)$ and $H_{\oz,\car_P}(\rn)$ coincide with equivalent norms.
\end{cor}
\begin{proof}[\bf Proof.]
We only give the proof of the equivalence between $H_{\oz,\car_h}(\rn)$
and $H_{\oz,L}(\rn)$, since the proof of the equivalence
between $H_{\oz,\car_P}(\rn)$ and $H_{\oz,L}(\rn)$ is similar.

For any $f\in (H_{\oz,L}(\rn)\cap L^2(\rn))$, by \eqref{5.2} and \eqref{5.4},
we obviously have $\car_hf\le \nn_hf$, which implies that
$H_{\oz,\nn_h}(\rn)\subset H_{\oz,\car_h}(\rn)$.

Conversely, since for all $f\in L^2(\rn)$,
we have $\nn_{h}^{1/2}f\ls \car_h f,$ where $\nn_{h}^{1/2}$
is as in \eqref{5.2}. Then by Lemma \ref{l5.3}, we obtain that for all
$f\in L^2(\rn)\cap H_{\oz,\car_h}(\rn)$,
$$\|\nn_{h}f\|_{L(\oz)}\ls \|\nn_{h}^{1/2}f\|_{L(\oz)}\ls \|\car_h f\|_{L(\oz)},$$
which implies that $H_{\oz,\car_h}(\rn)\subset H_{\oz,\nn_h}(\rn) $,
and hence, completes the proof of Corollary \ref{c5.1}.
\end{proof}

\section{The Carleson measure and the John-Nirenberg inequality}

\hskip\parindent In this section, we characterize the space $\bmoz$
via the $\ro$-Carleson measure and establish the John-Nirenberg
inequality for elements in $\bmoz$, where $L^\ast$ denotes the conjugate
operator of $L$ in $L^2(\rn)$.

Recall that a measure $d\mu$ on ${\rr}^{n+1}_+$ is called a $\ro$-Carleson measure if
\begin{equation*}
\|d\mu\|_\ro\equiv \sup_{B\subset
\rn}\lf\{\frac{1}{|B|[\ro(|B|)]^2}\iint_{\widehat{B}}|\,d\mu|\r\}^{1/2}<\fz,
\end{equation*}
where the supremum is taken over all balls $B$ of $\rn$ and
$\widehat{B}$ denotes the tent over $B$; see \cite{hsv}.

\begin{thm}\label{t6.1}
Let $\oz$ satisfy Assumption (A), $\ro$ be as in \eqref{2.4}
and $M> \frac n2 (\frac1{p_\oz}-\frac 12)$.

{\rm (i)} If $f\in \bmoz$, then $d\mu_f$ is a $\ro$-Carleson measure and there
exists a positive constant $C$ independent of $f$ such that
$\|d\mu_f\|_{\ro}\le C\|f\|^2_{\bmoz}$, where
\begin{equation}\label{6.1}
  d\mu_f\equiv \lf|(t^2L^\ast)^Me^{-t^2L^\ast}f(x)\r|^2 \frac{\,dx\,dt}{t}.
\end{equation}

{\rm (ii)} Conversely, if $f\in \cm_{\oz,L^\ast}^M(\rn)$ satisfies \eqref{4.8} with
certain $q\in (p_{L^\ast},2]$ and $\ez_1>0$, and $d\mu_f$ is
a $\ro$-Carleson measure, then $f\in \bmoz$
and there exists a positive constant $C$ independent of $f$ such that
$\|f\|^2_{\bmoz}\le C\|d\mu_f\|_{\ro}$, where $d\mu_f$ is as in \eqref{6.1}.
\end{thm}
\begin{proof}[\bf Proof.]
It follows from Lemma \ref{l4.2} that (i) holds.

To show (ii), let $\wz M>M+\ez_1+\frac n4$ and
$\ez>n(\frac{1}{p_\oz}-\frac{1}{p_\oz^+})$. By Lemma \ref{l4.3},
we have
\begin{equation*}
\la f, g\ra =\wz C_M\iint_\rnz (t^2L^\ast)^Me^{-t^2L^\ast}f(x)
\overline{t^2Le^{-t^2L}g(x)}\frac{\,dx\,dt}{t},
\end{equation*}
where $g$ is a finite combination of $(\oz,q',\wz M,\ez)$-molecules and
$q'=\frac{q}{q-1}$. Then by \eqref{4.10}, we obtain that
$$|\la f, g\ra|\ls \|d\mu_f\|_{\ro}\|g\|_{H_{\oz,L}(\rn)}.$$
Since $H_{\oz,\rm fin}^{q',\wz M,\ez}$ is dense in $H_{\oz,L}(\rn)$,
we obtain that $f\in (H_{\oz,L}(\rn))^\ast,$
which combined with Theorem \ref{t4.1} implies that
 $f\in \bmoz$ and $\|f\|_{\bmoz}\ls \|\,du_f\|_{\ro}$.
This finishes the proof of Theorem \ref{t6.1}.
\end{proof}

Recall that for every cube $Q$, $\ell(Q)$ denotes its side-length.
\begin{lem}\label{l6.1}
Let $F\in L^2_{\loc}(\rnz)$. Suppose that there exist $\bz\in(0,1)$
and $N\in (0,\fz)$ such that for certain $a\in(\frac{5\sqrt
n}{2},\fz)$ and all cubes $Q\subset \rn$,
\begin{equation*}
  \lf|\lf\{x\in Q:\, \lf(\int_0^{\ell(Q)}\int_{B(x,3at)}|F(y,t)|^2\frac{\,dy\,dt}
  {t^{n+1}}\r)^{1/2}>N\ro(|Q|)\r\}\r|\le \bz|Q|.
\end{equation*}
Then
\begin{equation}\label{6.2}
  \sup_{\mathrm{cubes} \,Q\subset \rn}\frac{1}{|Q|\ro(|Q|)^{p}}\int_{Q}
  \lf(\int_0^{\ell(Q)}\int_{B(x,at)}|F(y,t)|^2\frac{\,dy\,dt}{t^{n+1}}
  \r)^{p/2}\,dx\le \frac{2N^p}{1-\bz}
\end{equation}
for all $p\in (1,\fz)$.
\end{lem}

\begin{proof}[\bf Proof.] Let $\Omega\equiv \{x\in Q:\,(\int_0^{\ell(Q)}
\int_{B(x,3at)}|F(y,t)|^2\frac{\,dy\,dt}
{t^{n+1}})^{1/2}>N\ro(|Q|)\}$. Applying the Whitney decomposition to
$\Omega$, we obtain a  family $\{Q_j\}_j$ of disjoint cubes such
that $(\cup_jQ_j)=\Omega$ and $\dist(Q_j, Q\setminus\Omega) \in
(\sqrt n\ell(Q_j), 4\sqrt n\ell(Q_j))$; see the proof of \cite[Lemma
10.1]{hm1}. For $\dz\in (0,\ell(Q))$, define
$$M(\dz)\equiv \sup_{\mathrm{cubes} \,\wz Q\subset Q}\frac{1}{|\wz Q|}
\int_{\wz Q}\lf(\int_\dz^{\ell(\wz Q)}
\int_{B(x,a(t-\dz))}\lf|\frac{F(y,t)}{\ro(|Q|)}\r|^2
\frac{\,dy\,dt}{t^{n+1}}\r)^{p/2}\,dx,$$
where $B(x,a(t-\dz))\equiv\emptyset$ if $\dz\ge t$. Now, observe that
\begin{eqnarray*}
&&\int_Q\lf(\int_\dz^{\ell( Q)}\int_{B(x,a(t-\dz))}\lf|\frac{F(y,t)}
{\ro(|Q|)}\r|^2\frac{\,dy\,dt}{t^{n+1}}\r)^{p/2}\,dx\\
&&\hs\le \int_{Q\setminus \Omega}\lf(\int_0^{\ell( Q)}\int_{B(x,3at)}
\lf|\frac{F(y,t)}
{\ro(|Q|)}\r|^2\frac{\,dy\,dt}  {t^{n+1}}\r)^{p/2}\,dx\nonumber\\
&&\hs\hs+\sum_{\{j: \,\ell(Q_j)>\dz\}}\int_{Q_j}\lf(\int_\dz^{\ell(Q_j)}
\int_{B(x,a(t-\dz))}\lf|\frac{F(y,t)}{\ro(|Q|)}\r|^2\frac{\,dy\,dt}
{t^{n+1}}\r)^{p/2}\,dx\nonumber\\
&&\hs\hs+\sum_j\int_{Q_j}\lf(\int_{\max\{\ell(Q_j),\dz\}}^{\ell(Q)}
\int_{B(x,a(t-\dz))}\lf|\frac{F(y,t)}{\ro(|Q|)}\r|^2\frac{\,dy\,dt}
{t^{n+1}}\r)^{p/2}\,dx\nonumber\\
&&\hs\le N^p|Q|+\bz |Q|M(\dz)+\sum_j\int_{Q_j}
\lf(\int_{\max\{\ell(Q_j),\dz\}}^{\ell(Q)}
\int_{B(x,a(t-\dz))}\lf|\frac{F(y,t)}{\ro(|Q|)}\r|^2\frac{\,dy\,dt}
{t^{n+1}}\r)^{p/2}\,dx\nonumber\\
&&\hs\equiv N^p|Q|+\bz |Q|M(\dz)+\mathrm{I}.
\end{eqnarray*}
Since $\dist (Q_j, Q\setminus\Omega)\in (\sqrt n\ell(Q_j), 4\sqrt n\ell(Q_j))$,
there exists $\wz x\in (Q\setminus\Omega)$ such that for all $x\in Q_j$,
$$|x-\wz x|\le |x-x_{Q_j}|+|x_{Q_j}-\wz x|\le 5\sqrt n \ell({Q_j}).$$
Then by the fact that $a\ge 5\sqrt n/2$, we obtain
$$\{(y,t):\,y\in B(x, a(t-\dz)),\,\max\{\ell(Q_j),\dz\}<t<\ell(Q)\}\subset
\{(y,t):\,y\in B(\wz x, 3at), \,t<\ell(Q)\},$$ which implies that
\begin{eqnarray*}
  \mathrm{I}\le\sum_j\int_{Q_j}\sup_{\wz x\in Q\setminus \Omega}\lf(\int_{0}^{\ell(Q)}
\int_{B(\wz x,3at)}\lf|\frac{F(y,t)}{\ro(|Q|)}\r|^2\frac{\,dy\,dt}
{t^{n+1}}\r)^{p/2}\,dx\le N^p |Q|.
\end{eqnarray*}
For every cube $\wz Q\subset Q$,
let $\wz \Omega\equiv \{x\in \wz Q:\,(\int_0^{\ell(\wz Q)}
\int_{B(x,3at)}|F(y,t)|^2\frac{\,dy\,dt}  {t^{n+1}})^{1/2}>N\ro(|Q|)\}$.
Then
$$|\wz \Omega|\le \lf|\lf\{x\in \wz Q:\,\lf(\int_0^{\ell(\wz Q)}
\int_{B(x,3at)}|F(y,t)|^2\frac{\,dy\,dt}  {t^{n+1}}\r)^{1/2}>N
\ro(|\wz Q|)\r\}\r|\le
\bz|\wz Q|.$$
Repeating the above estimates, we obtain
\begin{eqnarray*}
&&\int_{\wz Q}\lf(\int_\dz^{\ell(\wz Q)}\int_{B(x,a(t-\dz))}\lf|\frac{F(y,t)}
{\ro(|Q|)}\r|^2\frac{\,dy\,dt}{t^{n+1}}\r)^{p/2}\,dx\le 2N^p
|\wz Q|+\bz |\wz Q|M(\dz),
\end{eqnarray*}
which via taking the supremum on $\wz Q$ implies that
$(1-\bz)M(\dz)\le 2N^p.$
Letting $\dz\to 0$, we finally obtain
 \begin{eqnarray*}
 &&\frac{1}{|Q|\ro(|Q|)^{p}}\int_{Q}
  \lf(\int_0^{\ell(Q)}\int_{B(x,at)}|F(y,t)|^2\frac{\,dy\,dt}
  {t^{n+1}}\r)^{p/2}\,dx
  \le\overline{\lim_{\dz\to 0}}M(\dz)\le \frac{2N^p}{1-\bz},
    \end{eqnarray*}
which implies \eqref{6.2}, and hence, completes the proof of Lemma \ref{l6.1}.
\end{proof}

\begin{thm}\label{t6.2}
Let $\oz$ satisfy Assumption (A), $\ro$ be as in \eqref{2.4}
and $M> \frac n2 (\frac1{p_\oz}-\frac 12)$.
Then the spaces $\bmos$  for all $q\in (p_{L^\ast},\wz p_{L^\ast})$
coincide with equivalent norms.
\end{thm}
\begin{proof}[\bf Proof.]
It follows from the H\"older inequality that
\begin{equation*}
\|f\|_{\mathrm{BMO}^{p,M}_{\ro,L^\ast}(\rn)}\le
\|f\|_{\mathrm{BMO}^{2,M}_{\ro,L^\ast}(\rn)}
\le \|f\|_{\mathrm{BMO}^{q,M}_{\ro,L^\ast}(\rn)},
\end{equation*}
where $p_{L^\ast}<p<2<q<\wz p_{L^\ast}$.

Let us now show that
$\|f\|_{\mathrm{BMO}^{2,M}_{\ro,L^\ast}(\rn)}\ls
\|f\|_{\mathrm{BMO}^{p,M}_{\ro,L^\ast}(\rn)}$
 for all $p\in (p_{L^\ast},2)$. Write
 \begin{eqnarray*}
   &&\lf\{\int_{Q}\lf[\int_0^{\ell(Q)}
   \int_{B(x,9\sqrt n t)}|(t^2L^\ast)^Me^{-t^2L^\ast}f(y)|^2\frac{\,dy\,dt}{t^{n+1}}
   \r]^{p/2}\,dx\r\}^{1/p}\\
   &&\hs\le\lf\{\int_{Q}\lf[\int_0^{\ell(Q)}
   \int_{B(x,9\sqrt n t)}|(t^2L^\ast)^Me^{-t^2L^\ast}\r.\r.\\
   &&\hs\hs\lf.\lf.\times[I-(I+[ 9\sqrt n \ell(Q)]^2L^\ast)^{-1}]^M f(y)|^2
   \frac{\,dy\,dt}{t^{n+1}}\r]^{p/2}\,dx\r\}^{1/p}\\
   &&\hs\hs+\lf\{\int_{Q}\lf[\int_0^{\ell(Q)}
   \int_{B(x, 9\sqrt n t)}|(t^2L^\ast)^Me^{-t^2L^\ast}\r.\r.\\
   &&\hs\hs\lf.\lf.\times[I-(I-(I+[9\sqrt n \ell(Q)]^2L^\ast)^{-1})^M]f(y)|^2
   \frac{\,dy\,dt}{t^{n+1}}\r]^{p/2}\,dx\r\}^{1/p}
   \equiv  \mathrm{H}+\mathrm{I}.
 \end{eqnarray*}
Let $9\sqrt nQ$ denote the cube with the same center as $Q$ and
side-length $9\sqrt n$ times $\ell(Q)$. Then by the H\"older inequality,
Lemmas \ref{l2.1} and \ref{l2.3}, we have
\begin{eqnarray*}
   \mathrm{H}&&\le \sum_{j=0}^2\lf\{\int_{Q}\lf[\int_0^{\ell(Q)}
   \int_{B(x,9\sqrt nt)} |(t^2L^\ast)^Me^{-t^2L^\ast}\r.\r.\\
 &&\hs\hs\lf.\lf.\times [\chi_{U_j(9\sqrt n Q)}(I-(I+[9\sqrt n\ell(Q)]^2
 L^\ast)^{-1})^M f](y)|^2
   \frac{\,dy\,dt}{t^{n+1}}\r]^{p/2}\,dx\r\}^{1/p}+\sum_{j=3}^\fz \cdots\\
&&\equiv \mathrm{H}_1+\mathrm{H}_2.
 \end{eqnarray*}
By Lemma \ref{l2.4} and Proposition 4 in \cite{cms},
we obtain
\begin{eqnarray*}
\mathrm{H}_1&&\le\sum_{j=0}^2 \lf\|\ca_{9\sqrt n}\lf\{(t^2L^\ast)^Me^{-t^2L^\ast}
[\chi_{U_j(9\sqrt n Q)}(I-(I+[9\sqrt n\ell(Q)]^2L^\ast)^{-1})^M f\r\}\r\|_{L^p(\rn)}\\
&&\ls \sum_{j=0}^2\lf\|\ca \lf\{(t^2L^\ast)^Me^{-t^2L^\ast}
[\chi_{U_j(9\sqrt n Q)}(I-(I+[9\sqrt n\ell(Q)]^2L^\ast)^{-1})^M f\r\}\r\|_{L^p(\rn)}\\
&&\ls \sum_{j=0}^2\|(I-(I+[9\sqrt n\ell(Q)]^2L^\ast)^{-1})^M
f\|_{L^p(\chi_{U_j(9\sqrt n Q)})}\\
&&\ls \ro(|Q|)|Q|^{1/p}\|f\|_{\mathrm{BMO}^{p,M}_{\ro,L^\ast}(\rn)}.
 \end{eqnarray*}
For the term $\mathrm{H}_2$, noticing that $U_j(9\sqrt nQ)$ can be covered
by $2^{jn}$ cubes of side-length $9\sqrt n\ell(Q)$, which together with
Lemma \ref{l2.3} and the H\"older inequality implies that
\begin{eqnarray*}
\mathrm{H}_2&&\le\sum_{j=3}^\fz |Q|^{1/p-1/2}\lf\{\int_0^{\ell(Q)}
\int_{10\sqrt nQ}\lf|(t^2L^\ast)^Me^{-t^2L^\ast}\r.\r.\\
&&\hs\lf.\lf.\times \lf[\chi_{U_j(9\sqrt n Q)}(I-(I+[9\sqrt n\ell(Q)]^2
   L^\ast)^{-1})^M f\r](x)\r|^2
 \frac{\,dx\,dt}{t}\r\}^{1/2}\\
 &&\ls |Q|^{1/p-1/2}\sum_{j=3}^\fz \lf\{\int_0^{\ell(Q)}
 e^{-\frac{(2^jn\ell(Q))^2}{ct^2}}\r.\\
 &&\hs\lf.\times \lf\|(I-(I+[9\sqrt n\ell(Q)]^2L^\ast)^{-1})^M
 f\r\|_{L^p(U_j(9\sqrt n Q))}^2
\frac{\,dt}{t^{1+n/p-n/2}}\r\}^{1/2}\\
&&\ls \ro(|Q|)|Q|^{1/p}\|f\|_{\mathrm{BMO}^{p,M}_{\ro,L^\ast}(\rn)}\\
&&\hs\times\lf\{|Q|^{1/p-1/2}\sum_{j=3}^\fz\lf[\int_0^{\ell(Q)}
 \lf(\frac{t}{2^j\ell(Q)}\r)^{N}
   2^{2jn}\frac{\,dt}{t^{1+n/p-n/2}}\r]^{1/2}\r\}\\
   &&\ls \ro(|Q|)|Q|^{1/p}\|f\|_{\mathrm{BMO}^{p,M}_{\ro,L^\ast}(\rn)},
\end{eqnarray*}
where $c$ is the positive constant as in Lemma \ref{l2.3} and
$N\in\cn$ is large enough. Thus, $$\mathrm{H}\le \mathrm{H}_1+\mathrm{H}_2\ls
\ro(|Q|)|Q|^{1/p}\|f\|_{\mathrm{BMO}^{p,M}_{\ro,L^\ast}(\rn)} .$$

Applying the formula that
\begin{eqnarray*}
&&\lf[I-\lf(I-[I+(9\sqrt n\ell(Q))^2L^\ast]^{-1}\r)^M\r]f\\
&&\hs\hs=\sum_{k=1}^M C_{M}^k ([9\sqrt n\ell(Q)]^2L^\ast)^{-k}
\lf[I-(I+[9\sqrt n\ell(Q)]^2L^\ast)^{-1}\r]^Mf,
\end{eqnarray*}
where $C_M^k$ denotes the combinatorial number, by a way similar to
the estimate of $\mathrm{H}$, we also have
$\mathrm{I}\ls \ro(|Q|)|Q|^{1/p}\|f\|_{\mathrm{BMO}^{p,M}_{\ro,L^\ast}(\rn)}.$

Combing the estimates of $\mathrm{H}$ and $\mathrm{I}$ yields that
 \begin{eqnarray*}
   &&\lf\{\int_{Q}\lf[\int_0^{\ell(Q)}
   \int_{B(x,9\sqrt n t)}|(t^2L^\ast)^Me^{-t^2L^\ast}f(y)|^2\frac{\,dy\,dt}{t^{n+1}}
   \r]^{p/2}\,dx\r\}^{1/p}\\
   &&\hs\hs\hs\hs\hspace{5cm}\ls \ro(|Q|)|Q|^{1/p}
   \|f\|_{\mathrm{BMO}^{p,M}_{\ro,L^\ast}(\rn)},
 \end{eqnarray*}
which together with Lemma \ref{l6.1} implies that
\begin{eqnarray*}
&&\sup_{\mathrm{balls}\, B\subset \rn}\frac{1}{\ro(|B|)}\lf[\frac{1}{|B|}\iint_{\widehat B}
|(t^2L^\ast)^Me^{-t^2L^\ast}f(x)|^2 \frac{\,dx\,dt}{t}\r]^{1/2}\\
&&\hs\hs\sim\sup_{\mathrm{cubes}\, Q\subset \rn}\frac{1}{\ro(|Q|)}
\lf[\frac{1}{|Q|}\iint_{\widehat Q}
|(t^2L^\ast)^Me^{-t^2L^\ast}f(x)|^2 \frac{\,dx\,dt}{t}\r]^{1/2}\\
&&\hs\hs\ls\sup_{\mathrm{cubes}\, Q\subset \rn} \lf(\frac{1}{|Q|\ro(|Q|)^{2}}\int_{Q}
\int_0^{\ell(Q)}\int_{B(x,3\sqrt n t)}|(t^2L^\ast)^Me^{-t^2L^\ast}f(y)|^2
\frac{\,dy\,dt}{t^{n+1}}\,dx\r)^{1/2}\\
&&\hs\hs\ls \|f\|_{\mathrm{BMO}^{p,M}_{\ro,L^\ast}(\rn)}.
\end{eqnarray*}
Then by Theorem \ref{t6.1}, we obtain
$\|f\|_{\mathrm{BMO}^{2,M}_{\ro,L^\ast}(\rn)}\ls
\|f\|_{\mathrm{BMO}^{p,M}_{\ro,L^\ast}(\rn)}$.

Finally, let us show that
$\|f\|_{\mathrm{BMO}^{q,M}_{\ro,L^\ast}(\rn)}
\ls\|f\|_{\mathrm{BMO}^{2,M}_{\ro,L^\ast}(\rn)}$
for all $q\in (2,\wz p_{L^\ast})$. Let $q'$ be the conjugate index $q$.
For any ball $B$, let $h\in L^2(B)\subset L^{q'}(B)$ such that
$\|h\|_{L^{q'}(B)}\le \frac{1}{|B|^{1-1/q'}\ro(|B|)}$.
From Lemma \ref{l2.3}, similarly to the proof of Theorem \ref{t4.1}, it is easy to follow
that $(I-e^{-r_B^2L})^Mh$ is a multiple of an $(\oz,q',M,\ez)$-molecule,
and hence,
$$\|(I-e^{-r_B^2L})^Mh\|_{H_{\oz,L}(\rn)}\ls 1.$$

Now let $f\in \mathrm{BMO}^{2,M}_{\ro,L^\ast}(\rn)$. By Theorem \ref{t4.1},
$f\in (H_{\oz,L}(\rn))^\ast$, and hence,
\begin{equation*}
|\la (I-e^{-r_B^2L^\ast})^Mf, h\ra|=|\la f, (I-e^{-r_B^2L})^Mh\ra|\ls
\|f\|_{\mathrm{BMO}^{2,M}_{\ro,L^\ast}(\rn)}.
\end{equation*}
Taking supremum over all such $h$ yields that
$\|f\|_{\mathrm{BMO}^{q,M}_{\ro,L^\ast}(\rn)}\ls
\|f\|_{\mathrm{BMO}^{2,M}_{\ro,L^\ast}(\rn)}$,
which completes the proof of Theorem \ref{t6.2}.
\end{proof}

\section{Some applications\label{s7}}

\hskip\parindent In this section, we establish the boundedness
on Orlicz-Hardy spaces of the Riesz
transform and the fractional integral associated with the operator $L$
as in \eqref{1.2}.

Recall that the Littlewood-Paley $g$-function $g_L$ is defined by setting,
for all $f\in L^2(\rn)$ and $x\in\rn$,
\begin{equation*}
     g_L f(x)\equiv \lf(\int_0^\fz  |t^2L e^{-t^2L}f(x)|^2\frac{\,dt}{t}\r)^{1/2}.
\end{equation*}
By the proof of Theorem 3.4 in \cite{hm1}, we know that $g_L$ is bounded on $L^2(\rn)$.

Similarly to Theorems 3.2 and 3.4 in \cite{hm1}, we have the following conclusion.

\begin{thm}\label{t7.1}
Let $\oz$ satisfy Assumption (A) and $p\in (p_L,2]$. Suppose that the non-negative
sublinear operator or linear operator $T$ is bounded on $L^p(\rn)$ and there exist
$C>0$, $M\in\cn$ and $M>\frac n2(\frac 1{p_\oz}-\frac 12)$
such that for all closed sets $E,\,F$ in $\rn$ with $\dist(E,F)>0$ and all
 $f\in L^p(\rn)$ supported in $E$,
 \begin{equation}\label{7.1}
   \|T(I-e^{-tL})^Mf\|_{L^p(F)}\le C\lf(\frac{t}{\dist(E,F)^2}\r)^M\|f\|_{L^p(E)}
 \end{equation}
 and
 \begin{equation}\label{7.2}
   \|T(tLe^{-tL})^Mf\|_{L^p(F)}\le C\lf(\frac{t}{\dist(E,F)^2}\r)^M\|f\|_{L^p(E)}
 \end{equation}
for all $t>0$. Then $T$ extends to a bounded sublinear or linear operator from $H_{\oz,L}(\rn)$ to
$L(\oz)$. In particular, the Riesz transform $\nz L^{-1/2}$ and
the Littlewood-Paley $g$-function $g_L$ is bounded from $H_{\oz,L}(\rn)$ to $L(\oz)$.
\end{thm}

\begin{proof}[\bf Proof.]
Let $\ez>n(1/ p_\oz-1/ \wz p_\oz)$, where $\wz p_\oz$ is as in
Convention (C). Since $T$ is bounded on $L^p(\rn)$, by Lemma \ref{l5.1},
to show that $T$ is bounded from $H_{\oz,L}(\rn)$ to $L(\oz)$, it suffices
to show that for all $\lz\in\cc$ and
$(\oz,\fz,M,\ez)$-molecules $\az$ adapted to balls $B$,
\begin{equation}\label{7.3}
  \int_\rn \oz[T(\lz\az)(x)]\,dx\ls |B|\oz\lf(\frac{|\lz|}{|B|\ro(|B|)}\r).
\end{equation}

To prove \eqref{7.3}, we write
\begin{eqnarray*}&&\int_{\rn}\oz(T(\lz\az)(x))\,dx\\
&&\hs\le \int_{\rn}\oz(|\lz|T([I-e^{-r^2_BL}]^M\az)(x))\,dx
+\int_{\rn}\oz(|\lz|T((I-[I-e^{-r^2_BL}]^M)\az)(x))\,dx\\
&&\hs\ls\sum_{j=0}^\fz \int_{\rn}\oz(|\lz|T([I-e^{-r^2_BL}]^M
(\az\chi_{U_j(B)}))(x))\,dx\\
&&\hs\hs+\sum_{j=0}^\fz \sup_{1\le k\le M}
\int_{\rn}\oz\lf(|\lz|T\lf\{\lf[\frac{k}{M}r_B^{2}Le^{-\frac{k}{M}r^2_BL}\r]^M
(\chi_{U_j(B)}(r_B^{-2}L^{-1})^{M}\az)\r\}(x)\r)\,dx\\
&&\hs\equiv\sum_{j=0}^\fz \mathrm{H}_j+\sum_{j=0}^\fz\mathrm{I}_j.
\end{eqnarray*}

For each $j\ge 0$, let $B_j\equiv 2^jB$.
Since $\oz$ is concave, by the Jensen inequality and the H\"older inequality, we obtain
\begin{eqnarray*}
\mathrm{H}_j&&\ls \sum_{k=0}^\fz \int_{U_k(B_j)}\oz(|\lz|
T([I-e^{-r^2_BL}]^M(\az\chi_{U_j(B)}))(x))\,dx\\
&&\sim \sum_{k=0}^\fz \int_{2^kB_j}\oz(|\lz|\chi_{U_k(B_j)}(x)
T([I-e^{-r^2_BL}]^M(\az\chi_{U_j(B)}))(x))\,dx\\
&&\ls \sum_{k=0}^\fz |2^kB_j|\oz\lf(\frac{|\lz|}{|2^kB_j|}\int_{U_k(B_j)}
T([I-e^{-r^2_BL}]^M(\az\chi_{U_j(B)}))(x)\,dx\r)\\
&&\ls \sum_{k=0}^\fz |2^kB_j|\oz\lf(\frac{|\lz|}{|2^kB_j|^{1/p}}
\|T([I-e^{-r^2_BL}]^M(\az\chi_{U_j(B)}))\|_{L^p(U_k(B_j))}\r).
\end{eqnarray*}
By the $L^p(\rn)$-boundedness of $T$, Lemma \ref{l2.3} and \eqref{7.1},
we have that for $k=0,\,1,\,2$,
$$\|T([I-e^{-r^2_BL}]^M(\az\chi_{U_j(B)}))\|_{L^p(U_k(B_j))}\ls
\|\az\|_{L^p(U_j(B))},$$
and that for $k\ge 3$,
$$\|T([I-e^{-r^2_BL}]^M(\az\chi_{U_j(B)}))\|_{L^p(U_k(B_j))}\ls
\lf(\frac{1}{2^{k+j}}\r)^{2M}\|\az\|_{L^p(U_j(B))}^2,$$
which, together with Definition \ref{d4.2}, $2Mp_\oz>n(1-p_\oz/2)$ and Assumption (A),
implies that
\begin{eqnarray*}
\mathrm{H}_j&&\ls |B_j|\oz\lf(\frac{|\lz|2^{-j\ez}}{|B_j|\ro(|B_j|) }\r)+
\sum_{k=3}^\fz |2^kB_j|\oz\lf(\frac{|\lz|2^{-{(2M)(j+k)}-j\ez}}
{|2^kB_j|^{1/p} |B_j|^{1-1/p}\ro(|B_j|) }\r)\\
&&\ls 2^{-jp_\oz\ez}\lf\{1+\sum_{k=3}^\fz 2^{kn(1-p_\oz/p)}2^{-2Mp_\oz(j+k)}\r\}
|B_j|\oz\lf(\frac{|\lz|}{|B_j|\ro(|B_j|) }\r)\\
&&\ls 2^{-jp_\oz\ez}
|B_j|\oz\lf(\frac{|\lz|}{|B_j|\ro(|B_j|) }\r).
\end{eqnarray*}
Since $\ro$ is of lower type $1/\wz p_\oz-1$ and $\ez>n(1/p_\oz-1/\wz p_\oz)$, we further have
\begin{eqnarray*}
\sum_{j=0}^\fz \mathrm{H}_j&&\ls\sum_{j=0}^\fz 2^{-jp_\oz\ez}
|B_j|\lf\{\frac{|B|\ro(|B|)}{|B_j|\ro(|B_j|)}\r\}^{p_\oz}
\oz\lf(\frac{|\lz|}{|B|\ro(|B|) }\r)\\
&& \ls\sum_{j=0}^\fz 2^{-jp_\oz\ez}|B_j|\lf\{\frac{|B|}{|B_j|}\r\}^{p_\oz/\wz p_\oz}
\oz\lf(\frac{|\lz|}{|B|\ro(|B|) }\r)\nonumber\\
&&\ls\sum_{j=0}^\fz 2^{-jp_\oz\ez} 2^{jn(1-p_\oz/\wz p_\oz)}|B|
\oz\lf(\frac{|\lz|}{|B|\ro(|B|) }\r)
\ls |B|\oz\lf(\frac{|\lz|}{|B|\ro(|B|) }\r).\nonumber
\end{eqnarray*}

Similarly, we have
$$\sum_{j=0}^\fz \mathrm{I}_j\ls |B|
\oz\lf(\frac{|\lz|}{|B|\ro(|B|) }\r).$$

Thus, \eqref{7.3} holds, and hence,
$T$ is bounded from $H_{\oz,L}(\rn)$ to $L(\oz)$.

It was proved in \cite[Theorem 3.4]{hm1} that operators $g_L$
and $\nz L^{-1/2}$ satisfy
\eqref{7.1} and \eqref{7.2}; thus $g_L$ and $\nz L^{-1/2}$
are bounded from $H_{\oz,L}(\rn)$ to $L(\oz)$, which completes the
proof of Theorem \ref{t7.1}.
\end{proof}

We now give a fractional variant of Theorem \ref{t7.1}.
To this end, we first make some assumptions.

Let $\oz$ and $p_\oz$ satisfy Assumption (A) and $q\in [p_\oz, 1]$.
In what follows, for all $t\in (0,\fz)$, define
\begin{equation}\label{7.4}
v(t)\equiv \oz^{-1}(t)t^{1/q-1/p_\oz}.
\end{equation}

\begin{proof}[\bf Assumption (B)]
Let $\oz$ satisfy Assumption (A), $q\in [p_\oz, 1]$ and
 $p_L<r_1\le \min\{2,r_2\}\le r_2<\wz p_L$ satisfying that $1/p_\oz-1/q=1/r_1-1/r_2$.
Suppose that $v$ as in \eqref{7.4} is convex and $v(0)\equiv \lim_{t\to 0^+}v(t)=0$.
Then for all $t\in (0,\fz)$, let $\wz\oz (t)\equiv v^{-1}(t)$ and
$\wz\rho(t)\equiv \frac{t^{-1}}{\wz\oz^{-1}(t^{-1})}.$
\end{proof}

\begin{rem}\label{r7.1}\rm
(i) It is easy to see that if $\wz\oz$ is as in Assumption (B), then $\wz\oz$ also
satisfies Assumption (A) with $p_{\wz\oz}=q$ and
$p_{\wz\oz}^+=\frac{1}{1/p_\oz^++1/q-1/p_\oz}$. Moreover,
$$\wz\rho(t)\equiv \frac{t^{-1}}{\wz\oz^{-1}(t^{-1})}=
\frac{t^{-1}}{\oz^{-1}(t^{-1})t^{1/p_\oz-1/q}}=\ro(t)t^{-1/p_\oz+1/q}.$$

(ii) Let $p\in (0,1]$ and $\oz(t)\equiv t^p$ for all $t\in(0,\fz)$. In this
case, $p_\oz=p=p_\oz^+$, $q\in [p,1]$ and $\wz\oz(t)\equiv t^q$ for all
$t\in (0,\fz)$.
\end{rem}

\begin{thm}\label{t7.2}
 Let $q,\,r_1, \,r_2$, $\oz$ and $\wz\oz$ satisfy Assumption (B).
 Suppose that the linear operator $T$ is bounded from $L^{r_1}(\rn)$ to $L^{r_2}(\rn)$
 and there exist  $C>0$, $M\in\cn$ and $M>\frac n2(\frac 1{p_\oz}-\frac 12)
 +n(\frac 1{p_\oz}-\frac 1{p_\oz^+})$
satisfying that for all closed sets $E,\,F$ in $\rn$ with $\dist(E,F)>0$, all
 $f\in L^{r_1}(\rn)$ supported in $E$, and all $t>0$,
 \begin{equation}\label{7.5}
   \|T(I-e^{-tL})^Mf\|_{L^{r_2}(F)}\le
    C\lf(\frac{t}{\dist(E,F)^2}\r)^M\|f\|_{L^{r_1}(E)}
 \end{equation}
 and
 \begin{equation}\label{7.6}
   \|T(tLe^{-tL})^Mf\|_{L^{r_2}(F)}\le C
   \lf(\frac{t}{\dist(E,F)^2}\r)^M\|f\|_{L^{r_1}(E)}.
 \end{equation}
If $T$ is commutative with $L$, then $T$
extends to a bounded linear operator
from $H_{\oz,L}(\rn)$ to $H_{\wz\oz,L}(\rn)$.
\end{thm}

\begin{proof}[\bf Proof.]
Let $\ez\in (n(\frac 1{p_\oz}-\frac 1{p_\oz^+}),M-\frac n2(\frac 1{p_\oz}-\frac 12))$.
 It follows from Proposition \ref{p4.2} that for every
$f\in (H_{\oz,L}(\rn)\cap L^2(\rn))$, $f\in L^{r_1}(\rn)$ and there
exist $\{\lz_j\}_{j=1}^\fz\subset\cc$ and
$(\oz,\fz,2M,\ez)$-molecules $\{\az_j\}_{j=1}^\fz$ adapted to balls
$\{B_j\}_{j=1}^\fz$  such that $f=\sum_{j=1}^\fz\lz_j\az_j$ holds in
both $H_{\oz,L}(\rn)$ and $L^{r_1}(\rn)$; moreover,
$\Lambda(\{\lz_j\az_j\}_j)\ls \|f\|_{H_{\oz,L}(\rn)}.$

We first show that $T$ maps each $(\oz,\fz,2M,\ez)$-molecule
into a multiple of an $(\wz\oz,r_2,M,\ez)$-molecule. To this end,
assume that $\az$ is an $(\oz,\fz,2M,\ez)$-molecule adapted to
a ball $B\equiv B(x_B,r_B)$. By the boundedness of $T$ from
$L^{r_1}(\rn)$ to $L^{r_2}(\rn)$ and Remark \ref{r7.1}, we have
$T\az\in L^{r_2}(\rn)$ and for
any $k\in\{0,\,\cdots,\,M\}$ and $j\in\{0,\cdots,10\}$,
\begin{eqnarray*}
\|(r_B^2L)^{-k}T\az\|_{L^{r_2}(U_j(B))}&&\ls\|(r_B^2L)^{-k}\az\|_{L^{r_1}(\rn)}
\ls 2^{-j\ez}|2^jB|^{1/r_1-1}[\ro(|2^jB|)]^{-1}\\
&&\sim  2^{-j\ez}|2^jB|^{1/r_2-1}[\wz\ro(|2^jB|)]^{-1}.
\end{eqnarray*}
For $j\ge 11$, let $W_j(B)\equiv (2^{j+3}B\setminus 2^{j-3}B)$
and $E_j(B)\equiv (W_j(B))^\com$. Thus,
\begin{eqnarray*}
\|(r_B^2L)^{-k}T\az\|_{L^{r_2}(U_j(B))}&&\le\|T(I-e^{-r_B^2L})^M[(r_B^2L)^{-k}\az]
\|_{L^{r_2}(U_j(B))}\\
&&\hs+\|T[I-(I-e^{-r_B^2L})^M][(r_B^2L)^{-k}\az]\|_{L^{r_2}(U_j(B))}
\equiv \mathrm{H}+\mathrm{I}.
\end{eqnarray*}
By the boundedness from $L^{r_1}(\rn)$ to $L^{r_2}(\rn)$ of $T$,
Lemma \ref{l2.3}, \eqref{7.5}, the choice of $\ez$ and Remark \ref{r7.1}, we have
\begin{eqnarray*}
\mathrm{H}&&\le \|T(I-e^{-r_B^2L})^M[\chi_{W_j(B)}(r_B^2L)^{-k}\az]
\|_{L^{r_2}(U_j(B))}\\
&&\hs+ \|T(I-e^{-r_B^2L})^M[\chi_{E_j(B)}(r_B^2L)^{-k}\az]
\|_{L^{r_2}(U_j(B))}\\
&&\ls \|(r_B^2L)^{-k}\az\|_{L^{r_1}(W_j(B))}+
\lf(\frac{r_B^2}{\dist(U_j(B),E_j(B))^2}\r)^M
\|(r_B^2L)^{-k}\az\|_{L^{r_1}(E_j(B))}\\
&&\ls 2^{-j\ez}|2^jB|^{1/r_1-1}[\ro(|2^jB|)]^{-1}+2^{-2jM}
|B|^{1/r_1-1}[\ro(|B|)]^{-1}\\
&&\ls 2^{-j\ez}|2^jB|^{1/r_2-1}[\wz\ro(|2^jB|)]^{-1}.
\end{eqnarray*}
Similarly, by \eqref{7.6}, we have
\begin{eqnarray*}
\mathrm{I}&&\ls \sup_{1\le k\le M}\lf\|
T\lf[\frac{kr_B^2}{M}Le^{-\frac{kr_B^2}{M}L}\r]^M
[\chi_{W_j(B)}(r_B^2L)^{-k-M}\az]\r\|_{L^{r_2}(U_j(B))}\\
&&\hs+\sup_{1\le k\le M}\lf\|
T\lf[\frac{kr_B^2}{M}Le^{-\frac{kr_B^2}{M}L}\r]^M
[\chi_{E_j(B)}(r_B^2L)^{-k-M}\az]\r\|_{L^{r_2}(U_j(B))}\\
&&\ls \|(r_B^2L)^{-k-M}\az\|_{L^{r_1}(W_j(B))}+
\lf(\frac{r_B^2}{\dist(U_j(B),E_j(B))^2}\r)^M
\|(r_B^2L)^{-k-M}\az\|_{L^{r_1}(E_j(B))}\\
&&\ls 2^{-j\ez}|2^jB|^{1/r_2-1}[\wz\ro(|2^jB|)]^{-1}.
\end{eqnarray*}
Combining the above estimates, we finally obtain that $T\az$ is
a multiple of an $(\wz\oz,r_2,M,\ez)$-molecule.

Since $T$ is bounded from $L^{r_1}(\rn)$ to $L^{r_2}(\rn)$, we have
$Tf=\sum_{j=1}^\fz\lz_jT(\az_j)$ in $L^{r_2}(\rn)$. To finish the
proof, it remains to show that $\|Tf\|_{H_{\wz\oz,L}(\rn)}\ls
\|f\|_{H_{\oz,L}(\rn)}$. To this end, by Lemma \ref{l2.4}, the
subadditivity and the continuity of $\oz$ and \eqref{4.3} with
$\oz$ and $\ro$ replaced respectively by  $\wz\oz$ and $\wz\ro$, we obtain
\begin{equation}\label{7.7}
\int_\rn \wz\oz(\cs_L(Tf)(x))\,dx\le \sum_{j=1}^\fz\int_\rn
\wz\oz(|\lz_j|\cs_L(T\az_j)(x))\,dx\ls \sum_{j=1}^\fz
|B_j|\wz\oz\lf(\frac{|\lz_j|} { |B_j|\wz\ro(|B_j|)}\r).
\end{equation}

Choose $\gz\in (\Lambda(\{\lz_j\az_j\}_j),2\Lambda(\{\lz_j\az_j\}_j)]$. Then
for each $j\in \cn$, we have $\gz\ge |\lz_j|$; otherwise, there
exists $i\in\cn$ such that $\gz< |\lz_i|$, which together with the strictly
increasing property of $\oz$ further implies that
$$\sum_{j=1}^\fz |B_j|\oz\lf(\frac{|\lz_j|}
{\gz |B_j|\ro(|B_j|)}\r)\ge |B_i|\oz\lf(\frac{|\lz_i|}
{\gz |B_i|\ro(|B_i|)}\r)>|B_i|\oz\lf(\frac{1}
{|B_i|\ro(|B_i|)}\r)=1.$$
This contradicts to the assumption $\lz>\Lambda(\{\lz_j\az_j\}_j)$.
Thus, the claim is true. Therefore, by this claim and the strictly
increasing property of $\oz$, for each $j\in\cn$, we have
\begin{eqnarray*}
\lf[|B_j|\oz\lf(\frac{|\lz_j|}
{\gz|B_j|\ro(|B_j|)}\r)\r]^{1/p_\oz-1/q}
\le\lf[|B_j|\oz\lf(\oz^{-1}(|B_j|^{-1})\r)\r]^{1/p_\oz-1/q}\le 1,
\end{eqnarray*}
which implies that
\begin{eqnarray*}
\frac{|\lz_j|}{\gz|B_j|\wz\ro(|B_j|^{-1})}&&=
\frac{|\lz_j|}{\gz}\wz\oz^{-1}(|B_j|^{-1})
=\frac{|\lz_j|}{\gz}\oz^{-1}(|B_j|^{-1})|B_j|^{1/p_\oz-1/q}\\
&&\le \frac{|\lz_j|}{\gz}\oz^{-1}(|B_j|^{-1})
\lf[\oz\lf(\frac{|\lz_j|}
{\gz|B_j|\ro(|B_j|)}\r)\r]^{1/q-1/p_\oz}\\
&&=\wz\oz^{-1}\lf[\oz\lf(\frac{|\lz_j|}
{\gz|B_j|\ro(|B_j|)}\r)\r].
\end{eqnarray*}
Since $\wz\oz$ satisfies Assumption (B), we further obtain
\begin{eqnarray*}
\wz\oz\lf(\frac{|\lz_j|}
{\gz|B_j|\wz\ro(|B_j|)}\r)\le \oz\lf(\frac{|\lz_j|}
{\gz|B_j|\ro(|B_j|)}\r),
\end{eqnarray*}
and hence, by \eqref{7.7},
\begin{eqnarray*}
\int_\rn \wz\oz\lf(\frac{\cs_L(Tf)(x)}{\gz}\r)\,dx
\ls\sum_{j=1}^\fz |B_j|\wz\oz\lf(\frac{|\lz_j|}
{\gz|B_j|\wz\ro(|B_j|)}\r)\ls \sum_{j=1}^\fz |B_j|\oz\lf(\frac{|\lz_j|}
{\gz|B_j|\ro(|B_j|)}\r)\ls 1.
\end{eqnarray*}
Thus $\|Tf\|_{H_{\wz\oz,L}(\rn)}\ls \gz\ls \|f\|_{H_{\oz,L}(\rn)}$,
which together with a standard density argument
completes the proof of Theorem \ref{t7.2}.
\end{proof}

\begin{rem}\label{r7.2}\rm If we let $p\in (0,1]$ and $\oz(t)\equiv t^p$
for all $t\in (0,\fz)$, by Remark \ref{r7.1}(ii) and Theorem \ref{t7.2},
we know that the operator $T$ of Theorem \ref{t7.2}
in this case extends to a bounded linear operator from $H_L^p(\rn)$
to $H_L^q(\rn)$.
\end{rem}

In what follows, let $\gz\in(0,\frac n2(\frac {1}{p_L}-\frac {1}{\wz p_L}))$.
Recall that the generalized fractional integral $L^{-\gz}$ is given by setting,
for all $f\in L^2(\rn)$ and $x\in\rn$,
\begin{equation*}
L^{-\gz}f(x)\equiv \frac{1}{\Gamma(\gz)}\int^\fz_0 t^\gz
e^{-tL}f(x)\frac{\,dt}{t}.
\end{equation*}

Applying Theorem \ref{t7.2}, we obtain the boundedness of $L^{-\gz}$
from $H_{\oz,L}(\rn)$ to $H_{\wz\oz,L}(\rn)$ as follows.

\begin{thm}\label{t7.3}
Let $q,\,r_1, \,r_2$, $\oz$ and $\wz\oz$ satisfy Assumption (B)
and $\gz\in(0,\frac n2(\frac {1}{p_L}-\frac {1}{\wz p_L}))$
satisfying that $n(1/p_\oz-1/q)=2\gz$. Then the operator $L^{-\gz}$
satisfies \eqref{7.5} and \eqref{7.6} and hence, is bounded from
$H_{\oz,L}(\rn)$ to $H_{\wz\oz,L}(\rn)$.
\end{thm}

\begin{proof}[\bf Proof.]
Let $M\in\cn$ and $M>\frac n2(\frac 1{p_\oz}-\frac 12)
 +n(\frac 1{p_\oz}-\frac 1{p_\oz^+})$.
By \cite[Proposition 5.3]{a1}, the operator $L^{-\gz}$ is bounded from $L^{r_1}(\rn)$
to $L^{r_2}(\rn)$. Thus, by Theorem \ref{t7.2}, to show Theorem \ref{t7.3},
we only need to prove that $L^{-\gz}$ satisfies \eqref{7.5} and \eqref{7.6}.
We only give the proof of the former one, since \eqref{7.6}
can be proved in a similar way.

Let $E,\,F$ be closed sets in $\rn$ with $\dist(E,F)>0$ and
$f\in L^{r_1}(\rn)$ supported in $E$. Write
\begin{eqnarray*}
  \|L^{-\gz}(I-e^{-tL})^Mf\|_{L^{r_2}(\rn)}&&=
  \frac{1}{\Gamma (\gz)}\lf\|\int_0^\fz s^{\gz-1}e^{-sL}
  (I-e^{-tL})^Mf\,ds\r\|_{L^{r_2}(\rn)}\\
  &&\ls \lf\|\int_0^t s^{\gz-1}e^{-sL}(I-e^{-tL})^Mf\,ds\r\|_{L^{r_2}(\rn)}+
\lf\|\int_t^\fz \cdots\r\|_{L^{r_2}(\rn)}\\
&&\equiv \mathrm{H}_1+\mathrm{H}_2.
\end{eqnarray*}

It follows from Lemma \ref{l2.3} that
\begin{eqnarray*}
\mathrm{H}_1&&\ls \int_0^t \lf\|s^{\gz-1}e^{-sL}f\r\|_{L^{r_2}(\rn)}\,ds
+\sup_{1\le k\le M}\int_0^t \lf\|s^{\gz-1}e^{-sL}e^{-ktL}f
\r\|_{L^{r_2}(\rn)}\,ds\\
&&\ls \lf\{\int_0^t s^{\gz-1}s^{\frac n2 (\frac{1}{r_2}-\frac{1}{r_1})}
e^{-\frac{\dist(E,F)^2}{cs}}\,ds+
\int_0^t s^{\gz-1}t^{\frac n2 (\frac{1}{r_2}-\frac{1}{r_1})}
e^{-\frac{\dist(E,F)^2}{ct}}\,ds\r\}\|f\|_{L^{r_1}(\rn)}\\
&&\ls \lf(\frac{t}{\dist(E,F)^2}\r)^M\|f\|_{L^{r_1}(\rn)},
\end{eqnarray*}
here and in what follows, $c$ is the positive constant as in Lemma \ref{l2.3}.

For the term $\mathrm{H}_2$, since $I-e^{-tL}=\int_0^t Le^{-rL}\,dr$, by Lemmas
\ref{l2.1} and \ref{l2.3}, and the Minkowski inequality, we obtain
\begin{eqnarray*}
\mathrm{H}_2&&\ls \int_t^\fz
\lf\|s^{\gz-1}e^{-sL}\lf(\int_0^t Le^{-rL}\,dr\r)^Mf\r\|_{L^{r_2}(\rn)}\,ds\\
&&\ls  \int_t^\fz\frac{s^{\gz-1}}{s^M}\int_0^t\cdots\int_0^t
\lf\|(sL)^Me^{-sL}e^{-(r_1+\cdots+r_M)L}f
\r\|_{L^{r_2}(\rn)}\,dr_1\cdots\,dr_M\,ds\\
&&\ls \int_t^\fz\frac{s^{\gz-1}}{s^M} s^{\frac n2 (\frac{1}{r_2}-\frac{1}{r_1})}t^M
e^{-\frac{\dist(E,F)^2}{cs}}\,ds\|f\|_{L^{r_1}(\rn)}\\
&&\ls \int_0^\fz\lf(\frac{rt}{\dist(E,F)^2}\r)^Me^{-r}
\frac{\,dr}{r}\|f\|_{L^{r_1}(\rn)}
\ls \lf(\frac{t}{\dist(E,F)^2}\r)^M\|f\|_{L^{r_1}(\rn)},
\end{eqnarray*}
which implies that \eqref{7.5} holds for the operator $L^{-\gz}$, and hence, completes
the proof of Theorem \ref{t7.3}.
\end{proof}

\begin{rem}\label{r7.3}\rm
Similarly to Remark \ref{r7.2}, as a special case of Theorem \ref{t7.3},
we know that the operator
$L^{-\gz}$ maps $H_L^p(\rn)$ continuously into $H_L^p(\rn)$,
where $\gz,\,p,\,q$ satisfy $0<p\le q\le1$ and $n(1/p-1/q)=2\gz$.

\end{rem}

Using Theorems \ref{t7.1} and \ref{t7.3}, we further obtain
the following boundedness of the Riesz transform $\nz L^{-1/2}$
from $H_{\oz,L}(\rn)$ to $H_\oz(\rn)$. We first recall some
notions; see \cite{v,se96,hsv}.

In what follows, let $\cs(\rn)$ denote the space of all Schwartz functions and
$\cs'(\rn)$ the space of all Schwartz distributions.

\begin{defn}\label{d7.1} Let $\oz$ satisfy Assumption (A) and $p_{\oz}\in
(\frac {n}{n+1},1]$.  A function $a$ is called a
$(\ro,2)$-atom if

 (i) $\supp a\subseteq B$, where $B$ is a ball of $\rn$;

(ii)  $\|a\|_{L^2(\rn)}\le |B|^{-1/2}[\ro(|B|)]^{-1}$;

 (iii) $\int_{\rn} a(x)\,dx=0$.
\end{defn}

 \begin{defn}\label{d7.2} Let $\oz$ satisfy Assumption (A) and $p_{\oz}\in
(\frac {n}{n+1},1]$. The Orlicz-Hardy space
$H_{\oz}(\rn)$ is defined to be the set of all distributions
$f\in \cs'(\rn)$ that can be written as $f=\sum_{j=1}^\fz b_j$ in
$\cs'(\rn)$, where $\{b_j\}_{j=1}^\fz$ is a sequence of multiples of
$(\ro,2)$-atoms such that
$$\sum_{j=1}^\fz|B_j|\oz\lf(\frac{\|b_j\|_{L^2(\rn)}}{|B_j|^{1/2}}\r)<\fz,$$
where $\supp b_j\subset B_j.$ Moreover, define
$$\|f\|_{H_{\oz}(\rn)}\equiv \inf\lf\{\lz>0: \,\sum_{j=1}^\fz|B_j|
\oz\lf(\frac{\|b_j\|_{L^2(\rn)}}{\lz|B_j|^{1/2}}\r)\le 1\r\},$$
where the infimum is taken over all decompositions of $f$ as above.
\end{defn}

It is well known that the classical Orlicz-Hardy space defined by
using grand maximal functions is equivalent to the above atomic
Orlicz-Hardy space $H_\oz(\rn)$ as in Definition \ref{d7.2};
see \cite{v,se96}. Based on this fact, in what follows, we denote
both spaces by the same notation. Recall that $H_\oz(\rn)$ is
complete.

\begin{thm}\label{t7.4}
Let $\oz$ satisfy Assumption (A) and $p_\oz\in(\frac{n}{n+1},1]$.
Then the Riesz transform $\nz L^{-1/2}$
is bounded from $H_{\oz,L}(\rn)$ to $H_\oz(\rn)$. In particular,
$\nz L^{-1/2}$ is bounded from $H_{L}^p(\rn)$ to $H^p(\rn)$ for all
$p\in(\frac{n}{n+1},1]$
\end{thm}

\begin{proof}[\bf Proof.]
Let $\ez>1+n(\frac 1{p_\oz}-\frac 1{\wz p_\oz})$ and $M>
\frac{n}{2}(\frac 1{p_\oz}-\frac 12) +\ez$, where $\wz p_\oz$ is as
in Convention (C). Suppose that $\az$ is an
$(\oz,\fz,M,\ez)$-molecule associated to a ball $B\equiv
B(x_B,r_B)$. We first show that $\int_\rn \nz L^{-1/2}\az(x)\,dx=0$.

From the $L^2(\rn)$-boundedness of $\nz L^{-1/2}$
(see \cite[Theorem 4.1]{a1}), it follows that for $j=0,1,\cdots,10$,
\begin{equation}\label{7.8}
\|\nz L^{-1/2}\az\|_{L^2(U_j(B))}\le\|\nz L^{-1/2}\az\|_{L^2(\rn)}
\ls \|\az\|_{L^2(\rn)}\ls |B|^{-1/2}\ro(|B|)^{-1}.
\end{equation}
For $j\ge 11$, let $W_j(B)\equiv (2^{j+3}B\setminus 2^{j-3}B)$
and $E_j(B)\equiv(W_j(B))^\com$.
By the fact that  $\nz L^{-1/2}$ satisfies \eqref{7.1} and \eqref{7.2}
(see Theorem 3.4 in \cite{hm1}) together with the $L^2(\rn)$-boundedness
of $\nz L^{-1/2}$ and Lemma \ref{l2.2}, we have
\begin{eqnarray}\label{7.9}
&&\|\nz L^{-1/2}\az\|_{L^2(U_j(B))}\nonumber\\
&&\hs\le
\lf\|\nz L^{-1/2}(I-e^{-r_B^2L})^M \az\r\|_{L^{2}(U_j(B))}+\lf\|\nz L^{-1/2}\lf[I-(I-e^{-r_B^2L})^M\r]
\az\r\|_{L^{2}(U_j(B))}\nonumber\\
&&\hs\ls \lf\|\nz L^{-1/2}(I-e^{-r_B^2L})^M \lf[(\chi_{W_j(B)}
+\chi_{E_j(B)})\az\r]\r\|_{L^{2}(U_j(B))}\nonumber\\
&&\hs\hs+\sup_{1\le k\le M}\lf\|\nz L^{-1/2}\lf[
\lf(\frac{ k r_B^2L}{M}\r)e^{-\frac{k r_B^2L}{M}}\r]^M
\lf[(\chi_{W_j(B)}+\chi_{E_j(B)})(r_B^2L)^{-M}\az\r]
\r\|_{L^{2}(U_j(B))}\nonumber\\
&&\hs\ls \|\az\|_{L^2(W_j(B))}
+2^{-2jM}\|\az\|_{L^2(\rn)}\nonumber\\
&&\hs\hs+\|(r_B^2L)^{-M}\az\|_{L^2(W_j(B))}+2^{-2jM}
\|(r_B^2L)^{-M}\az\|_{L^2(\rn)}\nonumber\\
&&\hs\ls [2^{-j\ez}+2^{-j(2M-n/p_\oz+n/2)}]|2^jB|^{-1/2}[\ro(|2^jB|)]^{-1}.
\end{eqnarray}
Combining the above estimates and using the H\"older inequality, we see that $\nz L^{-1/2}\az\in L^1(\rn)$.

We now choose $p_L<s\le\min\{t,2\}\le t<\wz p_L$ such that $1/s-1/t=1/n$.
Since an $(\oz,\fz,M,\ez)$-molecule is also an $(\oz,s,M,\ez)$-molecule,
by the fact that $L^{-1/2}$ is bounded from $L^s(\rn)$ to $L^t(\rn)$
 (see \cite[Proposition 5.3]{a1}) and the H\"older inequality,
we have that for $j=0,1,\cdots,10$,
\begin{eqnarray*}
\||L^{-1/2}\az\|_{L^1(U_j(B))}&\le&|U_j(B)|^{1-1/t}\|L^{-1/2}\az\|_{L^t(\rn)}
\ls (2^{j}r_B)^{n(1-1/t)}\|\az\|_{L^{s}(\rn)}\\
&\ls& |B|^{1-1/t+1/s-1}[\ro(|B|)]^{-1}\sim |B|^{1/n}[\ro(|B|)]^{-1}.
\end{eqnarray*}
For $j\ge 11$, let $W_j(B)\equiv (2^{j+3}B\setminus 2^{j-3}B)$
and $E_j(B)\equiv(W_j(B))^\com$.
By Theorem \ref{t7.3}, we have that $L^{-1/2}$ satisfies \eqref{7.5} and \eqref{7.6},
which together with Lemma \ref{l2.3} and the H\"older inequality yields that
\begin{eqnarray*}
&&\| L^{-1/2}\az\|_{L^1(U_j(B))}\\
&&\hs\le |U_j(B)|^{1-1/t}\lf\{\lf\| L^{-1/2}(I-e^{-r_B^2L})^M
\az\r\|_{L^{t}(U_j(B))}\r.\\
&&\hs\hs\lf.+\lf\| L^{-1/2}\lf[I-(I-e^{-r_B^2L})^M\r]
\az\r\|_{L^{t}(U_j(B))}\r\}\nonumber\\
&&\hs\ls |U_j(B)|^{1-1/t}
\Bigg\{\lf\| L^{-1/2}(I-e^{-r_B^2L})^M \lf[(\chi_{W_j(B)}
+\chi_{E_j(B)})\az\r]\r\|_{L^{t}(U_j(B))}\nonumber\\
&&\hs\hs\lf.+\sup_{1\le k\le M}\lf\|L^{-1/2}\lf[
\lf(\frac{ k r_B^2L}{M}\r)e^{-\frac{k r_B^2L}{M}}\r]^M
\lf[(\chi_{W_j(B)}+\chi_{E_j(B)})(r_B^2L)^{-M}\az\r]
\r\|_{L^{t}(U_j(B))}\r\}\nonumber\\
&&\hs\ls  |U_j(B)|^{1-1/t} \lf\{\|\az\|_{L^s(W_j(B))}
+2^{-2jM}\|\az\|_{L^s(\rn)}\r.\nonumber\\
&&\hs\hs\lf.+\|(r_B^2L)^{-M}\az\|_{L^s(W_j(B))}+2^{-2jM}
\|(r_B^2L)^{-M}\az\|_{L^s(\rn)}\r\}\nonumber\\
&&\hs\ls [2^{-j(\ez-1)}+2^{-j(2M-n[1-1/t])}]|B|^{1/n}[\ro(|B|)]^{-1}.\nonumber
\end{eqnarray*}
Since $\ez>1+n(1/p_\oz-\wz p_\oz)$ and $M> \frac{n}{2p_\oz}$,
we obtain that $L^{-1/2}\az\in L^1(\rn)$.

Now we choose $\{\vz_j\}_{j=0}^\fz\subset C_0^\fz(\rn)$
such that

(i) $\sum_{j=0}^\fz \vz_j(x)=1$ for almost every $x\in\rn$;

(ii) for each $j\in\zz_+$, $\supp \vz_j\subset 2B_j$,
$\vz_j=1$ on $B_j$ and $0\le \vz_j\le 1$;

(iii) there exists $C_\vz>0$ such that for all $j\in\zz_+$ and $x\in\rn$,
$|\vz_j(x)|+|\nz \vz_j(x)|\le C_\vz$;

(iv) there exists $N_\vz\in \cn$ such that $\sum_{j=0}^\fz \chi_{2B_j}\le N_\vz$.

Using the properties of $\{\vz_j\}_{j=0}^\fz$ and
the facts that $L^{-1/2}\az,\ \nz L^{-1/2}\az\in L^1(\rn)$, we obtain
$$\int_\rn \nz L^{-1/2}\az(x)\,dx=\sum_{j=0}^\fz \int_{\rn}
\nz (\vz_j L^{-1/2}\az)(x)\,dx.$$
For each $j$, let $\eta_j\in C_0^\fz(\rn)$ such that $\eta_j=1$ on $2B_j$
 and $\supp \eta_j\subset 4B_j$. Then for each $i=1,2,\cdots,n$, we have
\begin{eqnarray*}
  \int_{\rn}\frac{\pa}{\pa x_i} (\vz_j L^{-1/2}\az)(x)\,dx&&=
  \int_{\rn}\eta_j(x)\frac{\pa}{\pa x_i} (\vz_j L^{-1/2}\az)(x)\,dx\\
 &&=-\int_{\rn} \vz_j (L^{-1/2}\az)(x)\frac{\pa}{\pa x_i}\eta_j(x)\,dx=0,
\end{eqnarray*}
which implies that $\int_\rn \nz L^{-1/2}\az(x)\,dx=0$.

To finish the proof, we borrow some ideas from \cite{tw}.
For $k\in \zz_+$, let
$\chi_k\equiv\chi_{U_k(B)}$,
$\widetilde{\chi}_k\equiv|U_k(B)|^{-1}\chi_k$,
$m_k\equiv\int_{U_k(B)}\nz L^{-1/2}\az(x)\,dx$
 and
$M_k\equiv \nz L^{-1/2}\az\chi_k-m_k\widetilde{\chi}_k.$ Then we have
\begin{equation*} \nz L^{-1/2}\az=\sum_{k=0}^\fz M_k+\sum_{k=0}^\fz
m_k\widetilde{\chi}_k
\end{equation*}
in $L^2(\rn)$. Now let $N_j\equiv\sum_{k=j}^\fz m_k.$  Since
$\int_\rn\nz L^{-1/2}\az(x)\,dx=0$, we have
\begin{equation}\label{7.10}
\nz L^{-1/2}\az=\sum_{k=0}^\fz M_k+\sum_{k=0}^\fz
N_{k+1}(\widetilde{\chi}_{k+1}-\widetilde{\chi}_k).
\end{equation}
Obviously, for all $k\in \zz_+$, $\int_{\rn}M_k(x)\,dx=0.$ Furthermore,
by \eqref{7.8}, \eqref{7.9} and the H\"older inequality, we have that
for all $k\in\zz_+$,
$$\|M_k\|_{L^2(\rn)}\ls
\|\nz L^{-1/2}\az\|_{L^2(U_k(B))}\ls 2^{-k\ez}|2^kB|^{-1/2}[\ro(|2^kB|)]^{-1},$$
which implies that
$\{2^{k\ez}M_k\}_{k\in\zz_+}$ is a family of $(\ro,2)$-atoms up to a harmless constant.

To deal with the second sum in \eqref{7.10},  by \eqref{7.8}, \eqref{7.9},
$|\widetilde{\chi}_{k+1}-\widetilde{\chi}_{k}|\ls |2^kB|^{-1}$ and the H\"older
inequality, we have that for all $k\in\zz_+$,
\begin{eqnarray*}
\|N_{k+1}(\widetilde{\chi}_{k+1}-\widetilde{\chi}_{k})\|_{L^2(\rn)}
&\ls&
\sum_{j=k}^\fz \frac{|2^jB|^{1/2}}{|2^kB|^{1/2}}|
\|\nz L^{-1/2}\az\|_{L^2(U_j(B))}\\
&\ls& 2^{-k\ez}|2^kB|^{-1/2}[\ro(|2^kB|)]^{-1}.
\end{eqnarray*}
This, together with
$\int_{\rn}[\widetilde{\chi}_{k+1}-\widetilde{\chi}_{k}]\,dx=0$,
implies that for each $k\in\zz_+$,
$2^{k\ez}N_{k+1}(\widetilde{\chi}_{k+1}-\widetilde{\chi}_{k})$ is a
$(\ro,2)$-atom up to a harmless constant.

By Assumption (A) and Convention (C),
$\ro$ is of lower type $1/\wz p_\oz-1$, which implies that
\begin{eqnarray}\label{7.11}
&&\sum_{j=0}^\fz |2^jB|\oz\lf(\frac{\|M_j\|_{L^2(\rn)}}{\lz|2^jB|^{1/2}}\r)
+\sum_{j=0}^\fz |2^jB|\oz\lf(\frac{\|N_{j+1}(\wz\chi_{j+1}-\wz\chi_{j})
\|_{L^2(\rn)}}{\lz|2^jB|^{1/2}}\r)\nonumber\\
&&\hs\ls\sum_{j=0}^\fz 2^{-jp_\oz\ez+jn(1-p_\oz/\wz p_\oz)}
|B|\oz\lf(\frac{1}{\lz|B|\ro(|B|) }\r)\sim
|B|\oz\lf(\frac{1}{\lz|B|\ro(|B|) }\r).
\end{eqnarray}

Now, suppose that $f\in H_{\oz,L}(\rn)\cap L^2(\rn)$.
By Proposition \ref{p4.2}, there exist
$(\oz,\fz,M,\ez)$-molecules $\{\az_k\}_{k=1}^\fz$ adapted
to balls $\{B_k\}_{k=1}^\fz$ and
numbers $\{\lz_k\}_{k=1}^\fz\subset \cc$ such that
$f=\sum_{k=1}^\fz\lz_k\az_k$ in both $H_{\oz,L}(\rn)$ and $L^2(\rn)$
with $\Lambda(\{\lz_k \az_k\}_k)\ls \|f\|_{H_{\oz,L}(\rn)}$.

For each $(\oz,\fz,M,\ez)$-molecule $\az_k$, by the above
argument, we decompose $\nz L^{-1/2}\az_k$ into a summation of
multiples of $(\ro,2)$-atoms with harmless constants,
which converges in $L^2(\rn)$.
For simplicity, we write it as $\nz L^{-1/2}\az_k=\sum_{j=1}^\fz b_{k,j}$,
where $b_{k,j}$ is a multiple of a $(\ro,2)$-atom supported in $B_{k,j}$
with a harmless constant.
Thus, by \eqref{7.11}, we obtain
\begin{eqnarray*}
\|\nz L^{-1/2}f\|_{H_{\oz}(\rn)}&&
=\inf\lf\{\lz>0:\, \sum_{k=1}^\fz\sum_{j=1}^\fz |B_{k,j}|
\oz\lf(\frac{|\lz_k|\|b_{k,j}\|_{L^2(\rn)}}{\lz|B_{k,j}|^{1/2}}\r)
\le1\r\}\\
&&\ls \inf\lf\{\lz>0:\, \sum_{k=1}^\fz|B_{k}|
\oz\lf(\frac{|\lz_k|}{\lz|B_{k}|\ro(|B_k|)}\r)
\le1\r\}\\
&&\sim \Lambda(\{\lz_k\az_k\}_{k})\ls \|f\|_{H_{\oz,L}(\rn)}.
\end{eqnarray*}
Then, by a standard density argument, we see that $\nz L^{-1/2}$
extends to a bounded linear operator from $H_{\oz,L}(\rn)$ to $H_{\oz}(\rn)$.
 This finishes the proof of Theorem \ref{t7.4}.
\end{proof}

\begin{rem}\rm \label{r7.4} Let $\oz$ satisfy Assumption (A) and
$p_\oz\in (\frac{n}{n+1},1]$. We claim that the Orlicz-Hardy
spaces $H_{\oz,L}(\rn)\subset H_{\oz}(\rn)$. In particular,
$H_L^p(\rn)\subset H^p(\rn)$ for all $p\in (\frac{n}{n+1},1]$.

Let $\ez\in (n(1/p_\oz-1/p_\oz^+),\fz)$, $M\in \cn$ and
$M>\frac n2 (\frac 1{p_\oz}-\frac{1}2)$.
For all $(\oz,\fz,M,\ez)$-molecules $\az$, we claim that
$\int_\rn \az(x)\,dx=0$. To show this, write
$$\az=\mathop\mathrm{div}(A\nz L^{-1}\az)=
r_B\lf\{\mathop\mathrm{div}(A[r_B\nz (I+r_B^2L)^{-1}(r_B^2L)^{-1}\az+r_B
\nz (I+r_B^2L)^{-1} \az])\r\},$$
where $\az$ is adapted to the ball $B\equiv B(x_B,r_B)$.

From the H\"older inequality, Lemma \ref{l2.2} and Definition \ref{d4.2},
it follows that for $j=0,1,\cdots,10$,
\begin{eqnarray*}
\|r_B\nz (I+r_B^2L)^{-1}(r_B^2L)^{-1}\az\|_{L^1(U_j(B))}&& \le
|U_j(B)|^{1/2}\|r_B\nz (I+r_B^2L)^{-1}(r_B^2L)^{-1}\az\|_{L^2(\rn)} \\
&& \ls |B|^{1/2}\|(r_B^2L)^{-1}\az\|_{L^{2}(\rn)}\ls [\ro(|B|)]^{-1}.
\end{eqnarray*}
For $j\ge 11$, let $W_j(B)\equiv (2^{j+3}B\setminus 2^{j-3}B)$
and $E_j(B)\equiv (W_j(B))^\com$.
By Lemma \ref{l2.2} and the H\"older inequality, we have
\begin{eqnarray*}
&&\| r_B\nz (I+r_B^2L)^{-1}(r_B^2L)^{-1}\az\|_{L^1(U_j(B))}\\
&&\hs\le |U_j(B)|^{1/2}\lf\|r_B\nz (I+r_B^2L)^{-1}
\lf[(\chi_{W_j(B)}+\chi_{E_j(B)})(r_B^2L)^{-1}\az\r]
\r\|_{L^{2}(U_j(B))}\nonumber\\
&&\hs\ls  |U_j(B)|^{1/2}\lf\{\|(r_B^2L)^{-1}\az\|_{L^2(W_j(B))}
+\exp\lf\{-\frac{\dist(U_j(B),E_j(B))}{cr_B}\r\}\|\az\|_{L^s(\rn)}\r\}\\
&&\hs\ls 2^{-j\ez} [\ro(|2^jB|)]^{-1}+2^{jn/2}\lf(\frac{r_B}{2^jr_B}\r)^{n/2+\ez}
[\ro(|B|)]^{-1}\ls 2^{-j\ez}[\ro(|B|)]^{-1}.
\end{eqnarray*}
The above two estimates imply that
$r_B\nz (I+r_B^2L)^{-1}(r_B^2L)^{-1}\az\in L^1(\rn)$.

Similarly, we have that $r_B\nz (I+r_B^2L)^{-1} \az\in L^1(\rn)$,
and hence, $\nz L^{-1}\az\in L^1(\rn)$.

Let $\{\vz_j\}_{j=0}^\fz$ be as in the proof of Theorem \ref{t7.4}.
Using the properties of $\{\vz_j\}_{j=0}^\fz$ and the facts
that $\az,\ \nz L^{-1}\az\in L^1(\rn)$ together with the divergence
theorem, we obtain
\begin{eqnarray*}\int_\rn \az(x)\,dx&&=\sum_{j=0}^\fz \int_{\rn}
\mathop\mathrm{div}(\vz_j A\nz L^{-1}\az)(x)\,dx\\
&&=\sum_{j=0}^\fz\int_{\pa (2B_j)}\langle\vz_j(x)\vec N_{2B_j}(x),
A\nz L^{-1}\az(x)\rangle\,d\sz_{2B_j}{x}=0,
\end{eqnarray*}
where $\vec N_{2B_j}$ denotes the outward unit norm vector to $2B_j$
and $\sz_{2B_j}$ the surface measure over $\pa (2B_j)$.

Then following the proof of Theorem \ref{t7.4}, we obtain that
 for all $f\in H_{\oz,L}(\rn)\cap L^2(\rn)$,
 $\|f\|_{H_{\oz}(\rn)}\ls\|f\|_{H_{\oz,L}(\rn)}$.
By a density argument, we obtain that $H_{\oz,L}(\rn)\subset H_{\oz}(\rn)$,
which completes the proof of the above claim.
\end{rem}

\medskip

{\bf Acknowledgements.} Dachun Yang would like to thank Professor
Steve Hofmann, Professor Pascal Auscher
and Professor Lixin Yan very much for some helpful discussions
on the subject of this paper. The authors would also like
to thank the referee very much for his many valuable remarks
which made this article more readable. Dachun Yang is supported by the National
Natural Science Foundation (Grant No. 10871025) of China.

\end{document}